\documentclass[twoside]{article}
\usepackage{lmodern}
\renewcommand{\thefootnote}{\fnsymbol{footnote}}
\usepackage[accepted]{aistats2022}
\usepackage[utf8]{inputenc} 
\usepackage[T1]{fontenc}    
\usepackage{hyperref}       
\usepackage{xurl}            
\usepackage{booktabs}       
\usepackage{amsmath}
\usepackage{amsthm}
\usepackage{amsfonts}       
\usepackage{nicefrac}       
\usepackage{microtype}      
\usepackage{xcolor}         
\usepackage{enumerate}
\usepackage{natbib}
\usepackage{graphicx}
\usepackage{subcaption}
\usepackage{multirow}
\usepackage{multicol}
\usepackage{comment}

\usepackage{algorithm}
\usepackage{algorithmic}
\usepackage{pgfplots}
\usepackage{csvsimple}

\DeclareMathOperator{\dom}{dom}

\DeclareMathOperator{\interior}{int}

\DeclareMathOperator{\support}{supp}
\DeclareMathOperator*{\argmax}{arg\,max}

\DeclareMathOperator*{\OT}{OT}
\DeclareMathOperator*{\sumnorm}{sum}

\newcommand{\R}{\mathbb{R}}
\newcommand{\N}{\mathbb{N}}
\newcommand{\M}{\mathcal{M}}

\newcommand{\mb}{\mathbb}
\newcommand{\mc}{\mathcal}

\DeclareMathOperator{\Lip}{Lip}
\DeclareMathOperator{\diam}{diam}

\newtheorem{proposition}{Proposition}
\newtheorem{lemma}[proposition]{Lemma}

\newtheorem*{corollary*}{Corollary}
\newtheorem{theorem}[proposition]{Theorem}
\newtheorem*{theorem*}{Theorem}
\newtheorem{example}[proposition]{Example}
\theoremstyle{definition}
\newtheorem{definition}[proposition]{Definition}
\theoremstyle{remark}
\newtheorem{remark}[proposition]{Remark}

\begin{document}

%

%

\twocolumn[

\aistatstitle{Optimal transport with $f$-divergence regularization and generalized Sinkhorn algorithm}

\aistatsauthor{ D\'avid Terj\'ek\footnotemark[1] \And Diego Gonz\'alez-S\'anchez\footnotemark[1] }
\runningauthor{D\'avid Terj\'ek, Diego Gonz\'alez-S\'anchez}

\aistatsaddress{ Alfr\'ed R\'enyi Institute of Mathematics \And  Alfr\'ed R\'enyi Institute of Mathematics } ]

\begin{abstract}
Entropic regularization provides a generalization of the original optimal transport problem. It introduces a penalty term defined by the Kullback-Leibler divergence, making the problem more tractable via the celebrated Sinkhorn algorithm. Replacing the Kullback-Leibler divergence with a general $f$-divergence leads to a natural generalization. The case of divergences defined by superlinear functions was recently studied by Di Marino and Gerolin. Using convex analysis, we extend the theory developed so far to include all $f$-divergences defined by functions of Legendre type, and prove that under some mild conditions, strong duality holds, optimums in both the primal and dual problems are attained, the generalization of the $c$-transform is well-defined, and we give sufficient conditions for the generalized Sinkhorn algorithm to converge to an optimal solution. We propose a practical algorithm for computing an approximate solution of the optimal transport problem with $f$-divergence regularization via the generalized Sinkhorn algorithm. Finally, we present experimental results on synthetic 2-dimensional data, demonstrating the effects of using different $f$-divergences for regularization, which influences convergence speed, numerical stability and sparsity of the optimal coupling.
\end{abstract}

\footnotetext[1]{Both authors should be equally credited for this work.}
\renewcommand{\thefootnote}{\arabic{footnote}}

\section{INTRODUCTION}

Since its inception in the 18th century with the work of Gaspard Monge, the theory of optimal transport \citep{Villani2008} has found its applications in many areas such as physics, economics and statistics. Among other developments, the optimal transport problem led L. V. Kantorovich to develop his duality theory \citep{Kantorovich1940} and to pioneer the field of linear programming \citep{Kantorovich1939} for practical solutions during World War II. This theory has been applied successfully in computer vision in tasks such as image retrieval \citep{Rubneretal1997}. However, computing the optimal transport involved solving a linear program which was computationally too costly to apply it to machine learning. Cuturi showed that slightly modifying the original optimal transport problem by introducing a regularization term one can compute the (regularized) optimal transport cost using the Sinkhorn algorithm \citep{Sinkhorn1967} in significantly less time \citep{Cuturi2013}. In recent years, this generalization of the optimal transport problem called entropy-regularized optimal transport \citep{Peyreetal2019} has become a popular tool in the machine learning community \citep{Feydyetal2018,Lorenz2020,Dimarinoetal2020}.

\subsection{Our contributions}

Let $\mu$ and $\nu$ be Borel probability measures defined on compact metric spaces $X$ and $Y$ respectively. Let $D_\phi$ be an $f$-divergence defined by a convex and lower semicontinuous function $\phi: \R \to \overline{\R}$ such that $\phi(1)=0$. Let $c : X \times Y \to \R$ be a Lipschitz continuous cost function and $\epsilon>0$ a constant. We are interested in the optimal transport problem with $f$-divergence regularization (or Primal Problem) defined as
\begin{equation}\label{eq:primal}
\OT_{\epsilon}(\mu,\nu) = \inf_{\pi \in \Pi(\mu,\nu)}\left\{ \int c d\pi + \epsilon D_\phi(\pi \Vert \mu \otimes \nu) \right\}
\end{equation}
where $\Pi(\mu,\nu)$ is the set of Borel couplings of $\mu$ and $\nu$. The corresponding Dual Problem is then
\begin{multline}\label{eq:dual}
\sup_{f \oplus g \leq c + \epsilon \phi'(\infty)}\left\{ \int f \oplus g d\mu \otimes \nu\right. \\\left.- \epsilon \int \phi_+^* \circ \frac{1}{\epsilon}(f \oplus g - c) d\mu \otimes \nu \right\},
\end{multline}
where the potentials $f$ and $g$ are assumed to be real-valued Lipschitz functions on $X$ and $Y$, respectively, $\phi'(\infty)=\lim_{x \to \infty}{\frac{\phi(x)}{x}}$, and $\phi_+^*$ is the convex conjugate of $\phi_+ = \phi + \iota_{\R_+}$.

In this paper we prove that if $\phi$ is of Legendre type then the Primal and Dual Problems have equal optimums. Furthermore, there exists optimal couplings for \eqref{eq:primal} and optimal potentials for \eqref{eq:dual}. This generalizes the work of \cite{Dimarinoetal2020}, which develops the theory for superlinear $\phi$, i.e. for $\phi'(\infty)=\infty$. We also prove that the singular part (which is always 0 for superlinear $\phi$) of an optimal coupling is supported on a $c$-cyclically monotone set \citep[Definition~5.1]{Villani2008} (see Theorem~\ref{thm:f-kantorovich-duality}).

In order to prove these results, we also generalize the $(c,\epsilon,\phi)$-transform \citep[Definition~3.1]{Dimarinoetal2020} so that it also works for non-superlinear $\phi$, i.e. for $\phi'(\infty) < \infty$. This turned out to be a non-trivial task. Moreover, an interesting phenomenon occurs in the case of non-superlinear divergences as the corresponding $(c,\epsilon,\phi)$-transform sometimes collapses to (almost) the $c$-transform \citep[Definition~5.2]{Villani2008} (see Proposition~\ref{prop:properties-c-eps-phi-trans} in Appendix~\ref{appendix_proofs}). This shows a more explicit connection between the classical theory of optimal transport and the regularized versions.

We show that a generalized version of the Sinkhorn algorithm (also denoted IPFP sequences \citep{Dimarinoetal2020}) converge to an optimal solution even in the non-superlinear case under mild assumptions (see Definition~\ref{def:good_triple} and Theorem~\ref{thm:conv_sinkhorn}). Finally, we propose a practical algorithm for computing an approximate solution of the optimal transport problem with $f$-divergence regularization using the generalized Sinkhorn algorithm.

We demonstrate the method on synthetic 2-dimensional point clouds. Our results indicate that for practical implementations the $\chi^2$ divergence can compete with the Kullback-Leibler divergence of classical entropy-regularized OT. The corresponding algorithm is slightly slower but gives sparse optimal couplings. Thus, it could be useful in any task where we can make use of this sparsity, see Appendix~\ref{app:grad-through}.

\subsection{Related work}\label{sec:rel-work}

Since the breakthrough of \cite{Cuturi2013}, the area of entropy-regularized optimal transport has grown quickly \citep{Peyreetal2019, Santambrogio2015}. Some of them have focused on studying the case of the Kullback-Leibler divergence and $\Gamma$-convergence to the unregularized problem \citep{Clasonetal2019}. Others have focused on generalizing the regularization to tackle linear programming problems \citep{Benamouetal2015}. There are results on $\Gamma$-convergence for the squared Euclidean cost and a proof of convergence of the discrete entropic smoothing of the Wasserstein gradient flow \citep{Carlieretal2017}. We can also find a theoretical proof together with practical experiments of the usefulness of Sinkhorn divergences, which remove the bias introduced to the optimal coupling by the regularization term \citep{Feydyetal2018}. Other types of generalizations have also been proposed \citep{Robertsetal2017}.

But the work that motivated the most our results (and which is clearly closest to this paper) is \cite{Dimarinoetal2020}. In this paper we find general results on strong duality and convergence of Sinkhorn iterations in the superlinear case ($\phi'(\infty)=\infty$). Indeed, our initial motivation was to understand the difficulties that arise in the non-superlinear case as most of the popular $f$-divergences used nowadays are non-superlinear \citep[Table~1]{Agrawaletal2020}, whereas in many places in the literature this assumption seems necessary (see \cite[Assumption~3.1]{Carlieretal2017} and \cite[Section~4]{Lorenz2020}). Thus, we decided to follow the same structure as \cite{Dimarinoetal2020} in the theoretical section of our paper, generalizing the proofs and concepts present in their work. In addition, we wanted to give rigorous proofs in the context of Lipschitz functions, that, as we explain in the paper, model better the case of neural networks. 

To conclude this section, we would like to highlight \cite{Desseinetal2018} where we find results on regularized optimal transport in finite spaces with Bregman divergences, which intersect with the set of $f$-divergences only at the Kullback-Leibler divergence. And \cite{Muzellecetal2017} where the case of Tsallis entropies (which are a subset of $f$-divergences) for the discrete case is covered. Other works focusing on finite spaces are \cite{Genevayetal2016, Altschuler2017, Blondel2018, Luiseetal2018, Luiseetal2019}. Closer to our work are also \cite{Ferradansetal2014, Rakotomamonjy2015, CuturiPeyre2016, Lorenzetal2019, Dimarinoetal2019, Kuroseetal2021, EcksteinNutz2021, Linetal2019}.

\section{BACKGROUND} \label{section_background}

\subsection{Notation}
We denote the extended reals by $\overline{\R}=\R\cup\{\pm\infty\}$, the nonnegative reals by $\R_+$, and the extended nonnegative reals by $\overline{\R}_+=\R_+ \cup {\infty}$. The indicator of a set $A$ is denoted by $\iota_A$ with $\iota_A(x)=0$ if $x \in A$ and $\iota_A(x)=\infty$ otherwise. We denote by $\interior A$ the interior of a set $A$ inside a topological space. Absolute continuity and singularity of measures will be denoted by $\ll$ and $\perp$ respectively. The Radon-Nikodym derivative of a measure $\mu$ with respect to a nonnegative measure $\nu$ such that $\mu \ll \nu$ is denoted by $\frac{d\mu}{d\nu}$ and the support of a measure $\mu$ by $\support(\mu)$. The product of measures $\mu,\nu$ is denoted by $\mu \otimes \nu$ and the set of measures having $\mu$ and $\nu$ as marginals by $\Pi(\mu,\nu)$. The set of probability measures on a measurable space $X$ is denoted by $P(X)$. For functions $f : X \to \R$ and $g : Y \to \R$, the tensor sum $f \oplus g : X \times Y \to \R$ is defined as $f \oplus g(x,y) = f(x)+g(y)$. Given a convex function $\phi : \R \to \overline{\R}$, its effective domain $\dom \phi \subset \R$ is defined as $\dom \phi = \{ s \in \R : \phi(s) < \infty \}$ and the convex conjugate $\phi^* : \R \to \overline{\R}$ as $\phi^*(t) = \sup_{s \in \R}\{ st - \phi(s) \}$. Such a function $\phi$ is proper if $\dom \phi \neq \emptyset$ and $\phi > -\infty$.

\subsection{$f$-divergences}
Given a proper, convex and lower semicontinuous function\footnote{Originally, $f$ is used in place of $\phi$ (hence the name), but we reserve the symbol $f$ for other functions.} $\phi : \R \to \overline{\R}$, a measure $\mu$ and a nonnegative measure $\nu$ on a measurable space $X$, the $f$-divergence of $\mu$ from $\nu$ is defined \citep{Csiszar1963,Alietal1966, Csiszar1967,Csiszaretal1999,Borweinetal1993, Agrawaletal2020} as
\begin{multline*}
    D_\phi(\mu \Vert \nu) = \int \phi \circ \frac{d\mu_c}{d\nu} d\nu\\ + \phi'(\infty)\mu_s^+(X) - \phi'(-\infty)\mu_s^-(X).
\end{multline*}
Here, $\mu_c \ll \nu, \mu_s \perp \nu$ are the absolutely continuous and singular parts of the Lebesgue decomposition of $\mu$ with respect to $\nu$ and $\mu_s^+, \mu_s^- \geq 0$ is the Jordan decomposition of the singular part. By definition $\phi'(\pm\infty) = \lim_{x \to \pm\infty}{\frac{\phi(x)}{x}} \in \overline{\R}$. Restricting to nonnegative measures can be done by using $\phi_+ = \phi + \iota_{\R_+}$ in place of $\phi$, inducing $D_{\phi_+}(\mu \Vert \nu) = D_{\phi}(\mu \Vert \nu)$ if $\mu \geq 0$ and $\infty$ otherwise.

A subset of $f$-divergences including the Kullback-Leibler, reverse Kullback-Leibler, $\chi^2$, reverse $\chi^2$, squared Hellinger, Jensen-Shannon, Jeffreys and triangular discrimination divergences, but excluding the total variation, consists of those defined by functions $\phi$ of Legendre type. A proper, convex and lower semicontinuous function $\phi : \R \to \overline{\R}$ is said to be of Legendre type \citep[Definition~2.5]{Borweinetal1993} if it is strictly convex on $\dom \phi$ and differentiable on $\interior \dom \phi$ with $\lim_{s \to \inf \dom \phi} \phi'(s) = -\infty$ if $\inf \dom \phi > -\infty$ and $\lim_{s \to \sup \dom \phi} \phi'(s) = \infty$ if $\sup \dom \phi < \infty$.

\subsection{Entropy-regularized optimal transport}

Let $\mu \in P(X)$ and $\nu \in P(Y)$ be probability measures defined on spaces $X$ and $Y$ and let $D_\phi$ be an $f$-divergence. The generalized entropy regularized optimal transport problem with cost function $c : X \times Y \to \overline{\R}$ and regularization coefficient $\epsilon \in \R_+$ is defined in \eqref{eq:primal}, and the corresponding dual problem\footnote{The constraint $f \oplus g \leq c + \epsilon \phi'(\infty)$ is absent if $\phi$ is superlinear.} is defined in \eqref{eq:dual}
\citep{Dimarinoetal2020}. Research in this area deals with the problem of finding suitable conditions under which strong duality holds, i.e. \eqref{eq:dual} equals \eqref{eq:primal}. In some cases of interest, there are known sufficient conditions ensuring that the infimum and the supremum are achieved by optimal primal and dual variables, and characterizations of such optimal variables have been developed as well.

The case $\epsilon=0$ with any $\phi$ reduces to the original, unregularized optimal transport problem, the duality theory of which is named after its most prominent contributor L. V. Kantorovich  \citep[Theorem~5.10]{Villani2008}. In this case, one has that there exists a closed, $c$-cyclically monotone set $C \subset X \times Y$ such that any optimal primal variable $\pi$ is supported on $C$. A set $C \subset X \times Y$ is called $c$-cyclically monotone \citep[Definition~5.1]{Villani2008} if for any subset $\{ (x_1,y_1),\dots,(x_n,y_n) \} \subset C$ for $n \in \N$, one has $\sum_{i = 1}^n c(x_i,y_i) \leq \sum_{i = 1}^{n-1} c(x_i,y_{i+1}) + c(x_n,y_1)$. This means that an optimal $\pi$ only assigns mass to pairs $(x_1,y_1),(x_2,y_2)$ such that one can not get lower transport cost by rerouting $\pi$ to assign mass to $(x_1,y_2),(x_2,y_1)$ instead.

The case $\epsilon > 0$ with $\phi = x \log(x) - x + 1$, corresponding to the Kullback-Leibler divergence, became a popular tool in machine learning due to its better computational performance over the unregularized case. Cuturi proved that the Sinkhorn algorithm can be used in this case to obtain the optimal variables in a significantly smaller timeframe compared to the unregularized case \citep{Cuturi2013}. The price of efficiency is the optimal coupling being biased, an issue that has been investigated and remedied \citep{Feydyetal2018}. For more references on the state of the art see Section~\ref{sec:rel-work}.

\section{OPTIMAL TRANSPORT WITH $f$-DIVERGENCE REGULARIZATION} \label{section_main}

\subsection{$(c,\epsilon,\phi)$-transform and $f$-Kantorovich duality}

In this paper we study the problem of regularized optimal transport under the assumptions that the underlying spaces $X$ and $Y$ are compact metric spaces, and the cost function $c:X\times Y\to \R$ and the potentials $f:X\to\R$, $g:Y\to\R$ are Lipschitz. The reason we have chosen this family of functions is that for most applications the costs involved satisfy this hypothesis. Also, for deep learning applications, any function represented by a neural network is a Lipschitz function, and if one aims to implement the potentials by neural networks such as in a GAN setting, it makes sense to develop the theory of regularized optimal transport on Lipschitz functions.

\begin{remark}
The results presented in this paper can be also applied for Polish spaces $X$ and $Y$ as long as the measures $\mu\in P(X)$ and $\nu\in P(Y)$ are compactly supported. Furthermore, we can always assume that both $\mu$ and $\nu$ are of full support. To see this, note that if $\pi\in \Pi(\mu,\nu)$ then $\support(\pi) \subset \support(\mu)\times \support(\nu)$. Thus, for many problems (such as the ones we deal with in this paper), given compactly supported measures $\mu$ and $\nu$ on Polish spaces $X$ and $Y$ respectively, we can assume that $\support(\mu)=X$ and $\support(\nu)=Y$. If this is not the case, we can always restrict ourselves to the support, apply all the results that we are going to present to $\support(\mu)$ and $\support(\nu)$ and then go back to the original spaces. Given a measure defined in $\support(\mu)\times\support(\nu)$ it is trivial how to define a measure on $X\times Y$, and any function $f\in \Lip(\support(\mu))$ or $g\in\Lip(\support(\nu))$ can be extended to a Lipschitz function on $X$ or $Y$ with the same Lipschitz norm, respectively \citep[Theorem~4.1.1]{Cobzasetal2019}.
\end{remark}

Our first main result concerns the generalization of the $c$-transform. Recall the classical problem of optimal transport $\OT(\mu,\nu)=\inf_{\pi\in \Pi(\mu,\nu)}\{\int c\;d\pi\}$ for a cost function $c:X\times Y\to \mb{R}$ \citep{Villani2008}. For the sake of simplicity we assume that $c$ is continuous and $X$ and $Y$ are compact metric spaces. It is trivial that for any pair of continuous functions $f:X\to \R$ and $g:Y\to \R$, if $f\oplus g\le c$ then $\int f\;d\mu+\int g\;d\nu \le \int c\;d\pi$ for any $\pi\in \Pi(\mu,\nu)$. A classical result of Kantorovich shows that in fact the supremum of $\int f\;d\mu+\int g\;d\nu$ over all functions $f\oplus g\le c$ equals the infimum of $ \int c\;d\pi$ over all couplings \citep{Villani2008}.

Let us now think about this problem in the following way, if $f\oplus g\le c$ for any pair of functions $(f,g)$ we can ``improve'' the value of $\int f\;d\mu+\int g\;d\nu$ by replacing $g$ with $\inf_{x\in X}\{c(x,y)-f(x)\}:=f^c(y)$. The latter function is called the $c$-transform of $f$ \citep{Villani2008}. Clearly $f\oplus f^c\le c$. Similarly, we could replace $f$ with $(f^c)^c$, defined analogously. The values that we will obtain in the dual problem will never decrease, i.e. $\int f\;d\mu+\int g\;d\nu\le \int f\;d\mu+\int f^c\;d\nu\le \int (f^c)^c\;d\mu+\int f^c\;d\nu\le\cdots$. Unfortunately, after repeating this process we will see that we get stuck \citep[Proposition~5.8]{Villani2008} and in general we will not reach the value $\OT(\mu,\nu)$. The great advantage of regularized optimal transport is that if we replace $\OT(\mu,\nu)$ by $\OT_{\epsilon}(\mu,\nu)$ defined in \eqref{eq:primal}, at the cost of introducing a bias, the analogue of the previous argument will in fact converge (under certain conditions) to $\OT_{\epsilon}(\mu,\nu)$.

The analogue of the $c$-transform for the problem $\OT_{\epsilon}(\mu,\nu)$ was introduced by \cite{Dimarinoetal2020} for superlinear divergences ($\phi'(\infty)=\infty$) \citep[Definition~3.1]{Dimarinoetal2020}. We generalize that definition to the case of any $f$-divergence defined by $\phi$ of Legendre type.

\begin{definition}[$(c,\epsilon,\phi)$-transform] Let $c \in \Lip(X \times Y)$, $\phi : \R \to \overline{\R}$ a proper, convex and lower semicontinuous function of Legendre type with $\phi(1)=0$, $\epsilon > 0$, $\mu \in P(X)$ and $\nu \in P(Y)$ with full supports. We define the $(c,\epsilon,\phi)$-transform $f^{(c,\epsilon,\phi)} \in \Lip(Y)$ of $f \in \Lip(X)$ as follows:
\begin{multline*}
f^{(c,\epsilon,\phi)}(y):=\argmax_{\gamma\le f^c(y)+\epsilon\phi'(\infty)} \left\{ \frac{1}{\epsilon}\gamma\right.\\\left.-\int \phi_+^*\left(\frac{1}{\epsilon}(f(x)+\gamma-c(x,y))\right)\;d\mu(x) \right\}.
\end{multline*}
\end{definition}

See Proposition~\ref{prop:properties-c-eps-phi-trans} in Appendix~\ref{appendix_proofs} for properties of the $(c,\epsilon,\phi)$-transform. Let us now check why this definition is the natural generalization of the $c$-transform. Let $(f,g)$ be a pair of potentials such that $f\oplus g\le c+\epsilon\phi'(\infty)$. It follows from the convex conjugate of $D_{\phi_+}(\cdot\|\mu\otimes\nu)$ \citep{Borweinetal1993,Agrawaletal2020} and the Young-Fenchel inequality\footnote{As any coupling $\pi$ is by definition positive, we can replace $\phi$ by $\phi_+=\phi+\iota_{\R\ge 0}$ and $D_{\phi}(\pi\|\mu\otimes\nu)=D_{\phi_+}(\pi\|\mu\otimes\nu)$.} that $\int fd\mu+\int gd\nu-\epsilon \int \phi_+^* \circ \frac{1}{\epsilon}(f\oplus g-c) d\mu\otimes\nu \le \int cd\pi+\epsilon D_{\phi}(\pi\|\mu\otimes\nu)$. Looking at the left hand side of the inequality, notice that if we try to adjust the value of $g(y)$ \textbf{pointwise} at any fixed point $y$ in such a way that $g(y)-\epsilon \int \phi_+^*(\frac{1}{\epsilon}(f(x)+g(y)-c(x,y))d\mu(x)$ is maximized we obtain precisely the $(c,\epsilon,\phi)$-transform of $f$. Hence, we have an analogous inequality as before, $\int fd\mu+\int gd\nu-\epsilon \int \phi_+^* \circ \frac{1}{\epsilon}(f\oplus g-c) d\mu\otimes\nu \le \int fd\mu+\int f^{(c,\epsilon,\phi)}d\nu-\epsilon \int \phi_+^* \circ \frac{1}{\epsilon}(f\oplus f^{(c,\epsilon,\phi)}-c) d\mu\otimes\nu$. 

The similarities do not end here, we encourage the reader to compare Proposition~\ref{prop:properties-c-eps-phi-trans} with Proposition~\ref{prop:properties-c-transform} where we have stated many properties of the $(c,\epsilon,\phi)$- and $c$-transforms, respectively. For now, let us mention how we can compute the value of the $(c,\epsilon,\phi)$-transform. The following result is $(i)$ of Proposition~\ref{prop:properties-c-eps-phi-trans}:

\emph{$f^{(c,\epsilon,\phi)}(y)$ is well-defined for all $y \in Y$ implicitly by $\int_X {\phi_+^*}' \circ \frac{1}{\epsilon}(f + f^{(c,\epsilon,\phi)}(y) - c(\cdot,y)) d\mu = 1$ if there exists such a number $f^{(c,\epsilon,\phi)}(y) \in \R$ or explicitly as $f^{(c,\epsilon,\phi)}(y)=\min_{x \in X}\{\epsilon \phi'(\infty) + c(x,y)-f(x)\}=f^c(y)+\epsilon\phi'(\infty)$ otherwise.}

This shows precisely why if $\phi'(\infty)=\infty$ this definition reduces to solving the implicit equation
\begin{equation}\label{eq:implicit_equation}
\int_X {\phi_+^*}' \circ \frac{1}{\epsilon}(f + \gamma - c(\cdot,y)) d\mu = 1
\end{equation}
for $\gamma$. However, if this is not the case, there may be cases where the $(c,\epsilon,\phi)$-transform is just the $c$-transform plus $\epsilon\phi'(\infty)$. Indeed, this behaviour can happen as we can see in Example~\ref{example}. Analogously to the $c$-subdifferential \citep[Definition~5.2]{Villani2008}, the $(c,\epsilon,\phi)$-subdifferential of $f \in \Lip(X)$ defined as $\partial_{(c,\epsilon,\phi)} f = \{ (x,y) \in X \times Y : f(x)+f^{(c,\epsilon,\phi)}(y) = c(x,y)+\epsilon \phi'(\infty) \}$ is a closed, $c$-cyclically monotone set (see Proposition~\ref{prop:c-mon}). Analogous results hold for the $(c,\epsilon,\phi)$-transform, $g^{(c,\epsilon,\phi)} \in \Lip(X)$, of $g \in \Lip(Y)$.

We can now state one of the main results of this paper, generalizing the Kantorovich duality of optimal transport (see \citet[Section~3.2]{Dimarinoetal2020} and \citet[Theorem~5.10]{Villani2008}).

\begin{theorem}[$f$-Kantorovich duality]\label{thm:f-kantorovich-duality}
Let $\mu \in P(X)$ and $\nu \in P(Y)$ be probability measures of full support on compact metric spaces $(X,d_X)$ and $(Y,d_Y)$. Let $c \in \Lip(X \times Y)$, $0 < \epsilon \in \R$ be a regularization coefficient  and $\phi : \R \to \overline{\R}$ a proper, convex and lower semicontinuous function  of Legendre type. Then one has
\begin{align*}
&\min_{\pi \in \Pi(\mu,\nu)}\left\{ \int c d\pi + \epsilon D_{\phi}(\pi \Vert \mu \otimes \nu) \right\} \\
&= \max_{\substack{f\in\Lip(X), g\in \Lip(Y) \\ f \oplus g \leq c + \epsilon \phi'(\infty)}}\left\{ \int f \oplus g d\mu \otimes \nu \right.\\ & \left.- \epsilon \int \phi_+^* \circ \frac{1}{\epsilon}(f \oplus g - c) d\mu \otimes \nu \right\} \\
&= \max_{f \in \Lip(X)}\left\{ \int f \oplus f^{(c,\epsilon,\phi)} d\mu \otimes \nu \right.\\ & \left.- \epsilon \int \phi_+^* \circ \frac{1}{\epsilon}(f \oplus f^{(c,\epsilon,\phi)} - c) d\mu \otimes \nu \right\} \\
&= \max_{g \in \Lip(Y)}\left\{ \int g^{(c,\epsilon,\phi)} \oplus g d\mu \otimes \nu \right.\\ & \left.- \epsilon \int \phi_+^* \circ \frac{1}{\epsilon}(g^{(c,\epsilon,\phi)} \oplus g - c) d\mu \otimes \nu \right\},
\end{align*}
i.e., strong duality holds and optimums in both the Primal and Dual Problems are attained. The absolutely continuous part (with respect to $\mu \otimes \nu$) $\pi_c$ of any optimal coupling $\pi$ is unique with its density given by
\begin{equation}\label{eq:density}
\frac{d\pi_c}{d\mu \otimes \nu} = {\phi_+^*}' \circ \frac{1}{\epsilon}(f \oplus g - c),
\end{equation}
where $(f,g)\in \Lip(X)\times\Lip(Y)$ are any pair of optimal potentials. Optimal potentials $(f,g)$ are such that $f \oplus g$ is unique almost everywhere with respect to $\pi_c$, and $f^{(c,\epsilon,\phi)}=g$ and $g^{(c,\epsilon,\phi)}=f$ always hold. Moreover, there exists a closed, $c$-cyclically monotone set $C$, which can be taken to be the intersection of the $(c,\epsilon,\phi)$-subdifferentials $\partial_{(c,\epsilon,\phi)} f = \partial_{(c,\epsilon,\phi)} g = \{ (x,y) \in X \times Y : f(x)+g(y) = c(x,y)+\epsilon \phi'(\infty) \}$ of all optimal couplings $(f,g)$, such that the singular part (with respect to $\mu \otimes \nu$) $\pi_s$ of any optimal coupling $\pi$ is supported on $C$, i.e., $\support(\pi_s) \subset C$.
\end{theorem}

See Theorem~\ref{thm:strong-dual-primal} and Proposition~\ref{prop:properties-optimals} in Appendix~\ref{appendix_proofs} for the proof. As a sketch, the proof of this result consists of two main parts. The first one is proving that both the Primal \eqref{eq:primal} and Dual \eqref{eq:dual} problems have the same optimum. The proof of this fact follows from convex analytic tools \citep[Theorem~2.6.1(v)]{Zalinescu2002}. To prove attainment in the Primal problem we can use a standard functional analytic argument. To prove attainment in the Dual Problem we make use of the properties of the $(c,\epsilon,\phi)$-transform given by Proposition~\ref{prop:properties-c-eps-phi-trans}. Then, uniqueness properties of the optimal couplings and potentials follow from the characterization of the subdifferentials of $f$-divergences \citep[Theorem~2.10]{Borweinetal1993} and the Young-Fenchel inequality.

\subsection{Generalized Sinkhorn algorithm}

The goal of this section is to prove that under certain conditions, given any starting pair of potentials $(f,g)$ if we start replacing $g$ with $f^{(c,\epsilon,\phi)}$, then $f$ with $(f^{(c,\epsilon,\phi)})^{(c,\epsilon,\phi)}$ and so on, we are able to recover a pair of optimal potentials of the Dual Problem and an optimal coupling for the Primal Problem. This process is called the generalized Sinkhorn algorithm. A single Sinkhorn iteration is defined as follows. Note that this definition yields a generalization of IPFP sequences \citep[Section~4]{Dimarinoetal2020} but with a stabilizing factor that will be helpful both in theory to prove convergence and in practice to prevent overflow.

\begin{definition}[Sinkhorn operator] Let $X$ and $Y$ be compact metric spaces and $\mu\in P(X)$, $\nu\in P(Y)$ be Borel probability measures of full support. Let also $c \in \Lip(X \times Y)$, $0 < \epsilon \in \R$ be a regularization coefficient  and $\phi : \R \to \overline{\R}$ a proper, convex and lower semicontinuous function  of Legendre type. Fix any point $y_0\in Y$. Given a pair $(f,g)\in \Lip(X)\times \Lip(Y)$ we define the operator $\mc{F}^{(c,\epsilon,\phi)}:\Lip(X)\times \Lip(Y)\to \Lip(X)\times \Lip(Y)$ as\footnote{Technically this operator depends as well on $y_0$, but as this is fixed and arbitrary, we decided not to include it explicitly.}
\begin{multline*}
    \mc{F}^{(c,\epsilon,\phi)}(f,g):= \\ ((f^{(c,\epsilon,\phi)}-f^{(c,\epsilon,\phi)}(y_0))^{(c,\epsilon,\phi)},f^{(c,\epsilon,\phi)}-f^{(c,\epsilon,\phi)}(y_0)).
\end{multline*}
\end{definition}

The most important properties of this operator are that if $(f',g')=\mc{F}^{(c,\epsilon,\phi)}(f,g)$ then $\|f'\|_L,\|g'\|_L,\|f'\|_{\infty}$ and $\|g'\|_{\infty}$ are uniformly bounded in terms of the diameters of $X$ and $Y$, $\|c\|_L$, $\|c\|_{\infty}$ and $\epsilon$. Also, this operator is continuous in the product topology generated by $\|\cdot\|_{\infty}$ on $\Lip(X) \times \Lip(Y)$. See Proposition~\ref{prop:sink-cpct-op} for more details.

We have seen before that iterating the $c$-transform in the classical optimal transport problem usually does not converge to a pair of optimal potentials. However, we know that using the Kullback-Leibler divergence for regularization we get convergence of the Sinkhorn algorithm to optimal potentials \citep{Cuturi2013}. As we saw before, as soon as $\phi'(\infty)<\infty$ the $(c,\epsilon,\phi)$-transform can collapse to almost the usual $c$-transform, in which case convergence is not guaranteed. Therefore, we introduce a mild condition that ensures that even in this case, the $(c,\epsilon,\phi)$-transform never collapses to the $c$-transform plus $\epsilon\phi'(\infty)$. This condition on the other hand is general enough to be able to include many examples and different $f$-divergences, and ensures that the $(c,\epsilon,\phi)$-subdifferentials are always empty. This implies that any optimal coupling is absolutely continuous with respect to $\mu \otimes \nu$, so that the optimal coupling is actually unique, as in the case $\phi'(\infty)=\infty$.

\begin{definition}[Good triple]\label{def:good_triple} Let $X$ be a compact metric space and $\mu$ a Borel probability measure on $X$. Let $\phi$ be proper, convex and lower semicontinuous function of Legendre type and suppose that $\phi'(\infty)<\infty$. Let $C>0$ be a constant. We say that $(X,\mu,\phi)$ is a \emph{good triple} with respect to $C$ if for all $x_0\in X$ one has
\[
\lim_{\delta\downarrow 0}\int_X {\phi_+^*}'(\phi'(\infty)-Cd(x_0,x)-\delta)\;d\mu(x) > 1.
\]
\end{definition}

This condition can be trivially verified if $X$ is a discrete space and the measure $\mu$ has full support. But more generally it applies to other functions $\phi$ even in general compact metric spaces. For example, if $X=[0,1]$ with the usual Lebesgue measure and the Euclidean distance then it is easy to check by hand that if $\phi$ is the function defining either the Jensen-Shannon, the squared Hellinger or the reverse Kullback-Leibler divergence then $(X,\mu,\phi)$ is a good triple with respect to any fixed constant $C$.

With this definition we can now state the main result of this section. Note that this result generalizes \citet[Theorem~4.1]{Dimarinoetal2020} to some cases where the divergence is not superlinear and it is adapted to the context of Lipschitz functions.

\begin{theorem}[Convergence of generalized Sinkhorn algorithm]\label{thm:conv_sinkhorn} Let $X$ and $Y$ be compact metric spaces and $\mu\in P(X)$, $\nu\in P(Y)$ be Borel probability measures of full support. Let also $c \in \Lip(X \times Y)$, $0 < \epsilon \in \R$ be a regularization coefficient  and $\phi : \R \to \overline{\R}$ a proper, convex and lower semicontinuous function of Legendre type. Suppose that either $\phi'(\infty)=\infty$ or  $(X,\mu,\phi)$ and $(Y,\nu,\phi)$ are good triples with respect to $2\|c\|_L/\epsilon$. Take any pair $(f_0,g_0)\in \Lip(X)\times \Lip(Y)$ and define inductively $(f_n,g_n):=\mc{F}^{(c,\epsilon,\phi)}(f_{n-1},g_{n-1})$ for $n\ge 1$. Let us also define the dual functional for any pair of functions $(f,g)\in \Lip(X)\times \Lip(Y)$ as

\[
D_{\epsilon}(f,g):= \int f\oplus g  - \epsilon \phi_+^* \circ \left( \frac{1}{\epsilon}(f\oplus g-c)\right) d\mu\otimes\nu.
\]
Then one has $D_{\epsilon}(f_n,g_n)\to \OT_{\epsilon}(\mu,\nu)$ as $n\to \infty$, and $f_n\oplus g_n\to \tilde{f} \oplus \tilde{g}$ in $L^{\infty}(\pi)$ as well with $\pi$ being the unique optimal coupling and $(\tilde{f},\tilde{g})$ any pair of optimal potentials. Moreover, $\pi$ can be recovered as $\pi = {\phi_+^*}' \circ \frac{1}{\epsilon}(\tilde{f} \oplus \tilde{g} - c) \cdot \mu \otimes \nu$.
\end{theorem}

\section{EXPERIMENTS} \label{section_experiments}

\subsection{Practical implementation}
For measures with finite supports, if $\support(\mu)=\{x_1,\dots,x_k\} = X$, the potential reduces to a finite-dimensional vector $f \in \R^k$ as $f_i=f(x_i)$ (and similarly for $\nu$ and $g$). In this case, the equation \eqref{eq:implicit_equation} defining the values of $(c,\epsilon,\phi)$-transforms can always be solved \citep{Terjek2021} via Newton's method in parallel\footnote{Note that we solve for $-\frac{1}{\epsilon}\gamma$ for better stability.}, which is included here as Algorithm~\ref{algorithm_gamma}. For $\phi'(\infty)<\infty$, initial values are chosen to be just below the boundary value by some parameter $\delta > 0$. For $\phi'(\infty)=\infty$, initial values are chosen to be $\gamma_i = \log \langle e^{h_{\cdot,i}}, \xi \rangle$, which is exactly the closed-form solution of $\gamma_{\phi,\xi}(h)$ for $\phi$ corresponding to the Kullback-Leibler divergence. Theoretically, as we are minimizing a convex function any initial value will eventually converge using Newton's method. We tried several initializations and this one seemed to give the best performance and that is why we have used it. Since we are running $n$ parallel Newton's method iterations, we set the stopping criterion to be the mean of the squared Newton steps falling below a tolerance parameter $\tau$.

\begin{algorithm}[h]
   \caption{Calculate $\gamma_{\phi,\xi}(h)$}
   \label{algorithm_gamma}
\begin{algorithmic}
   \STATE {\bfseries Input:} 
   \STATE $h \in M_{m \times n}(\R)$, $\xi \in \R^m$, $\phi : \R \to \overline{\R}$, $0 < \delta, \tau \in \R$
   \STATE {\bfseries Output:}
   \STATE $\gamma_{\phi,\xi}(h) \in \R^n$
   \\\hrulefill
   \IF{$\phi'(\infty) < \infty$}
   \STATE $\gamma_i = \max(h_{\cdot,i}) - \phi'(\infty) + \delta$. 
   \ELSE
   \STATE $\gamma_i = \log \langle e^{h_{\cdot,i}}, \xi \rangle$
   \ENDIF
   \REPEAT
   \STATE $s_i = \frac{-\langle (\phi_+^*)'(h_{\cdot,i} - \gamma), \xi \rangle + 1}{\langle (\phi_+^*)''(h_{\cdot,i} - \gamma), \xi \rangle}$
   \STATE $\gamma = \gamma - s$
   \UNTIL{$\frac{1}{n} \sum_{i = 1}^n s_i^2 < \tau$}
\end{algorithmic}
\end{algorithm}

\begin{algorithm}[t]
   \caption{Generalized Sinkhorn algorithm for computing optimal potentials $f,g$ and optimal coupling $\pi$}
   \label{algorithm_sinkhorn}
\begin{algorithmic}
   \STATE {\bfseries Input:} 
   \STATE $\mu \in \R^k$, $\support(\mu)=\{x_1,\dots,x_k\} \subset X$, 
   \STATE $\nu \in \R^l$, $\support(\nu)=\{y_1,\dots,y_l\} \subset Y$, 
   \STATE $c : X \times Y \to \R$, $\phi : \R \to \overline{\R}$, $0 < \epsilon, \tau \in \R$
   \STATE {\bfseries Output:}
   \STATE $f \in \R^k$, $g \in \R^l$, $\pi \in M_{k \times l}(\R)$ 
   \\\hrulefill
   \STATE $C_{i,j} = c(x_i,y_i)$
   \STATE $f_i = 0$ 
   \REPEAT
   \STATE $f_{prev} = f$
   \STATE $g = -\epsilon \gamma_{\phi,\mu}\left(\frac{1}{\epsilon}(f \otimes \mathbf{1}^l - C)\right)$
   \STATE $g = g - g_1$
   \STATE $f = -\epsilon \gamma_{\phi,\nu}\left(\frac{1}{\epsilon}(g \otimes \mathbf{1}^k - C^*)\right)$
   \UNTIL{$\Vert f - f_{prev} \Vert_\infty < \tau$}
   \STATE $\pi_{i,j} = {\phi_+^*}'(\frac{1}{\epsilon}(f_i + g_j - C_{i,j})) \mu_i \nu_j$
\end{algorithmic}
\end{algorithm}

We propose a practical implementation of the generalized Sinkhorn algorithm in Algorithm~\ref{algorithm_sinkhorn}. In both algorithms, vectors are understood as row vectors, and statements containing indices $i$ and/or $j$ are to be executed for each index value in parallel. The vectors $\mathbf{1}^l$ and $\mathbf{1}^k$ represent column vectors of dimension $l$ and $k$ with all their coordinates equal to 1, and thus their tensor products with row vectors of dimension $k$ and $l$ give matrices of dimension $k\times l$ and $l \times k$, respectively. We denote the adjoint of $C \in M_{k \times l}(\R)$ by $C^* \in M_{l \times k}(\R)$. Since our convergence results are in terms of the infinity norm, we set the stopping criterion to be the infinity norm of the difference of the potantials falling below a given tolerance parameter $\tau$.

\subsection{Experimental setup}
To demonstrate the feasibility of the approach, we apply the algorithm to synthetic 2-dimensional data obtained from \url{https://github.com/jeanfeydy/global-divergences}, the official codebase of \cite{Feydyetal2018}\footnote{The data is used according to its terms of use, which can be found following the link above.}. The data consists of 4 pairs of densities on $\R^2$, nicknamed "crescents", "densities", "moons" and "slopes". The task with each pair is to compute the regularized optimal transport problem between measures obtained by sampling a set of points independently from each density. Using different $f$-divergences and $\epsilon$s influences many aspects of the task, which are detailed below. In all examples, the cost function is $c(x,y) = \frac{1}{2}\Vert x-y \Vert_2^2$, i.e. half of the squared Euclidean distance on the plane. We consider classical $f$-divergences defined by $\phi$ of Legendre type, specifically the Kullback-Leibler, reverse Kullback-Leibler, $\chi^2$ (or Neyman $\chi^2$), reverse $\chi^2$ (or Pearson $\chi^2$), squared Hellinger, Jensen-Shannon, Jeffreys and triangular discrimination (or Vincze-Le Cam) divergences. The corresponding functions needed for the algorithms (such as $\phi_+^*$ and its first and second derivatives) are collected in Appendix~\ref{appendix_phis}. 

The source code to reproduce the experimental results can be found at \url{https://github.com/renyi-ai/optimal-transport-with-f-divergence-regularization-and-generalized-sinkhorn-algorithm}. In order to make the experiments more robust, for each one of the four different densities, each $f$-divergence and a range of $\epsilon$s we run the experiments with four different point cloud sizes (500, 1000, 2000 and 5000) and five different random seeds (which determine the point clouds sampled from the densities). In Appendix~\ref{appendix_results} we include a detailed account of the data, the hyperparameters and the experimental results. All experiments were run on NVIDIA A100 40GB SXM GPUs.

\begin{figure*}[!t]
     \centering
     \begin{subfigure}[b]{0.49\textwidth}
         \centering
\begin{tikzpicture}
\begin{axis}[
	axis x line*=bottom,
	axis y line*=left,
	error bars/y dir=both,
	error bars/y explicit,
    xmode = log,
    x dir=reverse,
    xlabel = $\epsilon$,
    x label style={at={(current axis.left of origin)},anchor=east, above=0mm},
    ylabel = $\int c d\pi$,
    y label style={at={(axis description cs:0.175,1.05)},anchor=west,rotate=-90},
    yticklabel style={rotate=90},
]
\addplot[orange, thick] table [x=epsilon, y=cost_avg, y error=cost_std, col sep=comma] {csv/kl_crescents_avg_and_std_over_random_seed_and_size_and_dataset_filtered.csv};
\addplot[blue, thick] table [x=epsilon, y=cost_avg, y error=cost_std, col sep=comma] {csv/rkl_crescents_avg_and_std_over_random_seed_and_size_and_dataset_filtered.csv};
\addplot[red, thick] table [x=epsilon, y=cost_avg, y error=cost_std, col sep=comma] {csv/chi2_crescents_avg_and_std_over_random_seed_and_size_and_dataset_filtered.csv};
\addplot[cyan, thick] table [x=epsilon, y=cost_avg, y error=cost_std, col sep=comma] {csv/rchi2_crescents_avg_and_std_over_random_seed_and_size_and_dataset_filtered.csv};
\addplot[magenta, thick] table [x=epsilon, y=cost_avg, y error=cost_std, col sep=comma] {csv/hellinger2_lowtol_crescents_avg_and_std_over_random_seed_and_size_and_dataset_filtered.csv};
\addplot[green, thick] table [x=epsilon, y=cost_avg, y error=cost_std, col sep=comma] {csv/js_lowtol_crescents_avg_and_std_over_random_seed_and_size_and_dataset_filtered.csv};
\addplot[violet, thick] table [x=epsilon, y=cost_avg, y error=cost_std, col sep=comma] {csv/jeffreys_crescents_avg_and_std_over_random_seed_and_size_and_dataset_filtered.csv};
\addplot[teal, thick] table [x=epsilon, y=cost_avg, y error=cost_std, col sep=comma] {csv/triangular_crescents_avg_and_std_over_random_seed_and_size_and_dataset_filtered.csv};
\end{axis}
\end{tikzpicture}
         \caption{Cost of optimal coupling}
         \label{fig:cost}
     \end{subfigure}
     \hfill
     \begin{subfigure}[b]{0.49\textwidth}
         \centering
\begin{tikzpicture}
\begin{axis}[
	axis x line*=bottom,
	axis y line*=left,
	error bars/y dir=both,
	error bars/y explicit,
    xmode = log,
    x dir=reverse,
    xlabel = $\epsilon$,
    x label style={at={(current axis.left of origin)},anchor=east, above=0mm},
	ymode = log, 
    ylabel = time (s),
    y label style={at={(axis description cs:0.175,1.05)},anchor=west,rotate=-90},
    yticklabel style={rotate=90},
]
\addplot[orange, thick] table [x=epsilon, y=time_avg, y error=time_std, col sep=comma] {csv/kl_crescents_avg_and_std_over_random_seed_and_size_and_dataset_filtered.csv};
\addplot[blue, thick] table [x=epsilon, y=time_avg, y error=time_std, col sep=comma] {csv/rkl_crescents_avg_and_std_over_random_seed_and_size_and_dataset_filtered.csv};
\addplot[red, thick] table [x=epsilon, y=time_avg, y error=time_std, col sep=comma] {csv/chi2_crescents_avg_and_std_over_random_seed_and_size_and_dataset_filtered.csv};
\addplot[cyan, thick] table [x=epsilon, y=time_avg, y error=time_std, col sep=comma] {csv/rchi2_crescents_avg_and_std_over_random_seed_and_size_and_dataset_filtered.csv};
\addplot[magenta, thick] table [x=epsilon, y=time_avg, y error=time_std, col sep=comma] {csv/hellinger2_lowtol_crescents_avg_and_std_over_random_seed_and_size_and_dataset_filtered.csv};
\addplot[green, thick] table [x=epsilon, y=time_avg, y error=time_std, col sep=comma] {csv/js_lowtol_crescents_avg_and_std_over_random_seed_and_size_and_dataset_filtered.csv};
\addplot[violet, thick] table [x=epsilon, y=time_avg, y error=time_std, col sep=comma] {csv/jeffreys_crescents_avg_and_std_over_random_seed_and_size_and_dataset_filtered.csv};
\addplot[teal, thick] table [x=epsilon, y=time_avg, y error=time_std, col sep=comma] {csv/triangular_crescents_avg_and_std_over_random_seed_and_size_and_dataset_filtered.csv};
\end{axis}
\end{tikzpicture}
         \caption{Runtime in seconds}
         \label{fig:runtime}
     \end{subfigure}
     \hfill
     
     \begin{subfigure}[b]{0.49\textwidth}
         \centering
\begin{tikzpicture}
\begin{axis}[
	axis x line*=bottom,
	axis y line*=left,
	error bars/y dir=both,
	error bars/y explicit,
    xmode = log,
    x dir=reverse,
    xlabel = $\epsilon$,
    x label style={at={(current axis.left of origin)},anchor=east, above=0mm},
    ylabel = ratio,
    y label style={at={(axis description cs:0.175,1.05)},anchor=west,rotate=-90},
    yticklabel style={rotate=90},
]
\addplot[orange, thick] table [x=epsilon, y=pi_nonzeros_ratio_avg, y error=pi_nonzeros_ratio_std, col sep=comma] {csv/kl_crescents_avg_and_std_over_random_seed_and_size_and_dataset_filtered.csv};
\addplot[blue, thick] table [x=epsilon, y=pi_nonzeros_ratio_avg, y error=pi_nonzeros_ratio_std, col sep=comma] {csv/rkl_crescents_avg_and_std_over_random_seed_and_size_and_dataset_filtered.csv};
\addplot[red, thick] table [x=epsilon, y=pi_nonzeros_ratio_avg, y error=pi_nonzeros_ratio_std, col sep=comma] {csv/chi2_crescents_avg_and_std_over_random_seed_and_size_and_dataset_filtered.csv};
\addplot[cyan, thick] table [x=epsilon, y=pi_nonzeros_ratio_avg, y error=pi_nonzeros_ratio_std, col sep=comma] {csv/rchi2_crescents_avg_and_std_over_random_seed_and_size_and_dataset_filtered.csv};
\addplot[magenta, thick] table [x=epsilon, y=pi_nonzeros_ratio_avg, y error=pi_nonzeros_ratio_std, col sep=comma] {csv/hellinger2_lowtol_crescents_avg_and_std_over_random_seed_and_size_and_dataset_filtered.csv};
\addplot[green, thick] table [x=epsilon, y=pi_nonzeros_ratio_avg, y error=pi_nonzeros_ratio_std, col sep=comma] {csv/js_lowtol_crescents_avg_and_std_over_random_seed_and_size_and_dataset_filtered.csv};
\addplot[violet, thick] table [x=epsilon, y=pi_nonzeros_ratio_avg, y error=pi_nonzeros_ratio_std, col sep=comma] {csv/jeffreys_crescents_avg_and_std_over_random_seed_and_size_and_dataset_filtered.csv};
\addplot[teal, thick] table [x=epsilon, y=pi_nonzeros_ratio_avg, y error=pi_nonzeros_ratio_std, col sep=comma] {csv/triangular_crescents_avg_and_std_over_random_seed_and_size_and_dataset_filtered.csv};
\end{axis}
\end{tikzpicture}
         \caption{Sparsity of optimal coupling}
         \label{fig:ratio}
     \end{subfigure}
     \hfill
     \begin{subfigure}[b]{0.49\textwidth}
         \centering
\begin{tikzpicture}
\begin{axis}[
	axis x line*=bottom,
	axis y line*=left,
	error bars/y dir=both,
	error bars/y explicit,
    xmode = log,
    x dir=reverse,
    xlabel = $\epsilon$,
    x label style={at={(current axis.left of origin)},anchor=east, above=0mm},
	ymode = log, 
    ylabel = marginal error,
    y label style={at={(axis description cs:0.175,1.05)},anchor=west,rotate=-90},
    yticklabel style={rotate=90},
]
\addplot[violet, thick] table [x=epsilon, y=marginal_error_1_avg, y error=marginal_error_1_std, col sep=comma] {csv/jeffreys_crescents_avg_and_std_over_random_seed_and_size_and_dataset_filtered.csv};
\addplot[orange, thick] table [x=epsilon, y=marginal_error_1_avg, y error=marginal_error_1_std, col sep=comma] {csv/kl_crescents_avg_and_std_over_random_seed_and_size_and_dataset_filtered.csv};
\addplot[blue, thick] table [x=epsilon, y=marginal_error_1_avg, y error=marginal_error_1_std, col sep=comma] {csv/rkl_crescents_avg_and_std_over_random_seed_and_size_and_dataset_filtered.csv};
\addplot[red, thick] table [x=epsilon, y=marginal_error_1_avg, y error=marginal_error_1_std, col sep=comma] {csv/chi2_crescents_avg_and_std_over_random_seed_and_size_and_dataset_filtered.csv};
\addplot[cyan, thick] table [x=epsilon, y=marginal_error_1_avg, y error=marginal_error_1_std, col sep=comma] {csv/rchi2_crescents_avg_and_std_over_random_seed_and_size_and_dataset_filtered.csv};
\addplot[magenta, thick] table [x=epsilon, y=marginal_error_1_avg, y error=marginal_error_1_std, col sep=comma] {csv/hellinger2_lowtol_crescents_avg_and_std_over_random_seed_and_size_and_dataset_filtered.csv};
\addplot[green, thick] table [x=epsilon, y=marginal_error_1_avg, y error=marginal_error_1_std, col sep=comma] {csv/js_lowtol_crescents_avg_and_std_over_random_seed_and_size_and_dataset_filtered.csv};
\addplot[teal, thick] table [x=epsilon, y=marginal_error_1_avg, y error=marginal_error_1_std, col sep=comma] {csv/triangular_crescents_avg_and_std_over_random_seed_and_size_and_dataset_filtered.csv};
\end{axis}
\end{tikzpicture}
         \caption{Marginal error}
         \label{fig:marginalerror}
     \end{subfigure}

\begin{tikzpicture}
\begin{axis}[
	axis line style={draw=none},
	tick style={draw=none},
	yticklabels={,,},
	xticklabels={,,},
    y=0.01cm,
	legend style={at={(0.5,-0.175)},anchor=north},
	legend columns=4,
    legend entries={Kullback-Leibler,reverse Kullback-Leibler,$\chi^2$,reverse $\chi^2$,squared Hellinger,Jensen-Shannon,Jeffreys,triangular discrimination},
]
\addplot[orange, thick, draw=none] table [x=epsilon, y=pi_nonzeros_ratio_avg, col sep=comma] {csv/kl_crescents_avg_and_std_over_random_seed_and_size_and_dataset_filtered.csv};
\addplot[blue, thick, draw=none] table [x=epsilon, y=pi_nonzeros_ratio_avg, col sep=comma] {csv/rkl_crescents_avg_and_std_over_random_seed_and_size_and_dataset_filtered.csv};
\addplot[red, thick, draw=none] table [x=epsilon, y=pi_nonzeros_ratio_avg, col sep=comma] {csv/chi2_crescents_avg_and_std_over_random_seed_and_size_and_dataset_filtered.csv};
\addplot[cyan, thick, draw=none] table [x=epsilon, y=pi_nonzeros_ratio_avg, col sep=comma] {csv/rchi2_crescents_avg_and_std_over_random_seed_and_size_and_dataset_filtered.csv};
\addplot[magenta, thick, draw=none] table [x=epsilon, y=pi_nonzeros_ratio_avg, col sep=comma] {csv/hellinger2_lowtol_crescents_avg_and_std_over_random_seed_and_size_and_dataset_filtered.csv};
\addplot[green, thick, draw=none] table [x=epsilon, y=pi_nonzeros_ratio_avg, col sep=comma] {csv/js_lowtol_crescents_avg_and_std_over_random_seed_and_size_and_dataset_filtered.csv};
\addplot[violet, thick, draw=none] table [x=epsilon, y=pi_nonzeros_ratio_avg, col sep=comma] {csv/jeffreys_crescents_avg_and_std_over_random_seed_and_size_and_dataset_filtered.csv};
\addplot[teal, thick, draw=none] table [x=epsilon, y=pi_nonzeros_ratio_avg, col sep=comma] {csv/triangular_crescents_avg_and_std_over_random_seed_and_size_and_dataset_filtered.csv};
\end{axis}
\end{tikzpicture}

        \caption{Experimental results}
        \label{fig:results}
\end{figure*}
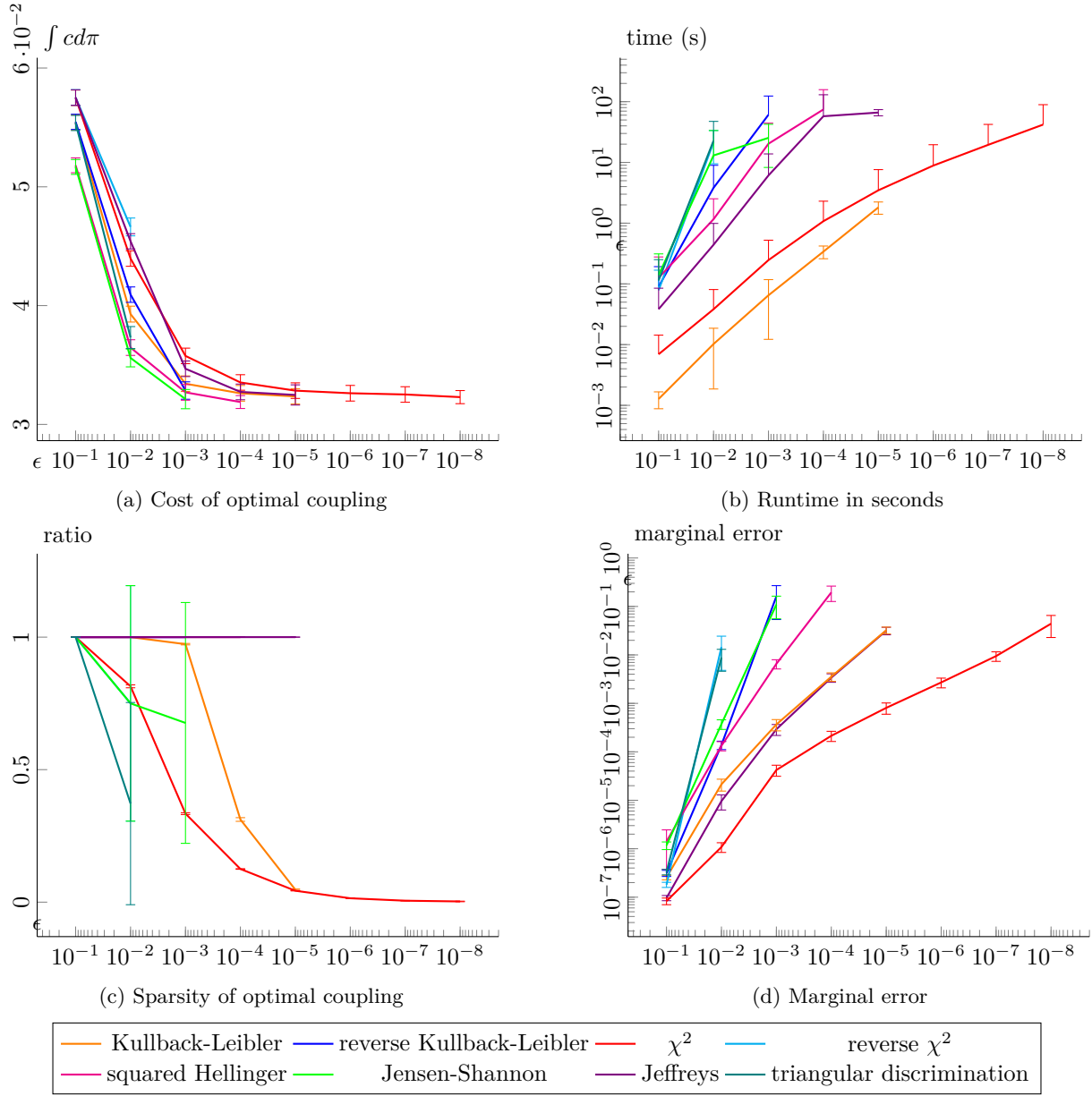

\subsection{Cost of optimal coupling and convergence speed}
Entropic regularization introduces a tradeoff between convergence speed of the Sinkhorn algorithm and bias in the optimal coupling. Increasing $\epsilon$ leads to faster convergence, but pushes the optimal coupling further away from the coupling which is optimal in the unregularized problem. In Figure~\ref{fig:cost} and Figure~\ref{fig:runtime} we can see, depending on $\epsilon$, the cost of the coupling obtained (i.e., $\int c d\pi$) as well as the time needed to compute it (in seconds). The values presented correspond to the "crescents" density pair, with means and standard deviations computed over all random seeds and pointcloud sizes\footnote{Different densities lead to markedly different ranges of costs of optimal couplings, which is why we did not average over them. Results for the other 3 pairs of densities can be found in Appendix~\ref{appendix_results}. The size of the point clouds sampled from the densities determines the memory requirements but seems to have little effect on convergence speed, which is why we averaged over this hyperparameter.}.  

We eliminated\footnote{We compute the marginal error with respect to $\nu$, because the order of $(c,\epsilon,\phi)$-transforms makes the marginal error with respect to $\mu$ vanish.} the data of couplings with a marginal error $\sum_{j = 1}^l \vert \langle \pi_{\cdot,j}, \mathbf{1}^k \rangle - \nu_j \vert$ greater than 0.2. In theory the optimal coupling should be (as its name says) a coupling, but for small $\epsilon$ the couplings obtained had marginals that differ greatly from its theoretical values $\mu$ and $\nu$. In these cases, one needs to set a lower tolerance parameter $\tau$ for Algorithm~\ref{algorithm_sinkhorn} in order to obtain a coupling with negligible marginal error. In Appendix~\ref{appendix_results} we include a more detailed account of this issue, and an additional experiment that shows the visual manifestation of the bias by pushing forward one of the point clouds through the transportation map defined by the gradient of $\int c d\pi$.

From the data it seems that the Kullback-Leibler and $\chi^2$ divergences lead to faster convergence compared to the others by a large margin. It is not surprising that Kullback-Leibler is the fastest, since $\gamma_{\phi,\xi}(h)$ is available in closed form for in this case, and no Newton's method iterations are needed. On the other hand, $\chi^2$ allows choosing $\epsilon$ from a much larger interval.

\subsection{Sparsity of optimal coupling and marginal error}
Let us first informally discuss why some divergences lead naturally to sparse optimal solutions. Recall from \eqref{eq:implicit_equation} that when computing the $(c,\epsilon,\phi)$-transform we are solving the problem of finding some $\gamma_y$ such that $\int {\phi_+^*}'\circ \frac{1}{\epsilon}(f+\gamma_y-c(\cdot,y))d\mu=1$. If $\epsilon$ starts to decrease, the argument of ${\phi_+^*}'$ increases. Thus, for the integral to be 1 (guaranteed by the theory) and since ${\phi_+^*}'$ is monotonic, $(f+\gamma_y-c(\cdot,y))$ needs to have small values. When ${\phi_+^*}'$ has 0 in its range, such as for the $\chi^2$ and triangular discrimination divergences, this will typically force ${\phi_+^*}'((f+\gamma_y-c(\cdot,y)))$ to be 0. Therefore, by equation \eqref{eq:density}, the density matrix of the optimal coupling will be sparse.

For the $\chi^2$ divergence, one has ${\phi_+^*}'((-\infty,-2])=0$, and for the triangular discrimination divergence, one has ${\phi_+^*}'((-\infty,-3])=0$. For all other divergences considered, one always has ${\phi_+^*}'>0$. Since the density of the optimal coupling is obtained as the image of ${\phi_+^*}'$, this leads to sparse couplings in the former case. In the latter, the optimal couplings are strictly positive. When using $\int c d\pi$ as a loss function, sparsity in the coupling $\pi$ leads to sparsity in the gradient tensors. This can be useful in practical scenarios as most automatic differentiation engines contain implementations of subroutines tailored for sparse tensors, which can be used in these cases to increase efficiency. An example is when a practitioner uses $\int c d\pi$ as the loss function with $\mu$ being a pointcloud output by a neural network and $\nu$ being a ground truth point cloud. In this case, if the coupling $\pi$ is not sparse, any point of $\mu$ receives backpropagated gradients from most of the points of $\nu$, whereas if $\pi$ is sparse, then it only receives gradients from a few of them.

Quantitatively, sparsity in terms of the quotient of positive elements to all elements in the optimal couplings and marginal errors obtained in the experiments are visualized in Figure~\ref{fig:ratio} and Figure~\ref{fig:marginalerror}. As expected, the $\chi^2$ and triangular discrimination divergences naturally lead to sparse couplings. A consequence of limited machine precision is that the couplings will be empirically sparse even in other cases, notably for the Kullback-Leibler divergence, which reaches the same sparsity as $\chi^2$ at $\epsilon=10^{-5}$. However, for this $\epsilon$, the marginal error in the Kullback-Leibler case is 40 times larger than for $\chi^2$. For greater values of $\epsilon$ (and therefore shorter running time), couplings obtained using $\chi^2$ are more sparse by a large margin.

\subsection{Conclusions}
The classical setup using the Kullback-Leiber divergence is the fastest to compute and gives low costs in general terms. However, the $\chi^2$ divergence, albeit being marginally slower, can obtain a similar cost but with a much more sparse coupling for values of $\epsilon$ corresponding to shorter running times. As we discussed above, the optimal coupling can be used to compute the gradient tensor with respect to the cost and thus a sparse tensor could lead to benefits using subroutines tailored for sparse tensors present in most automatic differentiation engines. Other $f$-divergences do not seem to induce practical benefits from this limited set of experiments (intended to showcase the feasibility of the generalized Sinkhorn algorithm), but may turn out to be useful in other scenarios.

\section{LIMITATIONS} \label{section_limitations}

From the theoretical side, the main limitation of our paper is the assumption that the cost function is Lipschitz. We explained in the corresponding section the reasons why we decided to work in this setup. A more general theory may be able to include lower semicontinuous costs, but we did not pursue this in the present work. The Legendre type assumption on $\phi$ excludes the total variation divergence, but it is necessary in order to have a well-defined $(c,\epsilon,\phi)$-transform. Another limitation in our work is that while we believe that the Sinkhorn algorithm may fail to converge to optimal variables if no condition like the Good Triple is assumed, we did not present an explicit example of this behavior. Finally, we did not study the theoretical complexity of the generalized Sinkhorn algorithm, $\Gamma$-convergence of $\OT_\epsilon$ to $\OT_0$, explicit formulas of the subdifferential of $\OT_{\epsilon}(\cdot,\nu)$, nor the generalization of Sinkhorn divergences. We leave these for future projects.

On the practical part, we believe that the implementation of Newton's method could be optimized for each $\phi$. For the tolerances, there should be at least a heuristic way of choosing them in terms of $\epsilon$ and $\phi$ in order to have the marginal conditions satisfied at convergence.

\subsubsection*{Acknowledgements}
D\'avid Terj\'ek is supported by the Hungarian National Excellence Grant 2018-1.2.1-NKP-00008 and by the Hungarian Ministry of Innovation and Technology NRDI Office within the framework of the Artificial Intelligence National Laboratory Program. Diego Gonz\'alez-S\'anchez is supported by projects KPP 133921 and Momentum (Lend\"{u}let) 30003.

The authors would like to thank Mih\'aly Weiner from the Department of Mathematical Analysis at Budapest University of Technology and Economics for his help and in particular for proposing Example~\ref{example}, as well as the anonymous reviewers for their useful comments.

\clearpage


\bibliographystyle{apalike}
\bibliography{934}


\clearpage
\appendix

\thispagestyle{empty}

\onecolumn \makesupplementtitle

\section{Mathematical background} \label{appendix_background}

\subsection{Functional analysis}

We are going to recite some of the most important results that we need for our paper. We refer the reader to standard reference in the area \citep{Cobzasetal2019} for a detailed account on them.

In this paper, $X$ and $Y$ will denote compact metric spaces except if noted otherwise. We will be interested in studying the set of Lipschitz functions on these sets. We say that a function $f:X\to \mb{R}$ is Lipschitz if there exists $K\ge 0$ such that $|f(x)-f(x')|\le Kd_X(x,x')$ for all $x,x'\in X$ where $d_X$ is a metric on $X$. We will denote the set of Lipschitz functions on $X$ by $\Lip(X)$. The optimal $K$ such that the above holds will be the Lipschitz constant of $f$, $\|f\|_L=\sup_{x\not=x'}\left\{\frac{|f(x)-f(x')|}{d_X(x,x')}\right\}$. Clearly this definition depends on the metric $d_X$ (resp. $d_Y$) chosen on $X$ (resp. $Y$) but for our purposes we will fix some metric on $X$ (resp. $Y$) and all the Lipschitz constants will be relative to it. Furthermore, in the product space $X\times Y$ we will assume that we have a metric $d_{X\times Y}$ such that $d_{X\times Y}((x,y),(x,y')) = d_Y(y,y')$ (and similarly fixing $y$ instead of $x$). For example, it can be assumed for the rest of the paper that $d_{X\times Y}((x,y),(x',y')) = \max\{d_X(x,x'),d_Y(y,y')\}$.

As $X$ and $Y$ are compact spaces, it will always make sense to talk also about the $\|\cdot\|_{\infty}$ norm of a Lipschitz function on $X$ or $Y$ (and it will always be finite). Thus, we define $\|f\|_{\infty}:=\sup_{x\in X}\{|f(x)|\}$. We can combine the Lipschitz constant with the supremum norm to create the following norm on $\Lip(X)$ (resp. $\Lip(Y)$), $\|f\|_{\max}:=\max\{\|f\|_L,\|f\|_{\infty}\}$.

For any compact metric space $(X,d_X)$ we will denote by $\mc{B}(X)$ the Borel $\sigma$-algebra on $X$. A measure on $X$ is a function $\mu:\mc{B}(X)\to \mb{R}$ such that $\mu(\emptyset)=0$ and for pairwise disjoint elements $(A_i\in \mc{B}(X))_{i\ge1}$ we have $\mu(\cup_{i=1}^{\infty}A_i)=\sum_{i=1}^{\infty} \mu(A_i)$. The variation of a measure $\mu$ is defined by:

\[
|\mu|(A) := \sup_{A=\cup_{i=1}^nB_i} \sum_{i=1}^n |\mu(B_i)|
\]
where the $B_i$ are pairwise disjoint and the supremum is taken over all possible partitions. The total variation of $\mu$ is then defined as $\|\mu\|:=|\mu|(X)$. We will denote by $\mc{M}(X)$ the set of Borel measures on $X$ with finite total variation. Similarly, $\mc{M}_+(X)$ will denote the subset of measures $\mu\in \mc{M}(X)$ such that $\mu\ge 0$ and $\mc{M}(X,\xi)$ for some $\xi \in \mb{R}$ will denote the set of measures such that $\mu(X)=\xi$. Finally $P(X)$ will denote the set of probability measures on $X$, i.e., $\mc{M}_+(X)\cap \mc{M}(X,1)$.

Given two measures $\mu,\nu\in \mc{M}(X)$ we will say that $\mu$ is absolutely continuous with respect to $\nu$ and denote it by $\mu\ll \nu$ if for all $A\in \mc{B}(X)$ if $\nu(A)=0$ then $\mu(A)=0$. In this case, we will denote the Radon-Nikodym derivative as $\frac{d\mu}{d\nu}\in L^1(\nu)$. We will say that two measures $\mu,\nu\in \mc{M}(X)$ are singular to each other when there exists $A\in \mc{B}(X)$ such that $|\mu|(A)=0$ and $|\nu|(X\setminus A)=0$. Given a measure $\mu\in \mc{M}(X)$ for a compact metric space $X$ we define its support as $\support(\mu):= X\setminus(\cup_{\{U\subset X \text{ open and }|\mu|(U)=0\}}U)$. Note that this is always a closed and compact subset of $X$.

For any $\mu\in \mc{M}(X,0)$ we are interested in defining

\[
\|\mu\|_{KR}:=\sup \left\{ \int f \;d\mu : \|f\|_L\le 1  \right\}.
\]
With this, we can define the Hanin norm, which will be central in this paper. Given $\mu\in \mc{M}(X)$ we define

\[
\|\mu\|_H := \inf_{\nu\in \mc{M}(X,0)} \{ \|\nu\|_{KR}+\|\mu-\nu\|\}.
\]
The importance of this norm relies on the following theorem:

\begin{theorem*} Let $(X,d_X)$ be a compact metric space. Then

\[
(\Lip(X),\|\cdot\|_{\max})\simeq (\mc{M}(X),\|\cdot\|_H)^*
\]
and furthermore there is an isometric linear isomorphism given by the mapping that sends $f\in \Lip(X)$ to the functional $\mu\mapsto \int f\;d\mu$.
\end{theorem*}

\begin{remark} For the rest of the paper, the normed spaces of measures $\mc{M}(X),\mc{M}(Y)$ and $\mc{M}(X\times Y)$ will be assumed to have the Hanin norm and the normed spaces $\Lip(X),\Lip(Y)$ and $\Lip(X\times Y)$ the $\max$ norm unless stated otherwise (as for example in $(vi)$ of Proposition~\ref{prop:properties-c-eps-phi-trans} where the $\|\cdot\|_{\infty}$-norm is used).\end{remark}

Given $X$ and $Y$ compact metric spaces and $\pi\in \mc{M}(X\times Y)$ let $p_1:X\times Y\to X$  be the map $(x,y)\mapsto x$. We denote by $p_1^*(\pi)\in \mc{M}(X)$ the pushforward measure of $\pi$, i.e. $p_1^*(\pi)(A) := \pi(p_1^{-1}(A))$ for any $A\in \mc{B}(X)$. We do an analogous definition with with $p_2:X\times Y\to Y$ $(x,y)\mapsto y$. If we let now $\mu\in P(X)$ and $\nu\in P(Y)$ we can define $\Pi(\mu,\nu):=\{\pi\in \mc{M}_+(X\times Y):p_1^*(\pi)=\mu,p_2^*(\pi)=\nu\}$. 

\subsection{Convex analysis \citep{Zalinescu2002}}
Given a topological vector space $X$, denote its topological dual by $X^*$, i.e. the set of real-valued continuous linear maps on $X$, which is a topological vector space itself, and the canonical pairing by $\langle \cdot,\cdot \rangle: X \times X^* \to \R$, which is the continuous bilinear map $(x,x^*) \to \langle x, x^* \rangle = x^*(x)$. Given a function $f: X \to \overline{\R}$, the set $\dom f = \left\{ x \in X : f(x) < \infty \right\}$ is the effective domain of $f$. A function $f$ is proper if $\dom f \neq \emptyset$ and $f(x) > -\infty$ for all $x \in X$, otherwise it is improper. For a convex function $f : X \to \overline{\R}$, its convex conjugate is $f^* : X^* \to \overline{\R}$ defined by $f^*(x^*)=\sup_{x \in X}\{\langle x, x^* \rangle - f(x) \}$, and its subdifferential at $x \in X$ is the set $\partial f(x) = \{ x^* \in X^* \ \vert \ \forall \hat{x} \in X : \langle \hat{x} - x , x^* \rangle \leq f(\hat{x}) - f(x) \}$, singleton if and only if $f$ is Gateaux differentiable at $x$.

\begin{remark}It should not be confused the pushforward operators $p_1^*$ and $p_2^*$ with the convex conjugate of a convex function (represented also with the symbol ${}^*$). We believe that this will make no confusion as the only pushforward of a measure will be represented as $p_1^*$ and $p_2^*$. All the rest of ${}^*$ are convex conjugates.\end{remark}

\section{Complete proofs} \label{appendix_proofs}

Let us start with an easy result that shows that the pushforward operation of a measure is continuous between the spaces of measures in the Hanin norm:

\begin{proposition}\label{prop:cont-M-1} Let $X$ and $Y$ be compact metric spaces.
The map $p_1^* : \mathcal{M}(X \times Y) \to \mathcal{M}(X)$ is linear and continuous.
\end{proposition}
\begin{proof} 
Let $p_1:X\times Y\to X$ be the projection to the first coordinate, $(x,y)\mapsto x$. As $p_1^*(\pi)=\pi\circ p_1^{-1}$ and thus it follows directly that this operator is linear. To see that it is continuous it is enough to check that it is bounded. Let $\pi\in \mathcal{M}(X \times Y)$ be such that $\|\pi\|_H\le 1$. Let $\nu\in \mathcal{M}(X \times Y,0)$ be such that $\|\nu\|_{KR}+\|\pi-\nu\|\le 1+\epsilon$ for some $\epsilon>0$. We have to see that $\|p_1^*(\pi)\|_H$ is bounded.

It suffices to see that $\|\nu\circ p_1^{-1}\|_{KR}+\|\pi\circ p_1^{-1}-\nu\circ p_1^{-1}\|$ is bounded by some constant (independent of $\nu$). Clearly we have that $\nu\circ p_1^{-1}\in \mathcal{M}(X,0)$ as $\nu\circ p_1^{-1}(X)=\nu(X\times Y)=0$. By definition

\[
\|\nu\circ p_1^{-1}\|_{KR}=\sup\left\{\int f\circ p_1 \;d\nu\; : \;\|f\|_L\le 1\right\}.
\]
Consider $X\times Y$ with the maximun distance, $d_{X\times Y} = \max(d_X,d_Y)$. It is easy to see that $f\circ p_1\in \Lip(X\times Y)$ as $|(f\circ p_1)(x,y)-(f\circ p_1)(x',y')| = |f(x)-f(x')|\le \|f\|_Ld_X(x,x')\le \|f\|_Ld_{X\times Y}((x,y),(x'y'))$. If $\|f\|_L\le 1$ then $\|f\circ p_1\|_L\le 1$ as well. Therefore $\|\nu\circ p_1^{-1}\|_{KR} \le \|\nu \|_{KR}$ (as essentially we are taking the supremum over a larger set).

For the total variation part, note that given a partition $A_1,\ldots,A_m$ of $X$, this automatically gives us a partition of $X\times Y$ induced by $p_1^{-1}$, namely $A_1\times Y,\ldots,A_m\times Y$. Therefore $\|\pi\circ p_1^{-1}-\nu\circ p_1^{-1}\| \le \|\pi-\nu\|$. Thus $\|p_1^*(\pi)\|_H\le 1+\epsilon$ for all positive $\epsilon$ and therefore $\|p_1^*(\pi)\|_H\le \|\pi\|_H$ and the functional is continuous. \end{proof}

Clearly a similar argument shows that $p_2^*$ is linear and continuous.

\begin{proposition}\label{prop:weak-Legendre}
If a proper, convex and lower semicontinuous function  $\phi : \R \to \overline{\R}$ is of Legendre type, then $\phi_+ = \phi + \iota_{\R_+}$ is strictly convex and differentiable on $\dom \phi_+ = \dom \phi \cap \R_+$, $\phi_+^*$ is strictly convex and differentiable on $\dom \phi_+^*$, and $(\phi_+')^{-1} = {\phi_+^*}'$ on the set $\{ t \in \R : {\phi_+^*}'(t) > 0 \}$.
\end{proposition}
\begin{proof}
If $\dom \phi \subset \R_+$, the proposition is immediate. Assume the contrary, so that $\phi(0) \in \R$. By definition, for $t \in \R$
\begin{equation}
\phi_+^*(t) = \sup_{s \in \R} \{ st - \phi_+(s) \} = \sup_{s \in \R_+} \{ st - \phi(s) \},
\end{equation}
which is a strictly concave constrained maximization problem. The first derivative test gives
\begin{equation}
t-\phi'(s)=0,
\end{equation}
giving the optimum
\begin{equation}
s = {\phi^*}'(t)
\end{equation}
\citep[Lemma~2.6]{Borweinetal1993}. If ${\phi^*}'(t) \geq 0$, the constraint is satisfied and the optimum is valid. Otherwise, since the problem is strictly concave, the optimum is going to be $s=0$, giving
\begin{equation}
\phi_+^*(t) = \begin{cases}
{\phi^*}'(t)t - \phi({\phi^*}'(t)) = \phi^*(t) & \text{ if } {\phi^*}'(t) \geq 0,\\
- \phi(0) & \text{ otherwise.}
\end{cases}
\end{equation}
Differentiating gives
\begin{equation}
{\phi_+^*}'(t) = \begin{cases}
{\phi^*}'(t) & \text{ if } {\phi^*}'(t) \geq 0,\\
0 & \text{ otherwise,}
\end{cases}
\end{equation}
proving the proposition.
\end{proof}

In the following, we denote by $\langle \mu, f \rangle$ the integral $\int f d\mu$ of $f \in \Lip(X)$ and $\mu \in \M(X)$, since it is exactly the dual pairing for the duality of $(\M(X),\Vert.\Vert_H)$ and $(\Lip(X),\Vert.\Vert_{\max})$ (and similarly for the spaces $Y$ and $X \times Y$). The mapping $(\mu \to D_\phi(\mu \Vert \nu))$ is denoted $I_{\phi,\nu}$, the theory of which can be found in the literature \citep{Agrawaletal2020, Borweinetal1993, Terjek2021}. We begin with a technical result that will help us prove strong duality.

\begin{proposition}\label{primal_map_conjugate} Let $X$ and $Y$ be compact metric spaces and $c \in \Lip(X \times Y)$. Let also $\phi : \R \to \overline{\R}$ be a proper, convex and lower semicontinuous function of Legendre type with $\phi(1)=0$. Let $\epsilon > 0$, $\mu \in P(X)$ and $\nu \in P(Y)$, and define the map $T_{c,\phi,\epsilon,\mu,\nu} : \mathcal{M}(X \times Y) \to \overline{\R}$ as 
\[
T_{c,\phi,\epsilon,\mu,\nu}(\pi) = \langle \pi,c\rangle + \epsilon I_{\phi_+,\mu \otimes \nu}(\pi) + \iota_{\{(\mu,\nu)\}}(p_1^*(\pi), p_2^*(\pi)).
\]
Then this map is proper, convex, and lower semicontinuous, and its conjugate $T_{c,\phi,\epsilon,\mu,\nu}^* : \Lip(X \times Y) \to \overline{\R}$ is 
\[
T_{c,\phi,\epsilon,\mu,\nu}^*(\varphi) = \inf_{(f,g) \in \Lip(X) \times \Lip(Y)}\left\{ \epsilon I_{\phi_+,\mu \otimes \nu}^*\left(\frac{1}{\epsilon} (\varphi - c - f \oplus g)\right) + \langle \mu, f \rangle + \langle \nu, g \rangle \right\}.
\]
\end{proposition}

\begin{proof} First we define $\Phi_\varphi : \mathcal{M}(X \times Y) \times (\mathcal{M}(X) \times \mathcal{M}(Y)) \to \overline{\R}$ as
\[
\Phi_\varphi(\pi, (\xi, \rho)) = \iota_{\{(\mu,\nu)\}}(p_1^*(\pi) - \xi, p_2^*(\pi) - \rho)+\langle \pi,c-\varphi\rangle +\epsilon I_{\phi_+,\mu \otimes\nu}(\pi) .\]
for a function $\varphi\in \Lip(X\times Y)$. Now note that

\begin{align*}
T^*_{c,\phi,\epsilon,\mu,\nu}(\varphi)&=\sup_{\pi\in \mathcal{M}(X\times Y)}\{ \langle\pi, \varphi\rangle-\epsilon I_{\phi_+,\mu \otimes\nu}(\pi)- \iota_{\{(\mu,\nu)\}}(p_1^*(\pi), p_2^*(\pi)) -\langle \pi,c\rangle\}\\
&= \sup_{\pi\in \mathcal{M}(X\times Y)}\left\{ -\Phi_{\varphi}(\pi,0,0) \right\}= -\inf_{\pi\in \mathcal{M}(X\times Y)}\left\{  \Phi_{\varphi}(\pi,0,0) \right\}.
\end{align*}
 
Further, suppose that the following convex optimization problem can be solved:

\begin{equation}\label{eq:f-entr-reg-1}
-\inf_{\pi \in \mathcal{M}(X \times Y)}\{ \Phi_\varphi(\pi, 0, 0) \} = \inf_{(f,g) \in \Lip(X) \times \Lip(Y)}\{ \Phi_\varphi^*(0, (-f, -g)) \}
\end{equation}
Then this would imply that

\[
T^*_{c,\phi,\epsilon,\mu,\nu}( \varphi)= \inf_{(f,g) \in \Lip(X) \times \Lip(Y)}\{ \Phi_\varphi^*(0, (-f, -g)) \}.
\]

Now, by definition of $\Phi_\varphi^*(0, (-f, -g))$ we have that

\begin{align*}
\Phi_\varphi^*(0, (-f, -g))&
= \sup_{\pi,\xi,\rho} \{ \langle \xi,-f\rangle+\langle \rho,-g\rangle+ \langle\pi, \varphi\rangle -\langle \pi,c\rangle \\&-\iota_{\{(\mu,\nu)\}}(p_1^*(\pi)-\xi,p_2^*(\pi)-\rho)-\epsilon I_{\phi_+,\mu\otimes\nu}(\pi) \}.
\end{align*}
Changing the variables $\eta:=p_1^*(\pi)-\xi$ and $\tau:=p_2^*(\pi)-\rho$ the previous equation equals:
\[
\sup_{\pi,\eta,\tau} \{ \langle \varphi-f\oplus g,\pi\rangle +\langle f,\eta\rangle+\langle g,\tau\rangle-\langle \pi,c\rangle -\iota_{\{(\mu,\nu)\}}(\eta,\tau)-\epsilon I_{\phi_+,\mu\otimes\nu}(\pi) \}.\\
\]
It is clear that without loss of generality we can assume that $\eta=\mu$ and $\tau=\nu$ (as otherwise the value inside the supremum is $-\infty$). Hence
\begin{align*}
\Phi_\varphi^*(0, (-f, -g))&
= \sup_{\pi} \{ \langle \varphi-f\oplus g,\pi\rangle-\epsilon I_{\phi_+,\mu\otimes\nu}(\pi) -\langle \pi,c\rangle\}+\langle f,\mu\rangle+\langle g,\nu\rangle\\
&= (\epsilon I_{\phi_+,\mu\otimes \nu})^*( \varphi-f\oplus g-c)+\langle f,\mu\rangle+\langle g,\nu\rangle.
\end{align*}

And this will conclude the proof since \[T^*_{c,\phi,\epsilon,\mu,\nu}( \varphi)= \inf_{(f,g) \in \Lip(X) \times \Lip(Y)}\{ (\epsilon I_{\phi_+,\mu\otimes \nu})^*( \varphi-f\oplus g-c)+\langle f,\mu\rangle+\langle g,\nu\rangle \}\]
and $(\epsilon I_{\phi_+,\mu\otimes \nu})^* = \epsilon I_{\phi_+,\mu\otimes \nu}^*(\frac{1}{\epsilon} \cdot)$ \citep[Theorem~2.3.1(v)]{Zalinescu2002}.

Thus, it only remains to check that \eqref{eq:f-entr-reg-1} holds. To do so, we know that we have strong duality if the marginal function

\[
h_{\varphi}(\xi,\rho):=\inf_{\pi\in \mathcal{M}(X\times Y)}\Phi_{\varphi}(\pi,\xi,\rho)
\]
is lower semicontinuous at the origin and $h_{\varphi}(0,0)\in \mathbb{R}$ \citep[Theorem~2.6.1(v)]{Zalinescu2002}. First note that taking $\pi=\mu\otimes \nu$ it is easy to see that the infumum is not $\infty$. To see that it is not equal to $-\infty$, note that \citep[Paragraph before Remark 4.1.4]{Agrawaletal2020}
\[
I_{\phi_+,\mu \otimes\nu}(\pi)\ge  0 
\]
for any $\pi\in \M(X\times Y)$. If we take now any $\pi$ such that $\Phi_{\varphi}(\pi,0,0)\not=\infty$ it is clear that $p_1^*(\pi)=\mu$ and $\pi$ is a positive measure. Thus, we have the bound  $\Phi_{\varphi}(\pi,0,0) \ge -\|c-\varphi\|_{\infty}$ for all $\pi\in \mc{M}(X\times Y)$. Hence, we have that $h_{\varphi}(0,0)\in \mb{R}$.

To prove lower semicontinuity at the origin we have to prove that given $(\xi_n,\rho_n)\in \mathcal{M}(X)\times \mathcal{M}(Y)$ with $(\xi_n,\rho_n)\to (0,0)$ as $n\to \infty$ then we have that $h_{\varphi}(0,0)\le \liminf_{n\to \infty}h_{\varphi}(\xi_n,\rho_n)$. Note that for $n$ large enough we can assume without loss of generality $\max(\|\xi_n\|_H,\|\rho_n\|_H)\le 1$. Note also that if $\|\xi\|_H \le 1$ then $\Phi_{\varphi}(\pi,\xi,\rho)$ is bounded from below. This is because in order to have a value different from $\infty$ we must have that $\pi\ge 0$ (otherwise $I_{\phi_+,\mu\otimes\nu}(\pi)=\infty$) and also $p_1^*(\pi)=\xi+\mu$. In particular, the total variation of $\pi$ can be bounded as follows
\[
\|\pi\| = \pi(X\times Y)=p_1^*(\pi)(X)=(\xi+\mu)(X) \le \|\xi\|_H+1\le 2,
\]
where in the first equality we have used that $\pi\ge 0$ and for the first inequality we have used \citet[Proposition~8.5.2(ii)]{Cobzasetal2019}. Thus, we have that
\[
\Phi_{\varphi}(\pi,\xi,\rho) \ge \langle c-\varphi,\pi\rangle \ge -\|\varphi-c\|_{\infty}\|\pi\|\ge -2\|\varphi-c\|_{\infty}
\]
whenever $\|\xi\|_H\le 1$.

Hence, if $\max(\|\xi_n\|_H,\|\rho_n\|_H)\le 1$ then $h_{\varphi}(\xi_n,\rho_n)>-\infty$. It is clear that if $\liminf_{n\to\infty} h_{\varphi}(\xi_n,\rho_n)=\infty$ then the lower semicontinuity of this sequence is verified. Hence, passing through a subsequence if necessary we can assume that $h_{\varphi}(\xi_n,\rho_n)$ are all finite and that $\liminf_{n\to\infty} h_{\varphi}(\xi_n,\rho_n)=\lim_{n\to\infty} h_{\varphi}(\xi_n,\rho_n)$. Now, for each $n$ let $\pi_n\in \mathcal{M}(X\times Y)$ be such that $|\Phi_{\varphi}(\pi_n,\xi_n,\rho_n)-h_{\varphi}(\xi_n,\rho_n)|<1/n$. Note that without loss of generality we can assume that $\pi_n\ge 0$ (using the same arguments as we used in the previous paragraph). In particular, we have that $\|\pi_n\|\le 2$ for all $n$ large enough (so that $\|\xi_n\|_H\le 1$). By \citet[Remark~8.5.9]{Cobzasetal2019} we have that the set $\{\pi\in \mathcal{M}(X\times Y):\|\pi\|\le 2\}$ is compact in the Hanin norm and therefore there exists a convergent subsequence $\pi_n\to\pi$ (that abusing the notation we denote just by $n$). Hence
\begin{align*}
    \liminf_{n\to\infty} h_{\varphi}(\xi_n,\rho_n) \ge \liminf_{n\to\infty} \Phi_{\varphi}(\pi_n,\xi_n,\rho_n)-1/n = \liminf_{n\to\infty} \Phi_{\varphi}(\pi_n,\xi_n,\rho_n)\\
    \ge \Phi_{\varphi}(\pi,0,0) \ge \inf_{\pi\in \mathcal{M}(X\times Y)}\Phi_{\varphi}(\pi,0,0) = h_{\varphi}(0,0).
\end{align*}
Where we have used that $\Phi_{\varphi}$ is lower semicontinuous. To prove this, we just have to prove that it is the sum of lower semicontinuous functions. Clearly $\langle \cdot,c-\varphi\rangle$ is continuous, the indicator function is also lower semi-continuous, and $I_{\phi,\mu\otimes \nu}$ is lower semi-continuous in the Hanin norm \citep[Proposition~7]{Terjek2021}. Using that clearly $(\pi_n,\xi_n,\rho_n)\to (\pi,0,0)$ as $n\to\infty$ the result follows. The map $T_{c,\phi,\epsilon,\mu,\nu}$ is easily seen to be proper, convex and lower semicontinuous.\end{proof}

Let us recall the definition of $c$-transform \citep[Definition~5.2]{Villani2008}.

\begin{definition}[$c$-transform] Let $f\in \Lip(X)$ and $c\in \Lip(X\times Y)$ for some compact metric spaces $X$ and $Y$. We define the $c$-transform of $f$ as follows:

\[
f^c(y):=\inf_{x\in X}\{c(x,y)-f(x)\}.
\]
\end{definition}

\begin{proposition}\label{prop:properties-c-transform} Let $f\in \Lip(X)$ and $c\in \Lip(X\times Y)$ for some compact metric spaces $X$ and $Y$. Then the $c$-transform of $f$ has the following properties:

\begin{enumerate}[(i)]
    \item If $g\in \Lip(Y)$ is such that $f\oplus g\le c$ then $g\le f^c$.
    \item $f\oplus f^c \le c$.
    \item $f^c\in \Lip(Y)$ and $\|f^c\|_L\le \|c\|_L$.
    \item $\|f^c\|_{\infty} \le \|f\|_{\infty}+\|c\|_{\infty}$.
\end{enumerate}
\end{proposition}

\begin{proof} Most of the properties follow immediatelly from the definitions. For $(iii)$, note that given $y,y'\in Y$ we have

\begin{align*}
    f^c(y)-f^c(y') & = \inf_{x\in X}\{c(x,y)-f(x)\} - \inf_{x\in X}\{c(x,y')-f(x)\}\\
    & \le \inf_{x\in X}\{c(x,y')-f(x)+\|c\|_Ld_Y(y,y')\}-\inf_{x\in X}\{c(x,y)-f(x)\} \\
     & \le \|c\|_Ld_Y(y,y'),
\end{align*}
where we have assumed that $d_{X\times Y}((x,y),(x,y')) = d_Y(y,y')$. Swapping the roles of $y$ and $y'$ we have the other inequality and therefore $|f^c(y)-f^c(y')|\le \|c\|_Ld_Y(y,y')$. \end{proof}

The natural generalization of the $c$-transform to the regularized optimal transport problem is the following.

\begin{definition}[$(c,\epsilon,\phi)$-transform] Let $c \in \Lip(X \times Y)$, $\phi : \R \to \overline{\R}$ a proper, convex and lower semicontinuous function of Legendre type with $\phi(1)=0$, $\epsilon > 0$, $\mu \in P(X)$ and $\nu \in P(Y)$. We define the $(c,\epsilon,\phi)$-transform of $f$ as follows:

\[
f^{(c,\epsilon,\phi)}(y):=\argmax_{\gamma\in \mb{R}} \left\{ \frac{1}{\epsilon}\gamma-\int \phi_+^*\left(\frac{1}{\epsilon}(f(x)+\gamma-c(x,y))\right)\;d\mu(x) \right\}.
\]
\end{definition}

Note that in this definition we can assume that $\gamma\le f^c(y)+\epsilon\phi'(\infty)$ as otherwise it is clear that the function inside the $\argmax$ is going to be $-\infty$ \citep[Lemma~2.1]{Borweinetal1993}.

Let us now prove some properties of the $(c,\epsilon,\phi)$-transform:

\begin{proposition}\label{prop:properties-c-eps-phi-trans} Let $(X,d_X)$ and $(Y,d_Y)$ be compact metric spaces. Let also $\mu \in P(X)$ be of full support, i.e. $\support(\mu)=X$, $c \in \Lip(X \times Y)$ a cost function, $0 < \epsilon \in \R$ a regularization coefficient and $\phi : \R \to \overline{\R}$ a proper, convex and lower semicontinuous function  of Legendre type. Then one has that for any $f \in\Lip(X)$:
\begin{enumerate}[(i)]
    \item $f^{(c,\epsilon,\phi)}(y)$ is well-defined for all $y \in Y$ implicitly by $\int_X {\phi_+^*}' \circ \frac{1}{\epsilon}(f + f^{(c,\epsilon,\phi)}(y) - c(\cdot,y)) d\mu = 1$ if there exists such number $f^{(c,\epsilon,\phi)}(y)$ or explicitly as $f^{(c,\epsilon,\phi)}(y)=\min_{x \in X}\{\epsilon \phi'(\infty) + c(x,y)-f(x)\}=f^c(y)+\epsilon\phi'(\infty)$ otherwise.
    \item $f(x)+f^{(c,\epsilon,\phi)}(y) \le c(x,y)+\epsilon \phi'(\infty)$ for all $x \in X$ and $y \in Y$.
    \item $\|f^{(c,\epsilon,\phi)}\|_L\le \|c\|_L$.
    \item $\|f^{(c,\epsilon,\phi)}\|_{\infty} \le \|f\|_{\infty}+\|c\|_{\infty}$ if $\phi'(\infty)=\infty$ and $\|f^{(c,\epsilon,\phi)}\|_{\infty} \le \|f\|_{\infty}+\|c\|_{\infty}+\epsilon\phi'(\infty)$ otherwise.
    \item For any $a\in \mb{R}$ we have $(f+a)^{(c,\epsilon,\phi)} = f^{(c,\epsilon,\phi)} -a$.
    \item The map from $\Lip(X)\to \Lip(Y)$ that sends a function to its $(c,\epsilon,\phi)$-transform is 1-Lipschitz with respect to the $\|\cdot\|_{\infty}$-norm.
\end{enumerate}
\end{proposition}

Clearly analogous properties hold if we consider the $(c,\epsilon,\phi)$-transform defined as
\[
g^{(c,\epsilon,\phi)}(x):=\argmax_{\gamma\in \mb{R}} \left\{ \frac{1}{\epsilon}\gamma-\int \phi_+^*\left(\frac{1}{\epsilon}(\gamma+g(y)-c(x,y))\right)\;d\nu(y) \right\}.
\]
of a function $g\in \Lip(Y)$.

\begin{proof} Let us start proving $(i)$. Fix any $y\in Y$ and consider the following Primal Problem

\begin{equation}
    \inf_{\xi\in \mc{M}(X,1)}\left\{ I_{\phi_+,\mu}(\xi)-\int \frac{1}{\epsilon}(f-c(\cdot,y))\; d\xi \right\}
\end{equation}
and the corresponding Dual Problem

\begin{equation}
\sup_{\gamma\in \mb{R}} \left\{ \frac{1}{\epsilon}\gamma-\int \phi_+^*\left(\frac{1}{\epsilon}(f+\gamma-c(\cdot,y))\right)\;d\mu \right\}.
\end{equation}

First, let us verify that the Primal Constraint Qualifications (Primal CQ) and the Dual Constraint Qualifications (Dual CQ) \citep[p. 254 and p. 255]{Borweinetal1993} are satisfied. To verify the Primal CQ just note that taking $\frac{d\xi}{d\mu}=1$ this condition holds (i.e. $\xi=\mu$). For the Dual CQ, note that if $\gamma$ is such that $\gamma < f^c(y)+\epsilon\phi'(\infty)$ then this condition holds as well (as $\phi_+'(-\infty)=-\infty$).

Thus, we get that both the Primal and Dual Problems have (in principle non-necessarily unique) optimal solutions $\widehat{\xi}$ and $\widehat{\gamma}$ respectively \citep[Theorem~4.1 (i), (ii) and (iii)]{Borweinetal1993}. Furthermore, if we decompose $\widehat{\xi} = \frac{d\widehat{\xi}_c}{d\mu}\mu+(\widehat{\xi}_s)_+-(\widehat{\xi}_s)_-$ where $\widehat{\xi}_c$ is the absolutely continuous part with respect to $\mu$ and $(\widehat{\xi}_s)_+$ and $(\widehat{\xi}_s)_-$ is the Jordan decomposition of the singular part we have that $\frac{d\widehat{\xi}_c}{d\mu}$ is uniquely defined $\mu$-a.e. and $\frac{d\widehat{\xi}_c}{d\mu} = {\phi_+^*}'(\frac{1}{\epsilon}(f+\widehat{\gamma}-c(\cdot,y)))$.

Now suppose that we have two optimal $\widehat{\gamma}_1$ and $\widehat{\gamma}_2$ and that the absolutely continuous part is nonzero. By uniqueness of the absolutely continuous part we have that ${\phi_+^*}'(\frac{1}{\epsilon}(f+\widehat{\gamma}_1-c(\cdot,y))) = {\phi_+^*}'(\frac{1}{\epsilon}(f+\widehat{\gamma}_2-c(\cdot,y)))$ $\mu$-a.e. First note that $\int_X  {\phi_+^*}'(\frac{1}{\epsilon}(f+\widehat{\gamma}_1-c(\cdot,y)))d\mu >0$ as otherwise the absolutely continuous part would be 0. Thus, there exists an open set $U_y\subset X$ of positive measure such that ${\phi_+^*}'(\frac{1}{\epsilon}(f(x)+\widehat{\gamma}_1-c(x,y)))>0$ for all $x\in U_y$. Without loss of generality we can assume that $\widehat{\gamma}_2\ge \widehat{\gamma}_1$ (otherwise swap the roles of $\widehat{\gamma}_2$ and $ \widehat{\gamma}_1$) so in particular this inequality holds as well for every $x\in U_y$ replacing $\widehat{\gamma}_1$ with $\widehat{\gamma}_2$. By Proposition~\ref{prop:weak-Legendre}, ${\phi_+^*}'$ is invertible in $U_y$ and thus we have that $\frac{1}{\epsilon}(f+\widehat{\gamma}_1-c(\cdot,y)) = \frac{1}{\epsilon}(f+\widehat{\gamma}_2-c(\cdot,y))$ for $ U_y$-a.e. (if we want to be very precise, this would be with the restriction of $\mu$ to $U_y$) and this clearly shows that $\widehat{\gamma}_1 = \widehat{\gamma}_2$. In particular, this unique value is precisely $f^{(c,\epsilon,\phi)}(y)$. For simplicity and smoothness of the notation we will denote $f^{(c,\epsilon,\phi)}(y)=\gamma_y$.

We have that $\support((\widehat{\xi}_s)_-)\subset \{\frac{1}{\epsilon}(f+\gamma_y-c(\cdot,y))=\phi_+'(-\infty)=-\infty\}=\emptyset$ \citep[Corollary~3.6]{Borweinetal1993} and, in particular $(\widehat{\xi}_s)_-=0$, and that $\support((\widehat{\xi}_s)_+)\subset \{\frac{1}{\epsilon}(f+\gamma_y-c(\cdot,y))=\phi_+'(\infty)\}$. As $\phi$ is of Legendre type ${\phi_+^*}'$ is always nonnegative and increasing by Proposition~\ref{prop:weak-Legendre}. In particular $\widehat{\xi}$ is a probability measure. If $\phi'(\infty)=\infty$ or $\int_X  {\phi_+^*}'(\frac{1}{\epsilon}(f+\widehat{\gamma}_y-c(\cdot,y))) =1$ then we have no singular part but if $\int \frac{d\widehat{\xi}_c}{d\mu}\;d\mu<1$ then there must exists some $x\in X$ such that $f(x)+\gamma_y-c(x,y)=\epsilon\phi'(\infty)$. If we assume that $(ii)$ of this proposition holds and $f(x)+f^{(c,\epsilon,\phi)}(y)\le c(x,y)+\epsilon \phi'(\infty)$ for all $x\in X$ and $y\in Y$, we have that in this case $f^{(c,\epsilon,\phi)}(y)=\gamma_y = f^c(y)+\epsilon\phi'(\infty)$, in particular uniqueness holds even if the absolutely continuous part is 0.

Let us prove $(ii)$ now. First assume that we have not proved the uniqueness part of $(i)$ yet. Suppose by contradiction that $f(x_0)+\gamma_{y_0}> c(x_0,y_0)+\epsilon \phi'(\infty)$ for some $x_0\in X$ and $y_0\in Y$ where $\gamma_{y_0}=\hat{\gamma}_{y_0}$ is an optimal solution of the Dual Problem. Let us now define $U_{y_0}:=\{x\in X:f(x)+\gamma_{y_0}> c(x,y_0)+\epsilon \phi'(\infty)\}$ which by hypothesis is a non-empty open set. Now we use the assumption that $\support(\mu)=X$ to see that in this case $\frac{1}{\epsilon}\gamma_{y_0}-\int_X \phi_+^*(\frac{1}{\epsilon}(f+\gamma_{y_0}-c))\;d\mu$ equals

\[
\frac{1}{\epsilon}\gamma_{y_0}-\int_{X\setminus U_{y_0}} \phi_+^*\left(\frac{1}{\epsilon}(f+\gamma_{y_0}-c)\right)\;d\mu-\int_{ U_{y_0}} \phi_+^*\left(\frac{1}{\epsilon}(f+\gamma_{y_0}-c)\right)\;d\mu.
\]
But $\mu(U_{y_0})>0$ and $\phi_+^*$ equals $\infty$ in that set. As the other part of the integral is always bounded from below, we get that $\frac{1}{\epsilon}\gamma_{y_0}-\int_X \phi_+^*(\frac{1}{\epsilon}(f+\gamma_{y_0}-c))\;d\mu = -\infty$ but this is impossible as we know that the Dual Problem has a solution strictly larger than $-\infty$. Hence, the uniqueness part of $(i)$ holds and therefore as $\gamma_{y}$ is by definition $f^{(c,\epsilon,\phi)}(y)$ we conclude $(ii)$.

We continue now by proving $(iii)$. Let us define $\beta^y(\gamma):= \int_X {\phi_+^*}'(\frac{1}{\epsilon}(f+\gamma-c(\cdot,y)))\;d\mu $ for any $\gamma\in \mb{R}$. As we saw before, either $\beta^y(\gamma_y)=1$ or $\beta^y(\gamma_y)<1$ and $\gamma_y = f^c(y)+\epsilon\phi'(\infty)$. To prove that $f^{(c,\epsilon,\phi)}$ defined pointwise by $\gamma_y$ is Lipschitz, given $y,y'\in Y$ first suppose that $\beta^{y'}(\gamma_{y'}) \ge \beta^y(\gamma_y)$. Then, as ${\phi_+^*}'$ is an increasing function so is $\beta^y(\gamma)$ as a function of $\gamma$. Thus 

\begin{align*}
\beta^{y'}(\gamma_{y'}) &\ge \beta^y(\gamma_y)  \ge \int_X {\phi_+^*}'\left(\frac{1}{\epsilon}(f(x)+\gamma-c(x,y')-\|c\|_Ld_Y(y,y'))\right)\;d\mu \\&= \beta^{y'}(\gamma_y-\|c\|_Ld_Y(y,y')).
\end{align*}
Hence, $\gamma_y-\gamma_{y'}\le \|c\|_Ld_Y(y,y')$.

If $\beta^{y'}(\gamma_{y'}) = \beta^y(\gamma_y)=1$ then we are done, as we can repeat the above argument switching the roles of $y$ and $y'$. Similarly, if $\beta^{y'}(\gamma_{y'})<1$ and $\beta^y(\gamma_y)<1$, using the fact that in this case the transform is just a translate of the regular $c$-transform we get the result. The only case left is what happens if (say) $\beta^{y'}(\gamma_{y'})=1$ and $\beta^y(\gamma_y)<1$. By the previous argument we already know that $\gamma_y-\gamma_{y'}\le \|c\|_Ld_Y(y,y')$. For the other inequality, note that as $f^c(y')-f^c(y)\le \|c\|_Ld_Y(y,y')$ but we also know that $\gamma_y = f^c(y)+\epsilon\phi'(\infty)$ and $\gamma_{y'}\le f^c(y')+\epsilon\phi'(\infty)$. Plugging this into the previous inequality the result follows.

Let us now prove $(iv)$. Again we have to divide into two cases. Given $y\in Y$, if $\beta^y(\gamma_y)=1$ using that $\phi(1)=0$ we know that ${\phi_+^*}'(0)=1$ and as ${\phi_+^*}'$ is increasing and nonnegative we have that $\sup_{x\in X}\{\frac{1}{\epsilon}(f+\gamma_y-c)\}\ge 0$ (as otherwise $\beta^y(\gamma_y)$ would be strictly smaller than 1). From this it is easy to see that $\gamma_y\ge -\|f\|_{\infty}-\|c\|_{\infty}$. An analogous argument shows that $\gamma_y\le \|f\|_{\infty}+\|c\|_{\infty}$. If  $\gamma_y=f^c(y)+\epsilon\phi'(\infty)$ then we use the bound $\|f^c\|_{\infty}\le \|f\|_{\infty}+\|c\|_{\infty}$ and the result follows.

Part $(v)$ follows directly from the definitions.

To prove the last part, let $f_1,f_2\in \Lip(X)$. We want to prove that if $\|f_1-f_2\|_{\infty}\le L$ then $\|f_1^{(c,\epsilon,\phi)}-f_2^{(c,\epsilon,\phi)}\|\le L$. We have to consider 3 different cases. Fix any $y\in Y$. First assume that both $f_i^{(c,\epsilon,\phi)}(y)$ for $i=1,2$ are calculated by the formula $f_i^{(c,\epsilon,\phi)}(y)=\min_{x \in X}\{\epsilon \phi'(\infty) + c(x,y)-f_i(x)\}$. If we use that $-f_1(x)\ge -f_2(x)-L$ for all $x\in X$ we have that $f_1^{(c,\epsilon,\phi)}(y) \ge \min_{x \in X}\{\epsilon \phi'(\infty) + c(x,y)-f_2(x)-L\} = f_2^{(c,\epsilon,\phi)}(y)-L$. Using the inequality $-f_1(x)\le -f_2(x)+L$ we obtain the converse inequality and we are done in this case.

Next, assume that for $i=1,2$, the value of $f_i^{(c,\epsilon,\phi)}(y)$ is given implicitly as the unique value such that $\int_X {\phi_+^*}' \circ \frac{1}{\epsilon}(f_i + f_i^{(c,\epsilon,\phi)}(y) - c(\cdot,y)) d\mu = 1$. Then we would have that for example $1 =\int_X {\phi_+^*}' \circ \frac{1}{\epsilon}(f_1 + f_1^{(c,\epsilon,\phi)}(y) - c(\cdot,y)) d\mu \le \int_X {\phi_+^*}' \circ \frac{1}{\epsilon}(f_2+L + f_1^{(c,\epsilon,\phi)}(y) - c(\cdot,y)) d\mu$. As the function $\beta^y$ as we defined it before is increasing, we must have\footnote{Note that in principle these integrals are only well-defined if the argument of ${\phi_+^*}'$ is less or equal than $\phi'(\infty)$. However, we can assume that the value of ${\phi_+^*}'$ is $\infty$ for values larger than $\phi'(\infty)$ as this will be consistent with the definition of the $(c,\epsilon,\phi)$-transform given in $(i)$.} that by definition $f_2^{(c,\epsilon,\phi)}(y) \le f_1^{(c,\epsilon,\phi)}(y)+L$. By an analogous argument but using that $f_1(x)\ge f_2(x)-L$ for all $x\in X$ we have the opposite inequality.

Finally, in the mixed case when (say) $f_1^{(c,\epsilon,\phi)}(y)$ is given explicitly and $f_2^{(c,\epsilon,\phi)}(y)$ is implicit, we have to combine the previous arguments to conclude our result. On the one hand, $f_1^{(c,\epsilon,\phi)}(y) \ge \min_{x \in X}\{\epsilon \phi'(\infty) + c(x,y)-f_2(x)-L\}\ge f_2^{(c,\epsilon,\phi)}(y)-L$ (as we always have the inequality $f_2^{(c,\epsilon,\phi)}(y) \le f_2^c(y)+\epsilon\phi'(\infty)$. For the other inequality note that $1\ge  \int_X {\phi_+^*}' \circ \frac{1}{\epsilon}(f_1+ f_1^{(c,\epsilon,\phi)}(y) - c(\cdot,y)) d\mu$ always (because $\widehat{\xi}$ is always a probability measure). Then we use the inequality $f_1(x)\ge f_2(x)-L$ which give us at the end that $1\ge \int_X {\phi_+^*}' \circ \frac{1}{\epsilon}(f_2-L+ f_1^{(c,\epsilon,\phi)}(y) - c(\cdot,y)) d\mu$. Similarly as before, this implies that $f_2^{(c,\epsilon,\phi)}(y)\ge f_1^{(c,\epsilon,\phi)}(y)-L$.
\end{proof}

\begin{remark} Note that in some cases the $(c,\epsilon,\phi)$-transform collapses to ``almost'' the $c$-transform. \end{remark}

\begin{example}\label{example} Consider the following example. Let $X=Y=[0,1]$ with the measure $d\mu=2x\;dx$ (where $dx$ is the usual Lebesgue measure). Let also $\epsilon=1$, the cost function $c(x,y)=3x-1$ and $f(x)=0$. Let also $D_\phi$ be the reverse Kullback-Leibler divergence (see Section~\ref{div:reverse-KL}). Then $f^{(c,\epsilon,\phi)}(y)=f^c(y)+\epsilon\phi'(\infty)=0$ for all $y\in Y$. To prove this, note that we just have to compute $\int \phi_+^*(\frac{1}{\epsilon}(f+\gamma-c)d\mu = \int_0^1 \frac{2x}{3x-\gamma}dx = \frac{2}{3}+\frac{2}{9}\gamma(\log(3-\gamma)-\log(-\gamma))$. From here it is easy to check that there is no $\gamma \le 0$ such that the previous integral equals 1. Thus, the $(c,\epsilon,\phi)$-transform of $f$ collapses to $f^c(y)+\epsilon\phi'(\infty)$ for all $y\in Y$. \end{example}

\begin{theorem}\label{thm:strong-dual-primal}
Let $\mu \in P(X)$ and $\nu \in P(Y)$ be probability measures of full support on compact metric spaces $(X,d_X)$ and $(Y,d_Y)$. Let $c \in \Lip(X \times Y)$, $0 < \epsilon \in \R$ be a regularization coefficient  and $\phi : \R \to \overline{\R}$ a proper, convex and lower semicontinuous function  of Legendre type. Then one has
\begin{align*}
&\min_{\pi \in \Pi(\mu,\nu)}\{ \langle \pi, c \rangle + \epsilon D_{\phi}(\pi \Vert \mu \otimes \nu) \} \\
&= \max_{\substack{f\in\Lip(X), g\in \Lip(Y) \\ f \oplus g \leq c + \epsilon \phi'(\infty)}}\{ \langle \mu \otimes \nu, f \oplus g \rangle - \epsilon \langle \mu \otimes \nu, \phi_+^* \circ \frac{1}{\epsilon}(f \oplus g - c)\rangle \} \\
&= \max_{f \in \Lip(X)}\{ \langle \mu \otimes \nu, f \oplus f^{(c,\epsilon,\phi)} \rangle - \epsilon \langle \mu \otimes \nu, \phi_+^* \circ \frac{1}{\epsilon}(f \oplus f^{(c,\epsilon,\phi)} - c)\rangle \} \\
&= \max_{g \in \Lip(Y)}\{ \langle \mu \otimes \nu, g^{(c,\epsilon,\phi)} \oplus g \rangle - \epsilon \langle \mu \otimes \nu, \phi_+^* \circ \frac{1}{\epsilon}(g^{(c,\epsilon,\phi)} \oplus g - c)\rangle \},
\end{align*}
and $\pi_* \in \Pi(\mu,\nu)$ is optimal in the primal problem if and only if there exists $(f_*,g_*) \in \Lip(X) \times \Lip(Y)$ such that
\begin{equation}
\frac{1}{\epsilon}(f_* \oplus g_* - c) \leq \phi'(\infty),
\end{equation}
\begin{equation}
\frac{d\pi_c}{d\mu \otimes \nu} = {\phi_+^*}' \circ \frac{1}{\epsilon}(f_* \oplus g_* - c)
\end{equation}
and
\begin{equation}
\support(\pi_s) \subset \{ (x,y) \in X \times Y : \frac{1}{\epsilon}(f_*(x) + g_*(y) - c(x,y)) = \phi'(\infty) \}
\end{equation}
hold. In this case, $(f_*,g_*)$ are a pair of optimal potentials in the dual problem.
\end{theorem}

\begin{proof}
Since $\inf_{x \in X} \{ f(x) \} = -f^*(0)$ for any proper, convex and lower semicontinuous function $f$, by Proposition~\ref{primal_map_conjugate} one has
\begin{multline*}
\inf_{\pi \in \mathcal{M}(X \times Y)}\{ \langle \pi,c\rangle + \epsilon I_{\phi_+,\mu \otimes \nu}(\pi) + \iota_{\{(\mu,\nu)\}}(\pi(\cdot \times Y), \pi(X \times \cdot)) \} \\
= -\inf_{(f,g) \in \Lip(X) \times \Lip(Y)}\left\{ \epsilon I_{\phi_+,\mu \otimes \nu}^*\left(\frac{1}{\epsilon} (- c - f \oplus g)\right) + \langle \mu, f \rangle + \langle \nu, g \rangle \right\},
\end{multline*}
or equivalently
\begin{equation*}
\sup_{(f,g) \in \Lip(X) \times \Lip(Y)}\left\{ \langle \mu \otimes \nu, f \oplus g \rangle - \epsilon I_{\phi_+,\mu \otimes \nu}^*\left(\frac{1}{\epsilon} (f \oplus g - c)\right) \right\}.
\end{equation*}
Since $I_{\phi,\mu \otimes \nu}^*(\varphi) = \infty$ unless $\varphi(X) \subseteq [\phi'(-\infty),\phi'(\infty)]$ \citep[Proposition~7]{Terjek2021}, $\phi_+'(-\infty)=-\infty$ and $\phi_+'(\infty) = \phi'(\infty)$, one has the constraint $\frac{1}{\epsilon} (f \oplus g - c) \leq \phi'(\infty)$, leading to
\begin{equation*}
\sup_{ f \oplus g \leq c + \epsilon \phi'(\infty)}\left\{ \langle \mu \otimes \nu, f \oplus g \rangle - \epsilon \left\langle \mu \otimes \nu, \phi_+^* \circ \frac{1}{\epsilon} (f \oplus g - c) \right\rangle \right\}.
\end{equation*}

\noindent By definition of the $(c,\epsilon,\phi)$-transform, it is clear that 
\begin{equation}\label{eq:estimate}
    g - \epsilon \int  \phi_+^* \circ \frac{1}{\epsilon} (f \oplus g - c) \;d\mu \le f^{(c,\epsilon,\phi)} - \epsilon \int  \phi_+^* \circ \frac{1}{\epsilon} (f \oplus f^{(c,\epsilon,\phi)} - c) \;d\mu
\end{equation}
for every $y\in Y$. Thus we can always replace $g$ by $f^{(c,\epsilon,\phi)}$. A similar argument shows that we can always replace $f$ by $g^{(c,\epsilon,\phi)}$.

Let us now check that both the supremum and the infimum are attained. Let us start with the infimum. Let $\pi_n\in \Pi(\mu,\nu)$ be such that $\langle \pi_n, c \rangle + \epsilon I_{\phi, \mu \otimes \nu}(\pi_n)\to \inf_{\pi \in \Pi(\mu,\nu)}\{ \langle \pi, c \rangle + \epsilon I_{\phi, \mu \otimes \nu}(\pi) \}$ as $n\to \infty$. As the set of probability measures is a compact set in the Hanin norm \citep[Theorem~8.4.25(3),Theorem~8.5.7]{Cobzasetal2019} and any coupling is a probability measure, we can assume that there is a convergent subsequence (that abusing the notation we denote by $\pi_n$) such that $\pi_n\to \pi^*$ in the Hanin norm. Moreover, as $p_1^*$ and $p_2^*$ are continuous functions we know that $\pi^*\in \Pi(\mu,\nu)$. And finally note that as the function $\langle \cdot, c \rangle + \epsilon I_{\phi, \mu \otimes \nu}(\cdot)$ is lower semicontinuous we have that $\langle \pi^*, c \rangle + \epsilon I_{\phi, \mu \otimes \nu}(\pi^*)\le \lim_{n\to\infty} \langle \pi_n, c \rangle + \epsilon I_{\phi, \mu \otimes \nu}(\pi_n) = \inf_{\pi \in \Pi(\mu,\nu)}\{ \langle \pi, c \rangle + \epsilon I_{\phi, \mu \otimes \nu}(\pi) \}$, so that the maximum is achieved by $\pi^*$.

As for the supremum, we want to prove that 
\[
S:=\sup_{ f \oplus g \leq c + \epsilon \phi'(\infty)}\left\{ \langle \mu \otimes \nu, f \oplus g \rangle - \epsilon \left\langle \mu \otimes \nu, \phi_+^* \circ \frac{1}{\epsilon} (f \oplus g - c) \right\rangle \right\}
\]
is attained for some pair of functions $(f,g)\in \Lip(X)\times \Lip(Y)$. Let $(f_n,g_n)\in \Lip(X)\times \Lip(Y)$ be a sequence of functions such that $f_n\oplus g_n\le c+\epsilon\phi'(\infty)$ and $|S-\langle \mu,f_n\rangle-\langle \nu,g_n\rangle+\epsilon \left\langle \mu \otimes \nu, \phi_+^* \circ \frac{1}{\epsilon} (f_n \oplus g_n - c) \rangle \right|\le 1/n$. First note that by \eqref{eq:estimate} we can replace $g_n$ by $f_n^{(c,\epsilon,\phi)}$ and we are still at most $1/n$ away from $S$.  As $Y$ is compact and metric, it has finite diameter, $\diam(Y)=\sup_{y,y'\in Y} d_Y(y,y')<\infty$. By $(iii)$ of Proposition~\ref{prop:properties-c-eps-phi-trans}, the Lipschitz constant of $f_n^{(c,\epsilon,\phi)}$ is bounded by $\|c\|_L$ for all $n\ge 0$. Moreover, note that we can replace the pair $(f_n,f_n^{(c,\epsilon,\phi)})$ by $(f_n+a,f_n^{(c,\epsilon,\phi)}-a)$ for any constant $a\in \mathbb{R}$. Thus, taking $a=f_n^{(c,\epsilon,\phi)}(y_0)$ for some $y_0\in Y$ we have that $f_n^{(c,\epsilon,\phi)}-f_n^{(c,\epsilon,\phi)}(y_0)$ is a function with Lipschitz constant at most $\|c\|_L$ and $|f_n^{(c,\epsilon,\phi)}-f_n^{(c,\epsilon,\phi)}(y_0)|\le \|c\|_Ld_Y(y,y') \le \|c\|_L\diam(Y)$.

Now, again we use \eqref{eq:estimate} and instead of the pair $(f_n+f_n^{(c,\epsilon,\phi)}(y_0),f_n^{(c,\epsilon,\phi)}-f_n^{(c,\epsilon,\phi)}(y_0))$ we take $((f_n^{(c,\epsilon,\phi)}-f_n^{(c,\epsilon,\phi)}(y_0))^{(c,\epsilon,\phi)},f_n^{(c,\epsilon,\phi)}-f_n^{(c,\epsilon,\phi)}(y_0))$. By Proposition~\ref{prop:properties-c-eps-phi-trans} we know that the Lipschitz constant of $(f_n^{(c,\epsilon,\phi)}-f_n^{(c,\epsilon,\phi)}(y_0))^{(c,\epsilon,\phi)}$ is at most $\|c\|_L$ and that $\|(f_n^{(c,\epsilon,\phi)}-f_n^{(c,\epsilon,\phi)}(y_0))^{(c,\epsilon,\phi)}\|_{\infty} \le \|c\|_{\infty}+\|c\|_L\diam(Y)$. Thus, if we denote by $h_n:=f_n^{(c,\epsilon,\phi)}-f_n^{(c,\epsilon,\phi)}(y_0)$ we have that $|S-\langle \mu,h_n^{(c,\epsilon,\phi)}\rangle-\langle \nu,h_n\rangle|\le 1/n$ and $\|h_n\|_{\max} = \max\{\|h_n\|_{\infty},\|h_n\|_L\} \le (\diam(Y)+1)\|c\|_L$. Similarly we get that $\|h^{(c,\epsilon,\phi)}_n\|_{\max} = \max\{\|h^{(c,\epsilon,\phi)}_n\|_{\infty},\|h^{(c,\epsilon,\phi)}_n\|_L\} \le (\diam(Y)+1)\|c\|_L+\|c\|_{\infty}$. The key fact now is that these constants do not depend on $n$, and therefore, as $(\Lip(X),\|\cdot\|_{\max})$ is the dual of a normed space (namely $(\mathcal{M}(X),\|\cdot\|_H)$), by the Banach-Alaoglu theorem we know that the unit ball is compact in the weak* topology. Thus, we can assume (passing to a subsequence if necessary) that $h_n\to h$ and $h_n^{(c,\epsilon,\phi)}\to h'$ in the weak* topology. Using the fact that $I^*_{\phi_+,\mu\otimes\nu}$ is weak* lower semicontinuous \citep[Theorem~2.3.1]{Zalinescu2002} this implies that 
\[
S= \langle \mu \otimes \nu, h' \oplus h \rangle - \epsilon \left\langle \mu \otimes \nu, \phi_+^* \circ \frac{1}{\epsilon} (h' \oplus h - c) \right\rangle
\]
and similarly changing $h'$ by $h^{(c,\epsilon,\phi)}$ or $h$ by $h'^{(c,\epsilon,\phi)}$. Note that $h'(x)+h(y)\le c(x,y)+\epsilon\phi'(\infty)$ for all $(x,y)\in X\times Y$ as otherwise the right hand side of the previous equation will be $-\infty$.

If $\pi$ is optimal and $(f,g)$ are optimal potentials then 
\[
\langle \pi,c\rangle+\epsilon I_{\phi_+,\mu\otimes \nu}(\pi) = \langle \pi,f\oplus g\rangle-\epsilon I^*_{\phi_+,\mu\otimes \nu}\left(\frac{1}{\epsilon}(f\oplus g-c)\right),
\]
or equivalently
\[
\left\langle \pi,\frac{1}{\epsilon}(f\oplus g-c)\right\rangle = I_{\phi_+,\mu\otimes \nu}(\pi) +  I^*_{\phi_+,\mu\otimes \nu}\left(\frac{1}{\epsilon}(f\oplus g-c)\right).
\]
The optimality conditions then follow \citet[Theorem~2.10]{Borweinetal1993}. \end{proof}

We can say even a little more about the structure of the optimal potentials and coupling. A set $C \subset X \times Y$ is called $c$-cyclically monotone \citep[Definition~5.1]{Villani2008} if for any subset $\{ (x_1,y_1),\dots,(x_n,y_n) \} \subset C$ for $n \in \N$, one has
\begin{equation}
\sum_{i = 1}^n c(x_i,y_i) \leq \sum_{i = 1}^{n-1} c(x_i,y_{i+1}) + c(x_n,y_1).
\end{equation}

\begin{proposition}\label{prop:c-mon}
The $(c,\epsilon,\phi)$-subdifferential of $f \in\Lip(X)$ defined as
\begin{equation}
\partial_{(c,\epsilon,\phi)} f = \{ (x,y) \in X \times Y : f(x)+f^{(c,\epsilon,\phi)}(y) = c(x,y)+\epsilon \phi'(\infty) \},
\end{equation}
and the $(c,\epsilon,\phi)$-subdifferential of $g \in\Lip(Y)$ defined as
\begin{equation}
\partial_{(c,\epsilon,\phi)} g = \{ (x,y) \in X \times Y : g(y)+g^{(c,\epsilon,\phi)}(x) = c(x,y)+\epsilon \phi'(\infty) \}
\end{equation}
are both closed, $c$-cyclically monotone sets.
\end{proposition}
\begin{proof}
If $\phi'(\infty)=\infty$, then $\partial_{(c,\epsilon,\phi)} f = \partial_{(c,\epsilon,\phi)} g = \emptyset$, so the statement is vacuously true. Now assume that $\phi'(\infty) \in \R$. Being the level sets of Lipschitz continuous functions implies that both sets are closed. Let $\{ (x_1,y_1),\dots,(x_n,y_n) \} \subset \partial_{(c,\epsilon,\phi)} f$, so that one has
\begin{equation}
\sum_{i = 1}^n c(x_i,y_i) = \sum_{i = 1}^n [f(x_i)+f^{(c,\epsilon,\phi)}(y_i)-\epsilon \phi'(\infty)].
\end{equation}
On the other hand, one always has $c(x_i,y_j) + \epsilon \phi'(\infty) \geq f(x_i) + f^{(c,\epsilon,\phi)}(y_j)$, implying that
\begin{multline}
\sum_{i = 1}^{n-1} c(x_i,y_{i+1}) + c(x_n,y_1) \geq \sum_{i = 1}^{n-1}[f(x_i)+f^{(c,\epsilon,\phi)}(y_{i+1})-\epsilon \phi'(\infty)]+f(x_n)+f^{(c,\epsilon,\phi)}(y_1)-\epsilon \phi'(\infty)\\
=\sum_{i = 1}^n [f(x_i)+f^{(c,\epsilon,\phi)}(y_i)-\epsilon \phi'(\infty)].
\end{multline}
The last two equations imply the proposition for $\partial_{(c,\epsilon,\phi)} f$, and a symmetric argument clearly works for $\partial_{(c,\epsilon,\phi)} g$.
\end{proof}

\begin{proposition}\label{prop:properties-optimals} Let $\mu \in P(X)$ and $\nu \in P(Y)$ be probability measures of full support on compact metric spaces $(X,d_X)$ and $(Y,d_Y)$. Let $c \in \Lip(X \times Y)$, $0 < \epsilon \in \R$ be a regularization coefficient  and $\phi : \R \to \overline{\R}$ a proper, convex and lower semicontinuous function  of Legendre type. Let $\pi' \in \Pi(\mu,\nu)$ be an optimal coupling for the primal problem. Let $\pi'_c$ be its absolutely continuous part with respect to $\mu\otimes\nu$ and $\pi'_s$ the singular part. Let also $(f',g') \in \Lip(X) \times \Lip(Y)$ be a pair of optimal potentials. Then $\frac{d\pi'_c}{d\mu\otimes \nu}$ is unique for any optimal coupling. If $(\tilde{f},\tilde{g})\in\Lip(X)\times\Lip(Y)$ are also optimal potentials then $f'\oplus g' = \tilde{f}\oplus\tilde{g}$ $\pi'_c$-a.e.. If ${\phi_+^*}'$ is invertible in $(-\infty,\phi'(\infty))$ then any optimal potential equals $(f'+a,g'-a)$ for some $a\in \mb{R}$. Finally, the support of $\pi_s$ lies in the intersection of the $(c,\epsilon,\phi)$-subdifferentials of all optimal dual variables.
\end{proposition}

\begin{proof} Let $\pi^1,\pi^2 \in \Pi(\mu,\nu)$ and $(f^1,g^1), (f^2,g^2) \in \Lip(X) \times \Lip(Y)$ be optimal primal and dual variables. If $g^j = {f^j}^{(c,\epsilon,\phi)}$ and $f^j = {g^j}^{(c,\epsilon,\phi)}$ would not hold for $j \in \{1,2\}$, one could replace $g^j$ with ${f^j}^{(c,\epsilon,\phi)}$ to increase the value of the dual problem, contradicting optimality of $(g^j,f^j)$. By optimality, one has
\begin{equation}
\int c d\pi^i + \epsilon I_{\phi_+,\mu \otimes \nu}(\pi^i)
= \int f^j \oplus g^j d\mu \otimes \nu - \epsilon I_{\phi_+,\mu \otimes \nu}^*\left( \frac{1}{\epsilon}(f^j \oplus g^j - c) \right)
\end{equation}
for $i,j \in \{1,2\}$. Since $\pi^i \in \Pi(\mu,\nu)$, one has $\int f^j \oplus g^j d\mu \otimes \nu = \int f^j \oplus g^j d\pi^i$, so we can rearrange as
\begin{equation}
I_{\phi_+,\mu \otimes \nu}(\pi^i) + I_{\phi_+,\mu \otimes \nu}^*\left( \frac{1}{\epsilon}(f^j \oplus g^j - c) \right)
= \int \frac{1}{\epsilon}(f^j \oplus g^j - c) d\pi^i.
\end{equation}
By \citet[Theorem~2.10]{Borweinetal1993}, since $\phi_+'(-\infty)=-\infty$, this holds if and only if
\begin{equation}
\frac{1}{\epsilon}(f^j \oplus g^j - c) \leq \phi'(\infty),
\end{equation}
\begin{equation}\label{eq:cont-part}
\frac{d\pi^i_c}{d\mu \otimes \nu} = {\phi_+^*}' \circ \frac{1}{\epsilon}(f^j \oplus g^j - c) \ \mu \otimes \nu\text{-a.e.}
\end{equation}
and
\begin{equation}\label{eq:singular-part}
\support(\pi^i_s) \subset \left\{ (x,y) \in X \times Y : \frac{1}{\epsilon}(f^j(x) + g^j(y) - c(x,y)) = \phi'(\infty) \right\},
\end{equation}
where one has $\left\{ (x,y) \in X \times Y : \frac{1}{\epsilon}(f^j(x) + g^j(y) - c(x,y)) = \phi'(\infty) \right\} = \partial_{(c,\epsilon,\phi)}f^j = \partial_{(c,\epsilon,\phi)}g^j$. As \eqref{eq:cont-part} holds for fixed $j$ and $i=1,2$ we have that the absolutely continuous part of any optimal coupling is unique. If we let $C:=\{(x,y)\in X\times Y: {\phi_+^*}'(\frac{1}{\epsilon}(f^1\oplus g^1-c))={\phi_+^*}'(\frac{1}{\epsilon}(f^2\oplus g^2-c))\}$ we know that $\mu\otimes \nu(C)=1$. As ${\phi_+^*}'\ge 0$ and it is invertible in the points where ${\phi_+^*}'> 0$ by Proposition~\ref{prop:weak-Legendre} if $P:=\{(x,y)\in X\times Y: {\phi_+^*}'(\frac{1}{\epsilon}(f^1\oplus g^1-c))>0\}$ we know that for all $(x,y)\in C\cap P$ we have $f^1\oplus g^1 =f^2\oplus g^2$. But clearly $\pi^1_c(C\cap P) = \pi^1_c(X\times Y)$.

Furthermore, if ${\phi_+^*}'$ is invertible in its domain from the same equation we deduce that $f^1\oplus g^1 = f^2\oplus g^2$ $\mu\otimes\nu$-a.e. As $\mu$ and $\nu$ have full support so do $\mu\otimes\nu$, and as $f^j$ and $g^j$ for $j=1,2$ are continuous functions, then $f^1\oplus g^1 = f^2\oplus g^2$ must hold for every $(x,y)\in X\times Y$\footnote{Here we use an standard continuity argument. If for some $(x_0,y_0)\in X\times Y$ we have $f^1(x_0)+ g^1(y_0) \not= f^2(x_0)+ g^2(y_0)$ then this will hold in an open neighborhood of $(x_0,y_0)$. But this will contradict the fact that $f^1\oplus g^1 = f^2\oplus g^2$ $\mu\otimes\nu$-a.e. as any open set has positive measure if the measure has full support.}. The last part of the proposition follows from \eqref{eq:singular-part} for a fixed $i$ and any $j=1,2$.\end{proof}

\begin{definition}[Sinkhorn operator] Let $X$ and $Y$ be compact metric spaces and $\mu\in P(X)$, $\nu\in P(Y)$ be Borel probability measures of full support. Let also $c \in \Lip(X \times Y)$, $0 < \epsilon \in \R$ be a regularization coefficient  and $\phi : \R \to \overline{\R}$ a proper, convex and lower semicontinuous function  of Legendre type. Fix any point $y_0\in Y$. Given a pair $(f,g)\in \Lip(X)\times \Lip(Y)$ we define the operator $\mc{F}^{(c,\epsilon,\phi)}:\Lip(X)\times \Lip(Y)\to \Lip(X)\times \Lip(Y)$ as
\[
\mc{F}^{(c,\epsilon,\phi)}(f,g):= ((f^{(c,\epsilon,\phi)}-f^{(c,\epsilon,\phi)}(y_0))^{(c,\epsilon,\phi)},f^{(c,\epsilon,\phi)}-f^{(c,\epsilon,\phi)}(y_0))
\]
\end{definition}

Technically this operator depends also on the point $y_0$ but as it is not very important which point it is, we decided not to put it in the definition of Sinkhorn iteration. This operator has the following very nice property:

\begin{proposition}\label{prop:sink-cpct-op} Let $X$ and $Y$ be compact metric spaces and $\mu\in P(X)$, $\nu\in P(Y)$ be Borel probability measures of full support. Let also $c \in \Lip(X \times Y)$, $0 < \epsilon \in \R$ be a regularization coefficient  and $\phi : \R \to \overline{\R}$ a proper, convex and lower semicontinuous function  of Legendre type. Fix any point $y_0\in Y$. Then for any $(f,g)\in \Lip(X)\times \Lip(Y)$ we have $\|\mc{F}^{(c,\epsilon,\phi)}(f,g)\|_{\max} \le K$ where $K$ depends only on the diameters of $X$ and $Y$, $\epsilon$ and on $\|c\|_{\max}$. Moreover, $\mc{F}^{(c,\epsilon,\phi)}$ is continuous with respect to the $\|\cdot\|_{\infty}$-norm\footnote{In the space $\Lip(X)\times \Lip(Y)$ we define the $\|\cdot\|_{\infty}$ norm as $\|(f,g)\|_{\infty}:=\max(\|f\|_{\infty},\|g\|_{\infty})$ for any $(f,g)\in\Lip(X)\times\Lip(Y)$.}.\end{proposition}

\begin{proof} By $(iii)$ of Proposition~\ref{prop:properties-c-eps-phi-trans} we have that the Lipschitz constants of $f^{(c,\epsilon,\phi)}$ and $(f^{(c,\epsilon,\phi)}-f^{(c,\epsilon,\phi)}(y_0))^{(c,\epsilon,\phi)}$ are uniformly bounded by $\|c\|_L$. As clearly $f^{(c,\epsilon,\phi)}-f^{(c,\epsilon,\phi)}(y_0)$ is a function that attains the value 0 we have that $\|f^{(c,\epsilon,\phi)}-f^{(c,\epsilon,\phi)}(y_0)\|_{\infty} \le \|f^{(c,\epsilon,\phi)}\|_L\sup_{y,y'\in Y}d_Y(y,y') \le \|c\|_L\diam(Y)$ where the diameter of $Y$ is finite because $Y$ is compact. Using $(iv)$ of Proposition~\ref{prop:properties-c-eps-phi-trans} we have that $\|(f^{(c,\epsilon,\phi)}-f^{(c,\epsilon,\phi)}(y_0))^{(c,\epsilon,\phi)}\|_{\infty} \le \|(f^{(c,\epsilon,\phi)}-f^{(c,\epsilon,\phi)}(y_0))\|_{\infty}+\|c\|_{\infty}+\epsilon\phi'(\infty)\chi_{\R}(\phi'(\infty))$ (where the last summand vanishes if $\phi'(\infty)=\infty$). The last part of the proposition follows easily from $(vi)$ of Proposition~\ref{prop:properties-c-eps-phi-trans}.\end{proof}

\begin{definition}[Good triple] Let $X$ be a compact metric space and $\mu$ a Borel probability measure on $X$. Let $\phi$ be a proper, convex and lower semicontinuous function of Legendre type and suppose that $\phi'(\infty)<\infty$. Let also $C>0$ be a constant. We say that $(X,\mu,\phi)$ is a \emph{good triple} with respect to $C$ if for all $x_0\in X$
\[
\lim_{\delta\downarrow 0}\int_X {\phi_+^*}'(\phi'(\infty)-Cd(x_0,x)-\delta)\;d\mu(x) > 1.
\]
\end{definition}

As we said in the main body of the paper, this condition is the one which ultimately will allow us to prevent the $(c,\epsilon,\phi)$-transform to collapse to the $c$-transform plus $\epsilon\phi'(\infty)$. More specifically, in the next proposition we will see how $\max_{x\in X,y\in Y}\{\frac{1}{\epsilon}(f(x)+f^{(c,\epsilon,\phi)}(y)-c(x,y))\}$ is separated from the critical value $\phi'(\infty)$ assuming this condition.

\begin{proposition}\label{prop:unif_separation} Let $X$ and $Y$ be compact metric spaces and $\mu\in P(X)$, $\nu\in P(Y)$ be Borel probability measures of full support. Let also $c \in \Lip(X \times Y)$, $0 < \epsilon \in \R$ be a regularization coefficient  and $\phi : \R \to \overline{\R}$ a proper, convex and lower semicontinuous function  of Legendre type. Suppose that $(X,\mu,\phi)$ is a good triple with respect to $2\|c\|_L/\epsilon$. Then for any $f\in \Lip(X)$ with $\|f\|_L\le \|c\|_L$ we have that $\max_{x\in X,y\in Y}\{\frac{1}{\epsilon}(f(x)+f^{(c,\epsilon,\phi)}(y)-c(x,y))\} \le \phi'(\infty)-\tau$ for some positive constant $\tau>0$.\end{proposition}

\begin{proof} Fix any $y\in Y$ and let $x_y\in X$ be such that the maximum of $\{\frac{1}{\epsilon}(f(x)+f^{(c,\epsilon,\phi)}(y)-c(x,y))\}$ with respect to $x\in X$ is attained at $x_y$ (it always exists because $X$ is compact and the functional continuous). Then
\[
\left|\frac{1}{\epsilon}(f(x)+f^{(c,\epsilon,\phi)}(y)-c(x,y))-\frac{1}{\epsilon}(f(x_y)+f^{(c,\epsilon,\phi)}(y)-c(x_y,y))\right|\le \frac{2\|c\|_L}{\epsilon}d_X(x,x_y).
\]
As ${\phi_+^*}'$ is increasing we have that

\begin{multline*}
    \int_X {\phi_+^*}'\left(\frac{1}{\epsilon}(f(x)+f^{(c,\epsilon,\phi)}(y)-c(x,y))\right)d\mu \ge\\ \int_X {\phi_+^*}'\left(\frac{1}{\epsilon}(f(x_y)+f^{(c,\epsilon,\phi)}(y)-c(x_y,y))-\frac{2\|c\|_L}{\epsilon}d_X(x,x_y)\right)d\mu.
\end{multline*}
Relabeling $\frac{1}{\epsilon}(f(x_y)+f^{(c,\epsilon,\phi)}(y)-c(x_y,y))$ as $\phi'(\infty)-\delta$ we see that this is a contradiction if $\delta$ is too small because the left hand side has to integrate to a value at most 1 by (the proof of) Proposition~\ref{prop:properties-c-eps-phi-trans}. By the definition of good triple we see that this is independent from the point $y\in Y$ so it holds for all of them. \end{proof}

Let us now state our main final result, which we will prove in several steps:

\begin{theorem}\label{thm:main-conv} Let $X$ and $Y$ be compact metric spaces and $\mu\in P(X)$, $\nu\in P(Y)$ be Borel probability measures of full support. Let also $c \in \Lip(X \times Y)$, $0 < \epsilon \in \R$ be a regularization coefficient  and $\phi : \R \to \overline{\R}$ a proper, convex and lower semicontinuous function  of Legendre type. Suppose that $(X,\mu,\phi)$ and $(Y,\nu,\phi)$ are a good triples with respect to $2\|c\|_L/\epsilon$. Take any pair $(f,g)\in \Lip(X)\times \Lip(Y)$ and define inductively $(f_0,g_0):=(f,g)$ and $(f_n,g_n):=\mc{F}^{(c,\epsilon,\phi)}(f_{n-1},g_{n-1})$ for $n\ge 1$. Let us also define the dual functional for any pair of functions $(f,g)\in \Lip(X)\times \Lip(Y)$:
\[
D_{\epsilon}(f,g):= \langle \mu,f\rangle+\langle \nu,g\rangle-\epsilon I^*_{\phi_+,\mu\otimes\nu}\left( \frac{1}{\epsilon}(f\oplus g-c)\right).
\]
The optimal primal problem $\OT_{\epsilon}(\mu,\nu)$ is defined as
\[
\OT_{\epsilon}(\mu,\nu):=\inf_{\pi\in \Pi(\mu,\nu)}\{\langle \pi,c\rangle+\epsilon I_{\phi_+,\mu\otimes\nu}(\pi)\}.
\]
Then $D_{\epsilon}(f_n,g_n)\to \OT_{\epsilon}(\mu,\nu)$ as $n\to \infty$. Also, there exists a unique optimal coupling $\tilde{\pi}$ that attains the infimum in $\OT_{\epsilon}(\mu,\nu)$ and if $(\tilde{f},\tilde{g})$ are optimal potentials for the dual problem we have that $f_n\oplus g_n\to \tilde{f}\oplus \tilde{g}$ in $L^{\infty}(\pi)$. \end{theorem}

\begin{proof} By Proposition~\ref{prop:sink-cpct-op} as soon as $n\ge1$ we have that all the functions $(f_n,g_n)$ will have $\|\cdot\|_{\max}$ norm bounded uniformly in terms of $c,\epsilon$ and the diameters of $X$ and $Y$. Therefore we will assume from now on that all the functions in this sequence have this property. We can apply the Arzel\`a-Ascoli theorem \citep[Theorem~8.4.11]{Cobzasetal2019} (as the $\sumnorm$ norm defines the same topology as the $\max$ norm, clearly $\|\cdot\|_{\max}\le \|\cdot\|_{\sumnorm}\le 2\|\cdot\|_{\max}$). Thus we have that $(f_{n_k},g_{n_k})\to (\tilde{f},\tilde{g})$ as $k\to \infty$ in the $\|\cdot\|_{\infty}$ norm for some subsequence $n_k$.

In particular, as all the elements $f_n$ and $g_n$ have Lipschitz norm bounded by $\|c\|_L$, so do $\tilde{f}$ and $\tilde{g}$. The Sinkhorn operator is continuous with the $\|\cdot\|_{\infty}$-norm by Proposition~\ref{prop:sink-cpct-op}. By Proposition~\ref{prop:unif_separation} and the definition of the pair $(f_n,g_n)$ we know that $\frac{1}{\epsilon}(f_n\oplus g_n-c)$ has its image in $(-\infty,\phi'(\infty)-\tau]$ and therefore \citep[Theorem~2.7]{Borweinetal1993} the operator $D_{\epsilon}(\cdot,\cdot)$ is continuous in the set where $(f_n,g_n)$ lives.

Thus, we have that $D_{\epsilon}(f_{n_k},g_{n_k})\to D_{\epsilon}(\tilde{f},\tilde{g})$ and $D_{\epsilon}(\mc{F}^{(c,\epsilon,\phi)}(f_{n_k},g_{n_k}))\to D_{\epsilon}(\mc{F}^{(c,\epsilon,\phi)}(\tilde{f},\tilde{g})))$. Furthermore, by definition of the sequence $(f_n,g_n)$ we have that $D_{\epsilon}(f_{n_k},g_{n_k})) \le D_{\epsilon}(\mc{F}^{(c,\epsilon,\phi)}(f_{n_k},g_{n_k}))\le D_{\epsilon}(f_{n_{k+1}},g_{n_{k+1}}))$. Hence, $D_{\epsilon}(\tilde{f},\tilde{g}) = D_{\epsilon}(\mc{F}^{(c,\epsilon,\phi)}(\tilde{f},\tilde{g}))$. In particular we have that $D_{\epsilon}(\tilde{f},\tilde{g})) \le D_{\epsilon}(\tilde{f},{\tilde{f}}^{(c,\epsilon,\phi)}) =  D_{\epsilon}(\tilde{f}+a,{\tilde{f}}^{(c,\epsilon,\phi)}-a))\le D_{\epsilon}(({\tilde{f}}^{(c,\epsilon,\phi)}-a)^{(c,\epsilon,\phi)},{\tilde{f}}^{(c,\epsilon,\phi)}-a))= D_{\epsilon}(\tilde{f},\tilde{g})$ for some constant $a\in \mb{R}$. Hence, all previous inequalities are equalities. By $(v)$ of Proposition~\ref{prop:properties-c-eps-phi-trans} we have that $D_{\epsilon}(\tilde{f},\tilde{g})) = D_{\epsilon}(\tilde{f},{\tilde{f}}^{(c,\epsilon,\phi)})$ and $D_{\epsilon}(\tilde{f},{\tilde{f}}^{(c,\epsilon,\phi)}))= D_{\epsilon}(({\tilde{f}}^{(c,\epsilon,\phi)})^{(c,\epsilon,\phi)},{\tilde{f}}^{(c,\epsilon,\phi)}))$. Now we need the following lemma:

\begin{lemma} With the same hypothesis as in Theorem~\ref{thm:main-conv} let $f\in\Lip(X)$ and $g\in \Lip(Y)$ be any functions. Suppose that $D_{\epsilon}(f,g)) = D_{\epsilon}(f,{f}^{(c,\epsilon,\phi)})$. Then $g = f^{(c,\epsilon,\phi)}$ for every $y\in Y$. \end{lemma}

\begin{proof}[Proof of lemma:] Assume that equality fails at some $y_0\in Y$. For any $h\in \Lip(Y)$ and $y\in Y$ let us define 
\[
H_h(y):=h(y)-\epsilon\int_X {\phi_+^*}(f+h(y)-c(\cdot,y))d\mu.
\]
By hypothesis we also have that $\int H_g(y)d\nu = \int H_{f^{(c,\epsilon,\phi)}}(y)d\nu$. Note that by definition of the $(c,\epsilon,\phi)$-transform we have that $H_g(y)\le H_{f^{(c,\epsilon,\phi)}}(y)$ for all $y\in Y$. Thus, those functions are equal $\nu$-a.e. If for some $y_0\in Y$ we have $g(y_0)\not=f^{(c,\epsilon,\phi)}(y_0)$ then $H_g(y_0) < H_{f^{(c,\epsilon,\phi)}}(y_0)$. Let us denote that positive difference as $e:=H_{f^{(c,\epsilon,\phi)}}(y_0)-H_g(y_0)$.

Now note first that $H_g(y)$ is upper semi-continuous in $y$ \citep[Theorem~2.3.1]{Zalinescu2002} (we apply this result for general Lipschitz functions in $\Lip(X)$ and then note that we are just specializing that result to the concrete family of functions $\frac{1}{\epsilon}(f(x)+g(y)-c(x,y))$ indexed by $y\in Y$). For the part of $H_{f^{(c,\epsilon,\phi)}}$ we need to prove continuity instead of just upper semicontinuity. To do so, recall that by Proposition~\ref{prop:unif_separation} we know that the functions $\frac{1}{\epsilon}(f(x)+f^{(c,\epsilon,\phi)}(y)-c(x,y))\in \Lip(X)$ (this is a family of functions indexed by $y\in Y$) have their image strictly inside the range $(-\infty,\phi'(\infty))$. Thus we have that for $y\in Y$, $H_{f^{(c,\epsilon,\phi)}}$ is continuous \citep[Theorem~2.7]{Borweinetal1993}.

Hence, by upper semicontinuity of $H_g(y)$ there exists some $\delta>0$ such that if $d_Y(y,y_0)<\delta$ then $H_g(y)<H_g(y_0)+e/3$. Similarly, by continuity of $H_{f^{(c,\epsilon,\phi)}}(y)$ there exists $\delta'>0$ such that if $d_Y(y,y_0)<\delta'$ then $H_{f^{(c,\epsilon,\phi)}}(y) \ge H_{f^{(c,\epsilon,\phi)}}(y_0)-e/3$. Therefore if $d_Y(y,y_0)<\min(\delta,\delta')$ then $H_g(y) < H_{f^{(c,\epsilon,\phi)}}(y)$ and this is a contradiction with the fact that $\nu$ has full support. In particular, we have found an open set $\{y\in Y: d_Y(y,y_0)<\min(\delta,\delta') \}$ (which has positive measure as $\nu$ has full support) such that $H_g(y) < H_{f^{(c,\epsilon,\phi)}}(y)$ and this is a contradiction with the fact that these two functions are equal $\nu$-a.e. \end{proof}

Note that under similar hypothesis an analogous argument shows that if $D_{\epsilon}(f,g) = D_{\epsilon}({g}^{(c,\epsilon,\phi)},g)$ then ${g}^{(c,\epsilon,\phi)}=f$.

This lemma implies that the potentials $(\tilde{f},\tilde{g})$ that we have found satisfy that they are $(c,\epsilon,\phi)$-transforms of each other. We claim that the measure 
\[\tilde{\pi}:={\phi_+^*}'(\tilde{f}\oplus \tilde{g}-c)\mu\otimes\nu\]
is an optimal solution to the primal problem, i.e. $D_{\epsilon}(\tilde{f},\tilde{g})= \OT_{\epsilon}(\mu,\nu)$. First, note that by Theorem~\ref{thm:strong-dual-primal} if we manage to prove that $\tilde{\pi}$ is in $\Pi(\mu,\nu)$ then we would be done (as this clearly satisfy all the remaining conditions to be optimal). As ${\phi_+^*}'$ is nonnegative so is $\tilde{\pi}$.

Let us see that $(p_2^*)(\tilde{\pi})=\nu$ (the other projection follows analogously). Given any $A\subset Y$ measurable we have that
\[
(p_2^*)(\tilde{\pi})(A) = \int_A\int_X {\phi_+^*}'(\tilde{f}(x) + \tilde{g}(y)-c(x,y))d\mu(x) d\nu(y) 
\]
\[= \int_A\int_X {\phi_+^*}'(\tilde{f}(x) + \tilde{f}^{(c,\epsilon,\phi)}(y)-c(x,y))d\mu(x) d\nu(y).
\]

By hypothesis, recall that $(X,\mu,\phi)$ is a good triple with respect to $2\|c\|_L/\epsilon$ and therefore by Proposition~\ref{prop:unif_separation} and $(i)$ of Proposition~\ref{prop:properties-c-eps-phi-trans} we know that for every $y\in Y$ we have $\int_X {\phi_+^*}'(\tilde{f} + \tilde{f}^{(c,\epsilon,\phi)}(y)-c(\cdot,y))d\mu = 1$. Thus, the integral above reduces to $\int_A d\nu = \nu(A)$ and the result follows.

This proves that for the subsequence $n_k$ we have that $D_{\epsilon}(f_{n_k},g_{n_k})\to \OT_{\epsilon}(\mu,\nu)$. However, it is easy to extend this result to the full sequence using the fact that $D_{\epsilon}(f_{n},g_{n})$ is an increasing sequence for all $n\ge 1$. Just note that given $\delta>0$ we know that there exists $K(\delta)$ such that if $k\ge K(\delta)$ we have that $|D_{\epsilon}(f_{n_k},g_{n_k})- \OT_{\epsilon}(\mu,\nu)|<\delta$. Thus, for all $n\ge n_{K(\delta)}$ we know that $|D_{\epsilon}(f_{n},g_{n})- \OT_{\epsilon}(\mu,\nu)|<\delta$ which proves convergence. 

By Proposition~\ref{prop:properties-optimals} we know that the absolutely continuous part of an optimal solution is unique. As in our case we know that this defines directly a probability measure, we know that the solution to the primal problem $\OT_{\epsilon}(\mu,\nu)$ is unique (the singular part must be just 0). To prove the last part of the theorem, let $(f',g')\in \Lip(X)\times \Lip(Y)$ be a pair of optimal potentials and suppose by contradiction that $f_n\oplus g_n \not\to f'\oplus g'$ in $L^{\infty}(\tilde{\pi})$ as $n\to \infty$. This means that there exists a subsequence $n_i$ and a positive $\delta>0$ such that $\|f_{n_i}\oplus g_{n_i} - f'\oplus g'\|_{L^{\infty}(\tilde{\pi})}\ge \delta$ for all $i\ge 1$. We can now repeat the same arguments as before using the Arzel\`a-Ascoli theorem to prove that there exists a subsubsequence $n_{i_j}$ for $j\ge 1$ such that $(f_{n_{i_j}},g_{n_{i_j}})\to (f'',g'')$ as $j\to\infty$ in the $\|\cdot\|_{\infty}$ norm. In particular, $f_{n_{i_j}}\oplus g_{n_{i_j}}\to f''\oplus g''$ in $L^{\infty}(\tilde{\pi})$. By Proposition~\ref{prop:properties-optimals} we have that $f'\oplus g'=f''\oplus g''$ $\pi$-a.e. which is clearly a contradiction. \end{proof}

\section{Functions related to $f$-divergences} \label{appendix_phis}

\subsection{Kullback-Leibler divergence}
\begin{equation}
\phi_+(x)= \begin{cases}
x \log(x) - x + 1 \text{ if $x \geq 0$,}\\
\infty \text{ otherwise.}
\end{cases}
\end{equation}
\begin{equation}
\partial \phi_+(x)= \begin{cases}
\left\{ \log(x) \right\} \text{ if $x > 0$,}\\
\emptyset \text{ otherwise.}
\end{cases}
\end{equation}
\begin{equation}
\phi'(\infty)=\infty.
\end{equation}
\begin{equation}
\phi_+^*(x)=e^x-1.
\end{equation}
\begin{equation}
{\phi_+^*}'(x)=e^x.
\end{equation}
\begin{equation}
{\phi_+^*}''(x)=e^x.
\end{equation}

\subsection{Reverse Kullback-Leibler divergence}\label{div:reverse-KL}
\begin{equation}
\phi_+(x)= \begin{cases}
x-1-\log(x) \text{ if $x \geq 0$,}\\
\infty \text{ otherwise.}
\end{cases}
\end{equation}
\begin{equation}
\partial \phi_+(x)= \begin{cases}
\left\{ \frac{x-1}{x} \right\} \text{ if $x > 0$,}\\
\emptyset \text{ otherwise.}
\end{cases}
\end{equation}
\begin{equation}
\phi'(\infty)=1.
\end{equation}
\begin{equation}
\phi_+^*(x)=\begin{cases}
-\log(1-x) \text{ if $x \leq 1$,}\\
\infty \text{ otherwise.}
\end{cases}
\end{equation}
\begin{equation}
{\phi_+^*}'(x)=\frac{1}{1-x}.
\end{equation}
\begin{equation}
{\phi_+^*}''(x)=\frac{1}{(1-x)^2}.
\end{equation}

\subsection{$\chi^2$ divergence}
\begin{equation}
\phi_+(x)= \begin{cases}
(x-1)^2 \text{ if $x \geq 0$,}\\
\infty \text{ otherwise.}
\end{cases}
\end{equation}
\begin{equation}
\partial \phi_+(x)= \begin{cases}
\left\{ 2x-2 \right\} \text{ if $x \geq 0$,}\\
\emptyset \text{ otherwise.}
\end{cases}
\end{equation}
\begin{equation}
\phi'(\infty)=\infty.
\end{equation}
\begin{equation}
\phi_+^*(x)=
\begin{cases}
\frac{1}{4}x^2+x \text{ if $x \geq -2$},\\
-1 \text{ otherwise.}
\end{cases}
\end{equation}
\begin{equation}
{\phi_+^*}'(x)=
\begin{cases}
\frac{1}{2}x+1 \text{ if $x \geq -2$},\\
0 \text{ otherwise.}
\end{cases}
\end{equation}
\begin{equation}
{\phi_+^*}''(x)=
\begin{cases}
\frac{1}{2} \text{ if $x \geq -2$},\\
0 \text{ otherwise.}
\end{cases}
\end{equation}

\subsection{Reverse $\chi^2$ divergence}
\begin{equation}
\phi_+(x)= \begin{cases}
\frac{1}{x}+x-2 \text{ if $x \geq 0$,}\\
\infty \text{ otherwise.}
\end{cases}
\end{equation}
\begin{equation}
\partial \phi_+(x)= \begin{cases}
\left\{ 1-\frac{1}{x^2} \right\} \text{ if $x > 0$,}\\
\emptyset \text{ otherwise.}
\end{cases}
\end{equation}
\begin{equation}
\phi'(\infty)=1.
\end{equation}
\begin{equation}
\phi_+^*(x)=\begin{cases}
2-2\sqrt{1-x} \text{ if $x \leq 1$,}\\
\infty \text{ otherwise.}
\end{cases}
\end{equation}
\begin{equation}
{\phi_+^*}'(x)=\frac{1}{\sqrt{1-x}}.
\end{equation}
\begin{equation}
{\phi_+^*}''(x)=\frac{1}{2\sqrt{1-x}^3}.
\end{equation}

\subsection{Squared Hellinger divergence}\label{sec:squared-hell}
\begin{equation}
\phi_+(x)= \begin{cases}
(\sqrt{x}-1)^2 \text{ if $x \geq 0$,}\\
\infty \text{ otherwise.}
\end{cases}
\end{equation}
\begin{equation}
\partial \phi_+(x)= \begin{cases}
\left\{ 1-\frac{1}{\sqrt{x}} \right\} \text{ if $x > 0$,}\\
\emptyset \text{ otherwise.}
\end{cases}
\end{equation}
\begin{equation}
\phi'(\infty)=1.
\end{equation}
\begin{equation}
\phi_+^*(x)=\begin{cases}
\frac{x}{1-x} \text{ if $x \leq 1$,}\\
\infty \text{ otherwise.}
\end{cases}
\end{equation}
\begin{equation}
{\phi_+^*}'(x)=\frac{1}{(1-x)^2}.
\end{equation}
\begin{equation}
{\phi_+^*}''(x)=\frac{2}{(1-x)^3}.
\end{equation}

\subsection{Jensen-Shannon divergence}
\begin{equation}
\phi_+(x)= \begin{cases}
x\log(x)-(x+1)\log(\frac{x+1}{2}) \text{ if $x \geq 0$,}\\
\infty \text{ otherwise.}
\end{cases}
\end{equation}
\begin{equation}
\partial \phi_+(x)= \begin{cases}
\left\{ \log(x)-\log(x+1)+\log(2) \right\} \text{ if $x > 0$,}\\
\emptyset \text{ otherwise.}
\end{cases}
\end{equation}
\begin{equation}
\phi'(\infty)=\log(2).
\end{equation}
\begin{equation}
\phi_+^*(x)=\begin{cases}
-\log(2-e^x) \text{ if $x \leq \log(2)$,}\\
\infty \text{ otherwise.}
\end{cases}
\end{equation}
\begin{equation}
{\phi_+^*}'(x)=\frac{1}{2e^{-x}-1}.
\end{equation}
\begin{equation}
{\phi_+^*}''(x)=\frac{2e^x}{(e^x-2)^2}.
\end{equation}

\subsection{Jeffreys divergence}
\begin{equation}
\phi_+(x)= \begin{cases}
(x-1)\log(x) \text{ if $x \geq 0$,}\\
\infty \text{ otherwise.}
\end{cases}
\end{equation}
\begin{equation}
\partial \phi_+(x)= \begin{cases}
\left\{ \log(x)-\frac{1}{x}+1 \right\} \text{ if $x > 0$,}\\
\emptyset \text{ otherwise.}
\end{cases}
\end{equation}
\begin{equation}
\phi'(\infty)=\infty.
\end{equation}
\begin{equation}
\phi_+^*(x)=x+W(e^{1-x})+\frac{1}{W(e^{1-x})}-2.
\end{equation}
\begin{equation}
{\phi_+^*}'(x)=\frac{1}{W(e^{1-x})}.
\end{equation}
\begin{equation}
{\phi_+^*}''(x)=\frac{1}{W(e^{1-x})}-\frac{1}{W(e^{1-x})+1}.
\end{equation}

Here $W$ denotes the principal branch of the Lambert W function, also called the product logarithm, defined implicitly by the relation $W(x)e^{W(x)}=x$. This can be computed by Newton's method and differentiated implicitly. For stability, since we only need the value of $W(e^{1-x})$, we compute $W(e^{1-x})$ explicitly instead of composing the Lambert W function with $e^{1-x}$.

\subsection{Triangular discrimination divergence}
\begin{equation}
\phi_+(x)= \begin{cases}
\frac{(x-1)^2}{x+1} \text{ if $x \geq 0$,}\\
\infty \text{ otherwise.}
\end{cases}
\end{equation}
\begin{equation}
\partial \phi_+(x)= \begin{cases}
\left\{ \frac{(x-1)(x+3)}{(x+1)^2} \right\} \text{ if $x \geq 0$,}\\
\emptyset \text{ otherwise.}
\end{cases}
\end{equation}
\begin{equation}
\phi'(\infty)=1.
\end{equation}
\begin{equation}
\phi_+^*(x)=\begin{cases}
-1 \text{ if $x < -3$,}\\
(\sqrt{1-x}-1)(\sqrt{1-x}-3) \text{ if $-3 \leq x \leq 1$,}\\
\infty \text{ otherwise.}
\end{cases}
\end{equation}
\begin{equation}
{\phi_+^*}'(x)=\begin{cases}
0 \text{ if $x < -3$,}\\
\frac{2}{\sqrt{1-x}}-1 \text{ if $-3 \leq x \leq 1$.}\\
\end{cases}
\end{equation}
\begin{equation}
{\phi_+^*}''(x)=\begin{cases}
0 \text{ if $x < -3$,}\\
\frac{1}{(\sqrt{1-x})^3} \text{ if $-3 \leq x \leq 1$.}\\
\end{cases}
\end{equation}

\clearpage
\section{Experimental results} \label{appendix_results}

\subsection{Experimental setup}
As we explained in the main paper, we tested our algorithm on synthetic 2-dimensional data obtained from the codebase of \citet{Feydyetal2018}. These data consists of 4 pairs of densities in the 2-dimensional space named "crescents", "densities", "moons" and "slopes". For each of these density pairs, and for point cloud sizes in $\{500,1000,2000,5000\}$, we sample 5 different point clouds fixing the random seed in $\{0,1,2,3,4\}$. Thus, we have in total $4\times 4\times 5= 80$ different pairs of point clouds that we are going to use in our experiments. An example of such pointclouds for each density pair can be seen in Figure~\ref{figure_data}. Then, for each of the 8 divergences considered (Kullback-Leibler, reverse Kullback-Leibler, $\chi^2$, reverse $\chi^2$, squared Hellinger, Jensen-Shannon, Jeffreys and triangular discrimination) we tried different $\epsilon$ regularization coefficients ranging from $0.1$ to $10^{-8}$.

\begin{figure}[H]
     \centering
     \begin{subfigure}[b]{0.24\textwidth}
         \centering
         \includegraphics[height=.09\paperheight]{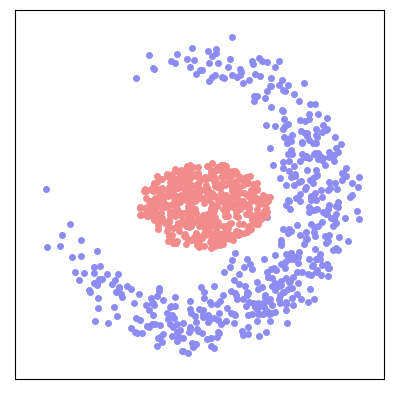}
         \caption{"crescents"}
     \end{subfigure}
     \begin{subfigure}[b]{0.24\textwidth}
         \centering
         \includegraphics[height=.09\paperheight]{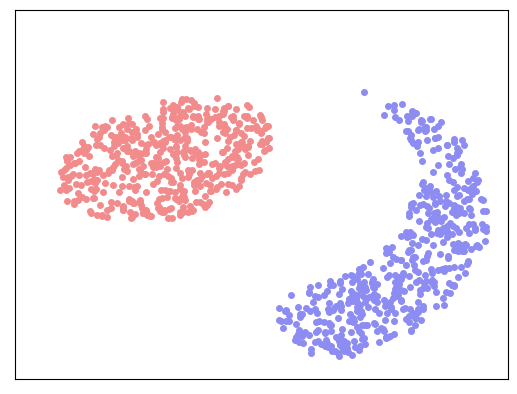}
         \caption{"densities"}
     \end{subfigure}
     \begin{subfigure}[b]{0.24\textwidth}
         \centering
         \includegraphics[height=.09\paperheight]{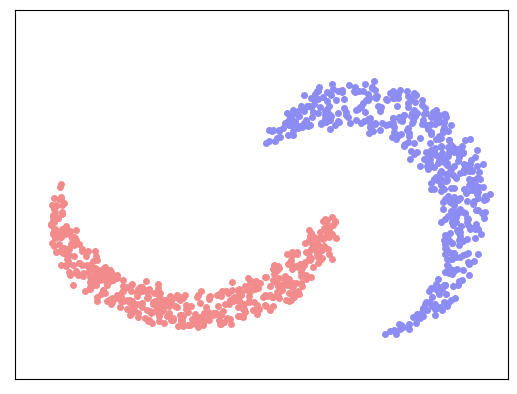}
         \caption{"moons"}
     \end{subfigure}
     \begin{subfigure}[b]{0.24\textwidth}
         \centering
         \includegraphics[height=.09\paperheight]{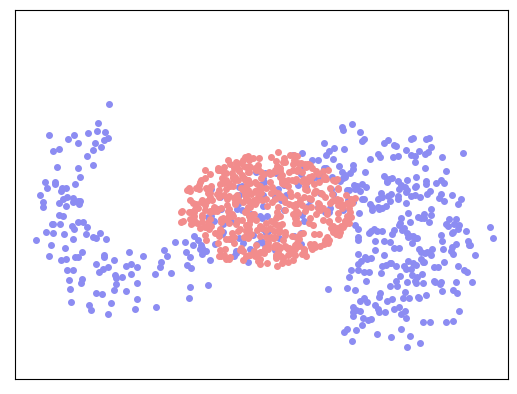}
         \caption{"slopes"}
     \end{subfigure}
        \caption{Example of generated point cloud.}
        \label{figure_data}
\end{figure}

\subsection{Cost of the optimal coupling, convergence speed, sparsity and marginal error}

We can find in Figures \ref{plot_epsilon_vs_cost_and_time_and_sparsity_moons}, \ref{plot_epsilon_vs_cost_and_time_and_sparsity_densities} and \ref{plot_epsilon_vs_cost_and_time_and_sparsity_slopes} the plots of the costs vs running time and the plots of sparsity vs marginal error corresponding to the datasets of "moons", "densities" and "slopes" respectively. Similarly as before, we eliminated the values with marginal error greater than 0.2.

\subsection{Gradients through the optimal coupling}\label{app:grad-through}

As we remarked in the main paper, the algorithm that we present can be used as the loss function between point clouds defined by empirical measures in automatic differentiation engines. In order to do so, one has to compute the gradient with respect to the points in the supports of the measures. An obvious solution is to backpropagate through the Sinkhorn iterations, which is computationally demanding. It is possible to do so via the optimal potentials $f,g$ by generalizing the "graph surgery" method of \citet{Feydyetal2018} and the gradient formula of \citet[Proposition~3.7]{Dimarinoetal2020}. Instead, we propose to do so via the optimal coupling $\pi$. Detaching $\pi$ from the computational graph and calculating $\int c d\pi = \sum_{i,j} C_{i,j} \pi_{i,j}$ leads to a scalar loss which depends on the points $\{x_i\}$ and $\{y_j\}$ only through the cost function $c$.

An intrinsic feature of entropic regularization is that introduces a tradeoff between convergence speed of the Sinkhorn algorithm and bias in the optimal coupling (i.e., the coupling obtained minimizes $\int c\;d\pi+\epsilon D_{f}(\pi\|\mu\otimes\nu)$ instead of the original $\int c\;d\pi$ for $\pi\in \Pi(\mu,\nu)$). Increasing $\epsilon$ leads to faster convergence, but pushes the optimal coupling further away from the coupling which is optimal in the unregularized problem. Using different $f$-divergences for regularization leads to different biases. Since the range of values of these divergences can be quite different, there is no point in comparing the induced biases with equal $\epsilon$s. To make a fair comparison, we tuned the value of $\epsilon$ for each task-divergence setting in order for the Sinkhorn algorithm to converge in $200$ iterations with a tolerance of $\tau=10^{-6}$. 

\clearpage
\begin{multicols}{3}

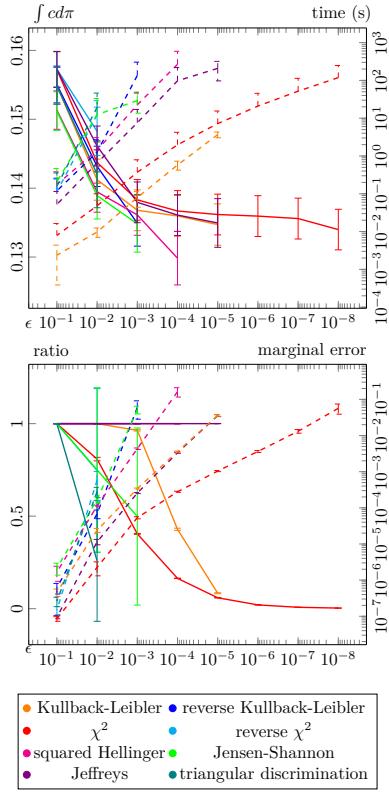
\begin{figure}[H]
\begin{center}
\resizebox{1.0\columnwidth}{!}{
\begin{tikzpicture}
\pgfplotsset{set layers}
\begin{axis}[
	axis y line*=left,
	error bars/y dir=both,
	error bars/y explicit,
    xlabel = iteration,
    xmode = log,
    x dir=reverse,
    xlabel = $\epsilon$,
    x label style={at={(current axis.left of origin)},anchor=east, above=0mm},
    ylabel = $\int c d\pi$,
    y label style={at={(axis description cs:0.175,1.05)},anchor=west,rotate=-90},
    yticklabel style={rotate=90},
]
\addplot[orange, thick] table [x=epsilon, y=cost_avg, y error=cost_std, col sep=comma] {csv/kl_moons_avg_and_std_over_random_seed_and_size_and_dataset_filtered.csv};
\addplot[blue, thick] table [x=epsilon, y=cost_avg, y error=cost_std, col sep=comma] {csv/rkl_moons_avg_and_std_over_random_seed_and_size_and_dataset_filtered.csv};
\addplot[red, thick] table [x=epsilon, y=cost_avg, y error=cost_std, col sep=comma] {csv/chi2_moons_avg_and_std_over_random_seed_and_size_and_dataset_filtered.csv};
\addplot[cyan, thick] table [x=epsilon, y=cost_avg, y error=cost_std, col sep=comma] {csv/rchi2_moons_avg_and_std_over_random_seed_and_size_and_dataset_filtered.csv};
\addplot[magenta, thick] table [x=epsilon, y=cost_avg, y error=cost_std, col sep=comma] {csv/hellinger2_lowtol_moons_avg_and_std_over_random_seed_and_size_and_dataset_filtered.csv};
\addplot[green, thick] table [x=epsilon, y=cost_avg, y error=cost_std, col sep=comma] {csv/js_lowtol_moons_avg_and_std_over_random_seed_and_size_and_dataset_filtered.csv};
\addplot[violet, thick] table [x=epsilon, y=cost_avg, y error=cost_std, col sep=comma] {csv/jeffreys_moons_avg_and_std_over_random_seed_and_size_and_dataset_filtered.csv};
\addplot[teal, thick] table [x=epsilon, y=cost_avg, y error=cost_std, col sep=comma] {csv/triangular_moons_avg_and_std_over_random_seed_and_size_and_dataset_filtered.csv};
\end{axis}
\begin{axis}[
	axis y line*=right,
	error bars/y dir=both,
	error bars/y explicit,
	axis x line=none,
    xlabel = iteration,
    xmode = log,
    x dir=reverse,
	ymode = log, 
    ylabel = time (s),
    y label style={at={(axis description cs:1.00,1.05)},anchor=west,rotate=-90},
    yticklabel style={rotate=90},   
]
\addplot[orange, dashed, thick] table [x=epsilon, y=time_avg, y error=time_std, col sep=comma] {csv/kl_moons_avg_and_std_over_random_seed_and_size_and_dataset_filtered.csv};
\addplot[blue, dashed, thick] table [x=epsilon, y=time_avg, y error=time_std, col sep=comma] {csv/rkl_moons_avg_and_std_over_random_seed_and_size_and_dataset_filtered.csv};
\addplot[red, dashed, thick] table [x=epsilon, y=time_avg, y error=time_std, col sep=comma] {csv/chi2_moons_avg_and_std_over_random_seed_and_size_and_dataset_filtered.csv};
\addplot[cyan, dashed, thick] table [x=epsilon, y=time_avg, y error=time_std, col sep=comma] {csv/rchi2_moons_avg_and_std_over_random_seed_and_size_and_dataset_filtered.csv};
\addplot[magenta, dashed, thick] table [x=epsilon, y=time_avg, y error=time_std, col sep=comma] {csv/hellinger2_lowtol_moons_avg_and_std_over_random_seed_and_size_and_dataset_filtered.csv};
\addplot[green, dashed, thick] table [x=epsilon, y=time_avg, y error=time_std, col sep=comma] {csv/js_lowtol_moons_avg_and_std_over_random_seed_and_size_and_dataset_filtered.csv};
\addplot[violet, dashed, thick] table [x=epsilon, y=time_avg, y error=time_std, col sep=comma] {csv/jeffreys_moons_avg_and_std_over_random_seed_and_size_and_dataset_filtered.csv};
\addplot[teal, dashed, thick] table [x=epsilon, y=time_avg, y error=time_std, col sep=comma] {csv/triangular_moons_avg_and_std_over_random_seed_and_size_and_dataset_filtered.csv};
\end{axis}
\end{tikzpicture}
}
\resizebox{1.0\columnwidth}{!}{
\begin{tikzpicture}
\pgfplotsset{set layers}
\begin{axis}[
	legend style={at={(0.5,-0.175)},anchor=north},
	legend columns=2,
	axis y line*=left,
	error bars/y dir=both,
	error bars/y explicit,
    xlabel = iteration,
    xmode = log,
    x dir=reverse,
    xlabel = $\epsilon$,
    x label style={at={(current axis.left of origin)},anchor=east, below=1mm},
    ylabel = ratio,
    y label style={at={(axis description cs:0.175,1.05)},anchor=west,rotate=-90},
    yticklabel style={rotate=90},
    legend entries={Kullback-Leibler,reverse Kullback-Leibler,$\chi^2$,reverse $\chi^2$,squared Hellinger,Jensen-Shannon,Jeffreys,triangular discrimination},
]
\addlegendimage{only marks, mark=*,mark options={scale=1, fill=orange}, orange}
\addlegendimage{only marks, mark=*,mark options={scale=1, fill=blue}, blue}
\addlegendimage{only marks, mark=*,mark options={scale=1, fill=red}, red}
\addlegendimage{only marks, mark=*,mark options={scale=1, fill=cyan}, cyan}
\addlegendimage{only marks, mark=*,mark options={scale=1, fill=magenta}, magenta}
\addlegendimage{only marks, mark=*,mark options={scale=1, fill=green}, green}
\addlegendimage{only marks, mark=*,mark options={scale=1, fill=violet}, violet}
\addlegendimage{only marks, mark=*,mark options={scale=1, fill=teal}, teal}
\addplot[orange, thick] table [x=epsilon, y=pi_nonzeros_ratio_avg, y error=pi_nonzeros_ratio_std, col sep=comma] {csv/kl_moons_avg_and_std_over_random_seed_and_size_and_dataset_filtered.csv};
\addplot[blue, thick] table [x=epsilon, y=pi_nonzeros_ratio_avg, y error=pi_nonzeros_ratio_std, col sep=comma] {csv/rkl_moons_avg_and_std_over_random_seed_and_size_and_dataset_filtered.csv};
\addplot[red, thick] table [x=epsilon, y=pi_nonzeros_ratio_avg, y error=pi_nonzeros_ratio_std, col sep=comma] {csv/chi2_moons_avg_and_std_over_random_seed_and_size_and_dataset_filtered.csv};
\addplot[cyan, thick] table [x=epsilon, y=pi_nonzeros_ratio_avg, y error=pi_nonzeros_ratio_std, col sep=comma] {csv/rchi2_moons_avg_and_std_over_random_seed_and_size_and_dataset_filtered.csv};
\addplot[magenta, thick] table [x=epsilon, y=pi_nonzeros_ratio_avg, y error=pi_nonzeros_ratio_std, col sep=comma] {csv/hellinger2_lowtol_moons_avg_and_std_over_random_seed_and_size_and_dataset_filtered.csv};
\addplot[green, thick] table [x=epsilon, y=pi_nonzeros_ratio_avg, y error=pi_nonzeros_ratio_std, col sep=comma] {csv/js_lowtol_moons_avg_and_std_over_random_seed_and_size_and_dataset_filtered.csv};
\addplot[violet, thick] table [x=epsilon, y=pi_nonzeros_ratio_avg, y error=pi_nonzeros_ratio_std, col sep=comma] {csv/jeffreys_moons_avg_and_std_over_random_seed_and_size_and_dataset_filtered.csv};
\addplot[teal, thick] table [x=epsilon, y=pi_nonzeros_ratio_avg, y error=pi_nonzeros_ratio_std, col sep=comma] {csv/triangular_moons_avg_and_std_over_random_seed_and_size_and_dataset_filtered.csv};
\end{axis}
\begin{axis}[
	axis y line*=right,
	error bars/y dir=both,
	error bars/y explicit,
	axis x line=none,
    xlabel = iteration,
    xmode = log,
    x dir=reverse,
	ymode = log, 
    ylabel = marginal error,
    y label style={at={(axis description cs:0.84,1.05)},anchor=west,rotate=-90},
    yticklabel style={rotate=90},   
]
\addplot[violet, dashed, thick] table [x=epsilon, y=marginal_error_1_avg, y error=marginal_error_1_std, col sep=comma] {csv/jeffreys_moons_avg_and_std_over_random_seed_and_size_and_dataset_filtered.csv};
\addplot[orange, dashed, thick] table [x=epsilon, y=marginal_error_1_avg, y error=marginal_error_1_std, col sep=comma] {csv/kl_moons_avg_and_std_over_random_seed_and_size_and_dataset_filtered.csv};
\addplot[blue, dashed, thick] table [x=epsilon, y=marginal_error_1_avg, y error=marginal_error_1_std, col sep=comma] {csv/rkl_moons_avg_and_std_over_random_seed_and_size_and_dataset_filtered.csv};
\addplot[red, dashed, thick] table [x=epsilon, y=marginal_error_1_avg, y error=marginal_error_1_std, col sep=comma] {csv/chi2_moons_avg_and_std_over_random_seed_and_size_and_dataset_filtered.csv};
\addplot[cyan, dashed, thick] table [x=epsilon, y=marginal_error_1_avg, y error=marginal_error_1_std, col sep=comma] {csv/rchi2_moons_avg_and_std_over_random_seed_and_size_and_dataset_filtered.csv};
\addplot[magenta, dashed, thick] table [x=epsilon, y=marginal_error_1_avg, y error=marginal_error_1_std, col sep=comma] {csv/hellinger2_lowtol_moons_avg_and_std_over_random_seed_and_size_and_dataset_filtered.csv};
\addplot[green, dashed, thick] table [x=epsilon, y=marginal_error_1_avg, y error=marginal_error_1_std, col sep=comma] {csv/js_lowtol_moons_avg_and_std_over_random_seed_and_size_and_dataset_filtered.csv};
\addplot[teal, dashed, thick] table [x=epsilon, y=marginal_error_1_avg, y error=marginal_error_1_std, col sep=comma] {csv/triangular_moons_avg_and_std_over_random_seed_and_size_and_dataset_filtered.csv};
\end{axis}
\end{tikzpicture}
}
\caption{Dataset: "moons". Above: cost of optimal coupling (solid line) and runtime in seconds (dashed line). Below: ratio of positive elements to all elements in optimal coupling (solid line) and marginal error (dashed line).}
\label{plot_epsilon_vs_cost_and_time_and_sparsity_moons}
\end{center}
\end{figure}

\columnbreak

\begin{figure}[H]
\begin{center}
\resizebox{1.0\columnwidth}{!}{
\begin{tikzpicture}
\pgfplotsset{set layers}
\begin{axis}[
	axis y line*=left,
	error bars/y dir=both,
	error bars/y explicit,
    xlabel = iteration,
    xmode = log,
    x dir=reverse,
    xlabel = $\epsilon$,
    x label style={at={(current axis.left of origin)},anchor=east, above=0mm},
    ylabel = $\int c d\pi$,
    y label style={at={(axis description cs:0.175,1.05)},anchor=west,rotate=-90},
    yticklabel style={rotate=90},
]
\addplot[orange, thick] table [x=epsilon, y=cost_avg, y error=cost_std, col sep=comma] {csv/kl_densities_avg_and_std_over_random_seed_and_size_and_dataset_filtered.csv};
\addplot[blue, thick] table [x=epsilon, y=cost_avg, y error=cost_std, col sep=comma] {csv/rkl_densities_avg_and_std_over_random_seed_and_size_and_dataset_filtered.csv};
\addplot[red, thick] table [x=epsilon, y=cost_avg, y error=cost_std, col sep=comma] {csv/chi2_densities_avg_and_std_over_random_seed_and_size_and_dataset_filtered.csv};
\addplot[cyan, thick] table [x=epsilon, y=cost_avg, y error=cost_std, col sep=comma] {csv/rchi2_densities_avg_and_std_over_random_seed_and_size_and_dataset_filtered.csv};
\addplot[magenta, thick] table [x=epsilon, y=cost_avg, y error=cost_std, col sep=comma] {csv/hellinger2_lowtol_densities_avg_and_std_over_random_seed_and_size_and_dataset_filtered.csv};
\addplot[green, thick] table [x=epsilon, y=cost_avg, y error=cost_std, col sep=comma] {csv/js_lowtol_densities_avg_and_std_over_random_seed_and_size_and_dataset_filtered.csv};
\addplot[violet, thick] table [x=epsilon, y=cost_avg, y error=cost_std, col sep=comma] {csv/jeffreys_densities_avg_and_std_over_random_seed_and_size_and_dataset_filtered.csv};
\addplot[teal, thick] table [x=epsilon, y=cost_avg, y error=cost_std, col sep=comma] {csv/triangular_densities_avg_and_std_over_random_seed_and_size_and_dataset_filtered.csv};
\end{axis}
\begin{axis}[
	axis y line*=right,
	error bars/y dir=both,
	error bars/y explicit,
	axis x line=none,
    xlabel = iteration,
    xmode = log,
    x dir=reverse,
	ymode = log, 
    ylabel = time (s),
    y label style={at={(axis description cs:1.00,1.05)},anchor=west,rotate=-90},
    yticklabel style={rotate=90},   
]
\addplot[orange, dashed, thick] table [x=epsilon, y=time_avg, y error=time_std, col sep=comma] {csv/kl_densities_avg_and_std_over_random_seed_and_size_and_dataset_filtered.csv};
\addplot[blue, dashed, thick] table [x=epsilon, y=time_avg, y error=time_std, col sep=comma] {csv/rkl_densities_avg_and_std_over_random_seed_and_size_and_dataset_filtered.csv};
\addplot[red, dashed, thick] table [x=epsilon, y=time_avg, y error=time_std, col sep=comma] {csv/chi2_densities_avg_and_std_over_random_seed_and_size_and_dataset_filtered.csv};
\addplot[cyan, dashed, thick] table [x=epsilon, y=time_avg, y error=time_std, col sep=comma] {csv/rchi2_densities_avg_and_std_over_random_seed_and_size_and_dataset_filtered.csv};
\addplot[magenta, dashed, thick] table [x=epsilon, y=time_avg, y error=time_std, col sep=comma] {csv/hellinger2_lowtol_densities_avg_and_std_over_random_seed_and_size_and_dataset_filtered.csv};
\addplot[green, dashed, thick] table [x=epsilon, y=time_avg, y error=time_std, col sep=comma] {csv/js_lowtol_densities_avg_and_std_over_random_seed_and_size_and_dataset_filtered.csv};
\addplot[violet, dashed, thick] table [x=epsilon, y=time_avg, y error=time_std, col sep=comma] {csv/jeffreys_densities_avg_and_std_over_random_seed_and_size_and_dataset_filtered.csv};
\addplot[teal, dashed, thick] table [x=epsilon, y=time_avg, y error=time_std, col sep=comma] {csv/triangular_densities_avg_and_std_over_random_seed_and_size_and_dataset_filtered.csv};
\end{axis}
\end{tikzpicture}
}
\resizebox{1.0\columnwidth}{!}{
\begin{tikzpicture}
\pgfplotsset{set layers}
\begin{axis}[
	legend style={at={(0.5,-0.175)},anchor=north},
	legend columns=2,
	axis y line*=left,
	error bars/y dir=both,
	error bars/y explicit,
    xlabel = iteration,
    xmode = log,
    x dir=reverse,
    xlabel = $\epsilon$,
    x label style={at={(current axis.left of origin)},anchor=east, below=1mm},
    ylabel = ratio,
    y label style={at={(axis description cs:0.175,1.05)},anchor=west,rotate=-90},
    yticklabel style={rotate=90},
    legend entries={Kullback-Leibler,reverse Kullback-Leibler,$\chi^2$,reverse $\chi^2$,squared Hellinger,Jensen-Shannon,Jeffreys,triangular discrimination},
]
\addlegendimage{only marks, mark=*,mark options={scale=1, fill=orange}, orange}
\addlegendimage{only marks, mark=*,mark options={scale=1, fill=blue}, blue}
\addlegendimage{only marks, mark=*,mark options={scale=1, fill=red}, red}
\addlegendimage{only marks, mark=*,mark options={scale=1, fill=cyan}, cyan}
\addlegendimage{only marks, mark=*,mark options={scale=1, fill=magenta}, magenta}
\addlegendimage{only marks, mark=*,mark options={scale=1, fill=green}, green}
\addlegendimage{only marks, mark=*,mark options={scale=1, fill=violet}, violet}
\addlegendimage{only marks, mark=*,mark options={scale=1, fill=teal}, teal}
\addplot[orange, thick] table [x=epsilon, y=pi_nonzeros_ratio_avg, y error=pi_nonzeros_ratio_std, col sep=comma] {csv/kl_densities_avg_and_std_over_random_seed_and_size_and_dataset_filtered.csv};
\addplot[blue, thick] table [x=epsilon, y=pi_nonzeros_ratio_avg, y error=pi_nonzeros_ratio_std, col sep=comma] {csv/rkl_densities_avg_and_std_over_random_seed_and_size_and_dataset_filtered.csv};
\addplot[red, thick] table [x=epsilon, y=pi_nonzeros_ratio_avg, y error=pi_nonzeros_ratio_std, col sep=comma] {csv/chi2_densities_avg_and_std_over_random_seed_and_size_and_dataset_filtered.csv};
\addplot[cyan, thick] table [x=epsilon, y=pi_nonzeros_ratio_avg, y error=pi_nonzeros_ratio_std, col sep=comma] {csv/rchi2_densities_avg_and_std_over_random_seed_and_size_and_dataset_filtered.csv};
\addplot[magenta, thick] table [x=epsilon, y=pi_nonzeros_ratio_avg, y error=pi_nonzeros_ratio_std, col sep=comma] {csv/hellinger2_lowtol_densities_avg_and_std_over_random_seed_and_size_and_dataset_filtered.csv};
\addplot[green, thick] table [x=epsilon, y=pi_nonzeros_ratio_avg, y error=pi_nonzeros_ratio_std, col sep=comma] {csv/js_lowtol_densities_avg_and_std_over_random_seed_and_size_and_dataset_filtered.csv};
\addplot[violet, thick] table [x=epsilon, y=pi_nonzeros_ratio_avg, y error=pi_nonzeros_ratio_std, col sep=comma] {csv/jeffreys_densities_avg_and_std_over_random_seed_and_size_and_dataset_filtered.csv};
\addplot[teal, thick] table [x=epsilon, y=pi_nonzeros_ratio_avg, y error=pi_nonzeros_ratio_std, col sep=comma] {csv/triangular_densities_avg_and_std_over_random_seed_and_size_and_dataset_filtered.csv};
\end{axis}
\begin{axis}[
	axis y line*=right,
	error bars/y dir=both,
	error bars/y explicit,
	axis x line=none,
    xlabel = iteration,
    xmode = log,
    x dir=reverse,
	ymode = log, 
    ylabel = marginal error,
    y label style={at={(axis description cs:0.84,1.05)},anchor=west,rotate=-90},
    yticklabel style={rotate=90},   
]
\addplot[violet, dashed, thick] table [x=epsilon, y=marginal_error_1_avg, y error=marginal_error_1_std, col sep=comma] {csv/jeffreys_densities_avg_and_std_over_random_seed_and_size_and_dataset_filtered.csv};
\addplot[orange, dashed, thick] table [x=epsilon, y=marginal_error_1_avg, y error=marginal_error_1_std, col sep=comma] {csv/kl_densities_avg_and_std_over_random_seed_and_size_and_dataset_filtered.csv};
\addplot[blue, dashed, thick] table [x=epsilon, y=marginal_error_1_avg, y error=marginal_error_1_std, col sep=comma] {csv/rkl_densities_avg_and_std_over_random_seed_and_size_and_dataset_filtered.csv};
\addplot[red, dashed, thick] table [x=epsilon, y=marginal_error_1_avg, y error=marginal_error_1_std, col sep=comma] {csv/chi2_densities_avg_and_std_over_random_seed_and_size_and_dataset_filtered.csv};
\addplot[cyan, dashed, thick] table [x=epsilon, y=marginal_error_1_avg, y error=marginal_error_1_std, col sep=comma] {csv/rchi2_densities_avg_and_std_over_random_seed_and_size_and_dataset_filtered.csv};
\addplot[magenta, dashed, thick] table [x=epsilon, y=marginal_error_1_avg, y error=marginal_error_1_std, col sep=comma] {csv/hellinger2_lowtol_densities_avg_and_std_over_random_seed_and_size_and_dataset_filtered.csv};
\addplot[green, dashed, thick] table [x=epsilon, y=marginal_error_1_avg, y error=marginal_error_1_std, col sep=comma] {csv/js_lowtol_densities_avg_and_std_over_random_seed_and_size_and_dataset_filtered.csv};
\addplot[teal, dashed, thick] table [x=epsilon, y=marginal_error_1_avg, y error=marginal_error_1_std, col sep=comma] {csv/triangular_densities_avg_and_std_over_random_seed_and_size_and_dataset_filtered.csv};
\end{axis}
\end{tikzpicture}
}
\caption{Dataset: "densities". Above: cost of optimal coupling (solid line) and runtime in seconds (dashed line). Below: ratio of positive elements to all elements in optimal coupling (solid line) and marginal error (dashed line).}
\label{plot_epsilon_vs_cost_and_time_and_sparsity_densities}
\end{center}
\end{figure}
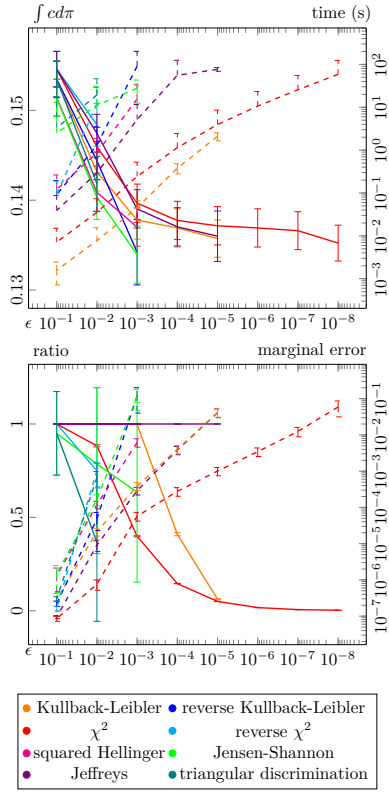

\columnbreak

\begin{figure}[H]
\begin{center}
\resizebox{1.0\columnwidth}{!}{
\begin{tikzpicture}
\pgfplotsset{set layers}
\begin{axis}[
	axis y line*=left,
	error bars/y dir=both,
	error bars/y explicit,
    xlabel = iteration,
    xmode = log,
    x dir=reverse,
    xlabel = $\epsilon$,
    x label style={at={(current axis.left of origin)},anchor=east, above=0mm},
    ylabel = $\int c d\pi$,
    y label style={at={(axis description cs:0.175,1.05)},anchor=west,rotate=-90},
    yticklabel style={rotate=90},
]
\addplot[orange, thick] table [x=epsilon, y=cost_avg, y error=cost_std, col sep=comma] {csv/kl_slopes_avg_and_std_over_random_seed_and_size_and_dataset_filtered.csv};
\addplot[blue, thick] table [x=epsilon, y=cost_avg, y error=cost_std, col sep=comma] {csv/rkl_slopes_avg_and_std_over_random_seed_and_size_and_dataset_filtered.csv};
\addplot[red, thick] table [x=epsilon, y=cost_avg, y error=cost_std, col sep=comma] {csv/chi2_slopes_avg_and_std_over_random_seed_and_size_and_dataset_filtered.csv};
\addplot[cyan, thick] table [x=epsilon, y=cost_avg, y error=cost_std, col sep=comma] {csv/rchi2_slopes_avg_and_std_over_random_seed_and_size_and_dataset_filtered.csv};
\addplot[magenta, thick] table [x=epsilon, y=cost_avg, y error=cost_std, col sep=comma] {csv/hellinger2_lowtol_slopes_avg_and_std_over_random_seed_and_size_and_dataset_filtered.csv};
\addplot[green, thick] table [x=epsilon, y=cost_avg, y error=cost_std, col sep=comma] {csv/js_lowtol_slopes_avg_and_std_over_random_seed_and_size_and_dataset_filtered.csv};
\addplot[violet, thick] table [x=epsilon, y=cost_avg, y error=cost_std, col sep=comma] {csv/jeffreys_slopes_avg_and_std_over_random_seed_and_size_and_dataset_filtered.csv};
\addplot[teal, thick] table [x=epsilon, y=cost_avg, y error=cost_std, col sep=comma] {csv/triangular_slopes_avg_and_std_over_random_seed_and_size_and_dataset_filtered.csv};
\end{axis}
\begin{axis}[
	axis y line*=right,
	error bars/y dir=both,
	error bars/y explicit,
	axis x line=none,
    xlabel = iteration,
    xmode = log,
    x dir=reverse,
	ymode = log, 
    ylabel = time (s),
    y label style={at={(axis description cs:1.00,1.05)},anchor=west,rotate=-90},
    yticklabel style={rotate=90},   
]
\addplot[orange, dashed, thick] table [x=epsilon, y=time_avg, y error=time_std, col sep=comma] {csv/kl_slopes_avg_and_std_over_random_seed_and_size_and_dataset_filtered.csv};
\addplot[blue, dashed, thick] table [x=epsilon, y=time_avg, y error=time_std, col sep=comma] {csv/rkl_slopes_avg_and_std_over_random_seed_and_size_and_dataset_filtered.csv};
\addplot[red, dashed, thick] table [x=epsilon, y=time_avg, y error=time_std, col sep=comma] {csv/chi2_slopes_avg_and_std_over_random_seed_and_size_and_dataset_filtered.csv};
\addplot[cyan, dashed, thick] table [x=epsilon, y=time_avg, y error=time_std, col sep=comma] {csv/rchi2_slopes_avg_and_std_over_random_seed_and_size_and_dataset_filtered.csv};
\addplot[magenta, dashed, thick] table [x=epsilon, y=time_avg, y error=time_std, col sep=comma] {csv/hellinger2_lowtol_slopes_avg_and_std_over_random_seed_and_size_and_dataset_filtered.csv};
\addplot[green, dashed, thick] table [x=epsilon, y=time_avg, y error=time_std, col sep=comma] {csv/js_lowtol_slopes_avg_and_std_over_random_seed_and_size_and_dataset_filtered.csv};
\addplot[violet, dashed, thick] table [x=epsilon, y=time_avg, y error=time_std, col sep=comma] {csv/jeffreys_slopes_avg_and_std_over_random_seed_and_size_and_dataset_filtered.csv};
\addplot[teal, dashed, thick] table [x=epsilon, y=time_avg, y error=time_std, col sep=comma] {csv/triangular_slopes_avg_and_std_over_random_seed_and_size_and_dataset_filtered.csv};
\end{axis}
\end{tikzpicture}
}
\resizebox{1.0\columnwidth}{!}{
\begin{tikzpicture}
\pgfplotsset{set layers}
\begin{axis}[
	legend style={at={(0.5,-0.175)},anchor=north},
	legend columns=2,
	axis y line*=left,
	error bars/y dir=both,
	error bars/y explicit,
    xlabel = iteration,
    xmode = log,
    x dir=reverse,
    xlabel = $\epsilon$,
    x label style={at={(current axis.left of origin)},anchor=east, below=1mm},
    ylabel = ratio,
    y label style={at={(axis description cs:0.175,1.05)},anchor=west,rotate=-90},
    yticklabel style={rotate=90},
    legend entries={Kullback-Leibler,reverse Kullback-Leibler,$\chi^2$,reverse $\chi^2$,squared Hellinger,Jensen-Shannon,Jeffreys,triangular discrimination},
]
\addlegendimage{only marks, mark=*,mark options={scale=1, fill=orange}, orange}
\addlegendimage{only marks, mark=*,mark options={scale=1, fill=blue}, blue}
\addlegendimage{only marks, mark=*,mark options={scale=1, fill=red}, red}
\addlegendimage{only marks, mark=*,mark options={scale=1, fill=cyan}, cyan}
\addlegendimage{only marks, mark=*,mark options={scale=1, fill=magenta}, magenta}
\addlegendimage{only marks, mark=*,mark options={scale=1, fill=green}, green}
\addlegendimage{only marks, mark=*,mark options={scale=1, fill=violet}, violet}
\addlegendimage{only marks, mark=*,mark options={scale=1, fill=teal}, teal}
\addplot[orange, thick] table [x=epsilon, y=pi_nonzeros_ratio_avg, y error=pi_nonzeros_ratio_std, col sep=comma] {csv/kl_slopes_avg_and_std_over_random_seed_and_size_and_dataset_filtered.csv};
\addplot[blue, thick] table [x=epsilon, y=pi_nonzeros_ratio_avg, y error=pi_nonzeros_ratio_std, col sep=comma] {csv/rkl_slopes_avg_and_std_over_random_seed_and_size_and_dataset_filtered.csv};
\addplot[red, thick] table [x=epsilon, y=pi_nonzeros_ratio_avg, y error=pi_nonzeros_ratio_std, col sep=comma] {csv/chi2_slopes_avg_and_std_over_random_seed_and_size_and_dataset_filtered.csv};
\addplot[cyan, thick] table [x=epsilon, y=pi_nonzeros_ratio_avg, y error=pi_nonzeros_ratio_std, col sep=comma] {csv/rchi2_slopes_avg_and_std_over_random_seed_and_size_and_dataset_filtered.csv};
\addplot[magenta, thick] table [x=epsilon, y=pi_nonzeros_ratio_avg, y error=pi_nonzeros_ratio_std, col sep=comma] {csv/hellinger2_lowtol_slopes_avg_and_std_over_random_seed_and_size_and_dataset_filtered.csv};
\addplot[green, thick] table [x=epsilon, y=pi_nonzeros_ratio_avg, y error=pi_nonzeros_ratio_std, col sep=comma] {csv/js_lowtol_slopes_avg_and_std_over_random_seed_and_size_and_dataset_filtered.csv};
\addplot[violet, thick] table [x=epsilon, y=pi_nonzeros_ratio_avg, y error=pi_nonzeros_ratio_std, col sep=comma] {csv/jeffreys_slopes_avg_and_std_over_random_seed_and_size_and_dataset_filtered.csv};
\addplot[teal, thick] table [x=epsilon, y=pi_nonzeros_ratio_avg, y error=pi_nonzeros_ratio_std, col sep=comma] {csv/triangular_slopes_avg_and_std_over_random_seed_and_size_and_dataset_filtered.csv};
\end{axis}
\begin{axis}[
	axis y line*=right,
	error bars/y dir=both,
	error bars/y explicit,
	axis x line=none,
    xlabel = iteration,
    xmode = log,
    x dir=reverse,
	ymode = log, 
    ylabel = marginal error,
    y label style={at={(axis description cs:0.84,1.05)},anchor=west,rotate=-90},
    yticklabel style={rotate=90},   
]
\addplot[violet, dashed, thick] table [x=epsilon, y=marginal_error_1_avg, y error=marginal_error_1_std, col sep=comma] {csv/jeffreys_slopes_avg_and_std_over_random_seed_and_size_and_dataset_filtered.csv};
\addplot[orange, dashed, thick] table [x=epsilon, y=marginal_error_1_avg, y error=marginal_error_1_std, col sep=comma] {csv/kl_slopes_avg_and_std_over_random_seed_and_size_and_dataset_filtered.csv};
\addplot[blue, dashed, thick] table [x=epsilon, y=marginal_error_1_avg, y error=marginal_error_1_std, col sep=comma] {csv/rkl_slopes_avg_and_std_over_random_seed_and_size_and_dataset_filtered.csv};
\addplot[red, dashed, thick] table [x=epsilon, y=marginal_error_1_avg, y error=marginal_error_1_std, col sep=comma] {csv/chi2_slopes_avg_and_std_over_random_seed_and_size_and_dataset_filtered.csv};
\addplot[cyan, dashed, thick] table [x=epsilon, y=marginal_error_1_avg, y error=marginal_error_1_std, col sep=comma] {csv/rchi2_slopes_avg_and_std_over_random_seed_and_size_and_dataset_filtered.csv};
\addplot[magenta, dashed, thick] table [x=epsilon, y=marginal_error_1_avg, y error=marginal_error_1_std, col sep=comma] {csv/hellinger2_lowtol_slopes_avg_and_std_over_random_seed_and_size_and_dataset_filtered.csv};
\addplot[green, dashed, thick] table [x=epsilon, y=marginal_error_1_avg, y error=marginal_error_1_std, col sep=comma] {csv/js_lowtol_slopes_avg_and_std_over_random_seed_and_size_and_dataset_filtered.csv};
\addplot[teal, dashed, thick] table [x=epsilon, y=marginal_error_1_avg, y error=marginal_error_1_std, col sep=comma] {csv/triangular_slopes_avg_and_std_over_random_seed_and_size_and_dataset_filtered.csv};
\end{axis}
\end{tikzpicture}
}
\caption{Dataset: "slopes". Above: cost of optimal coupling (solid line) and runtime in seconds (dashed line). Below: ratio of positive elements to all elements in optimal coupling (solid line) and marginal error (dashed line).}
\label{plot_epsilon_vs_cost_and_time_and_sparsity_slopes}
\end{center}
\end{figure}
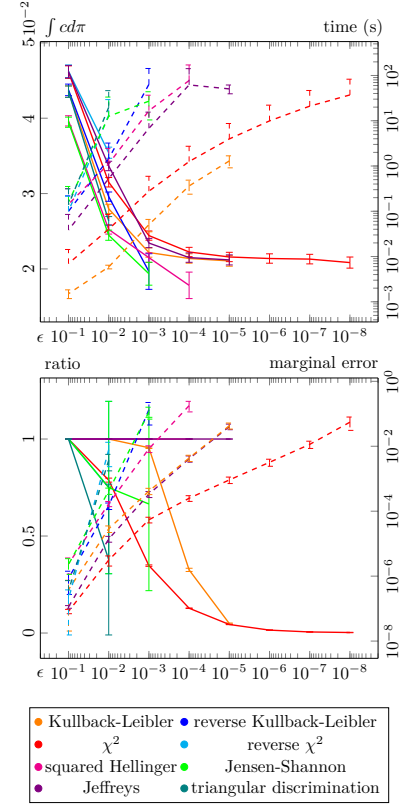

\end{multicols}

\begin{figure}[h]
     \centering
     
     \begin{subfigure}[b]{0.49\textwidth}
         \centering
         \includegraphics[height=.047\paperheight]{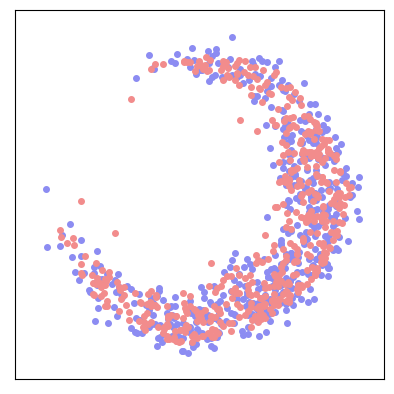}
         \includegraphics[height=.047\paperheight]{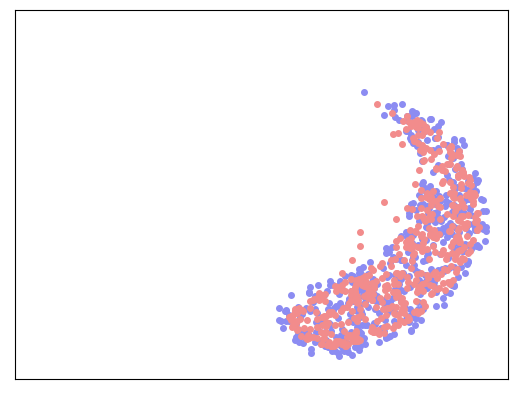}
         \includegraphics[height=.047\paperheight]{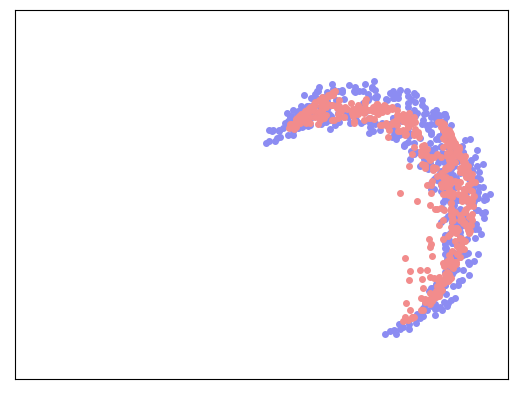}
         \includegraphics[height=.047\paperheight]{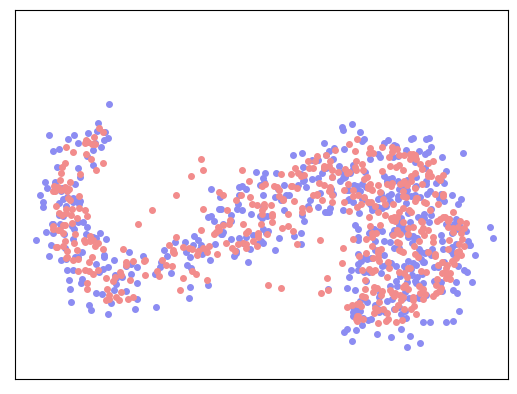}
         \caption{Kullback-Leibler.}
     \end{subfigure}
     \begin{subfigure}[b]{0.49\textwidth}
         \centering
         \includegraphics[height=.047\paperheight]{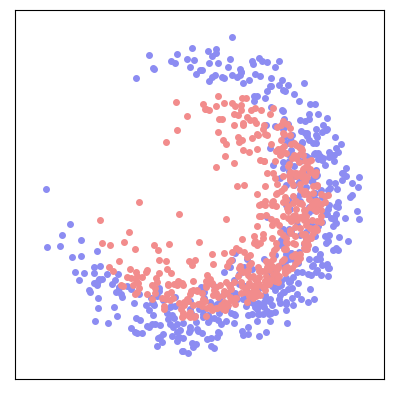}
         \includegraphics[height=.047\paperheight]{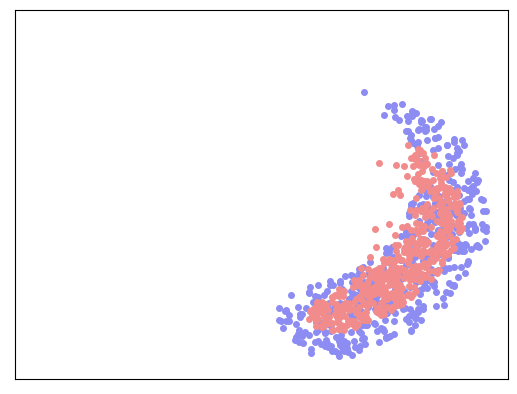}
         \includegraphics[height=.047\paperheight]{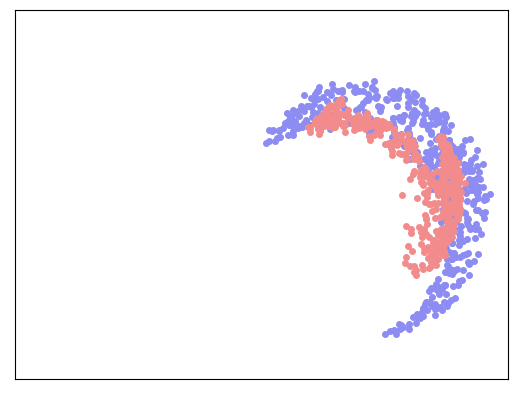}
         \includegraphics[height=.047\paperheight]{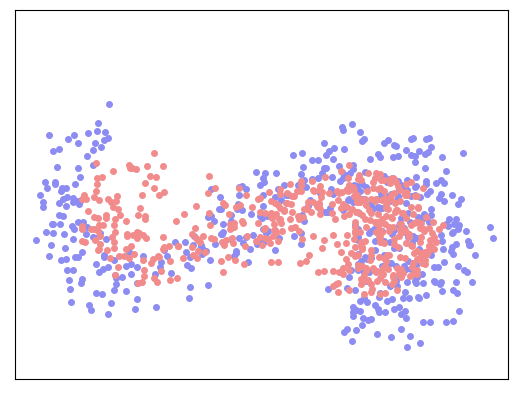}
         \caption{Reverse Kullback-Leibler.}
     \end{subfigure}
     
     \begin{subfigure}[b]{0.49\textwidth}
         \centering
         \includegraphics[height=.047\paperheight]{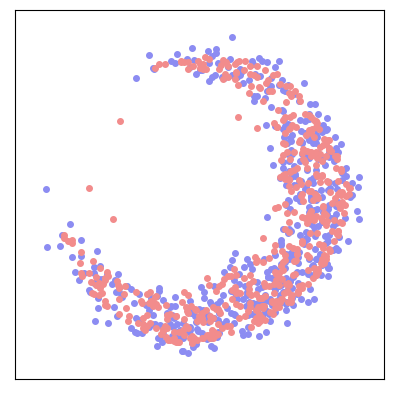}
         \includegraphics[height=.047\paperheight]{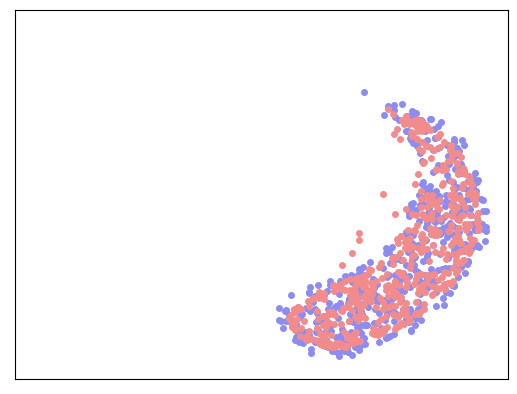}
         \includegraphics[height=.047\paperheight]{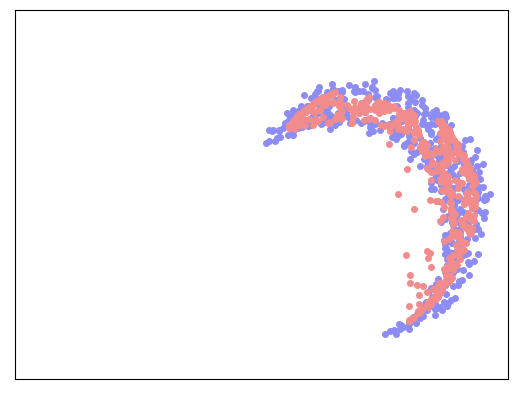}
         \includegraphics[height=.047\paperheight]{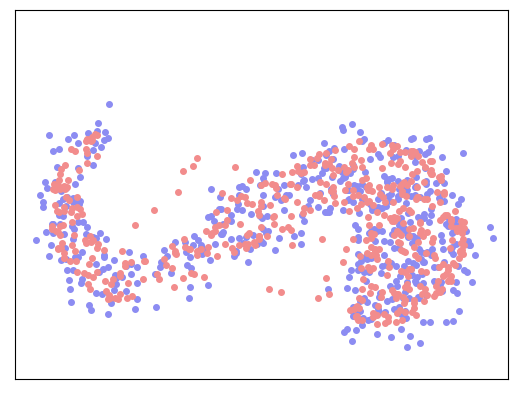}
         \caption{$\chi^2$.}
     \end{subfigure}
     \begin{subfigure}[b]{0.49\textwidth}
         \centering
         \includegraphics[height=.047\paperheight]{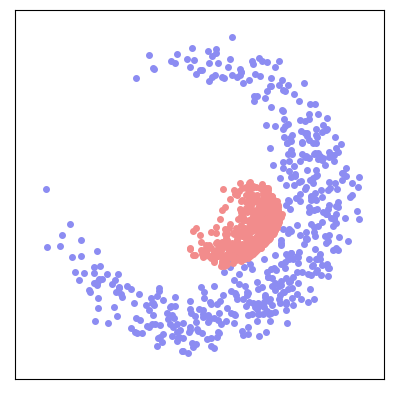}
         \includegraphics[height=.047\paperheight]{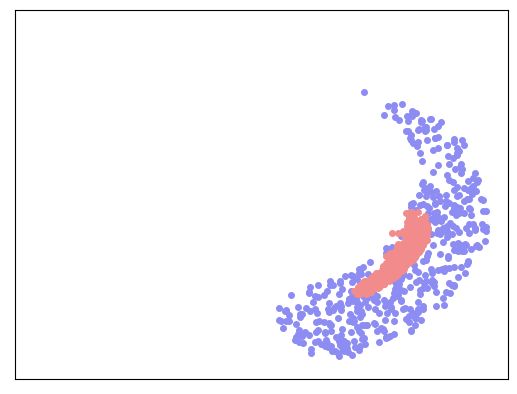}
         \includegraphics[height=.047\paperheight]{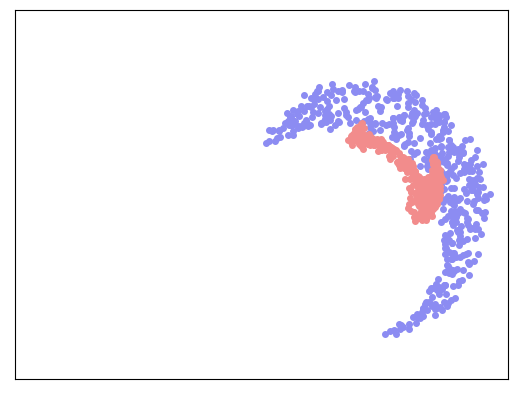}
         \includegraphics[height=.047\paperheight]{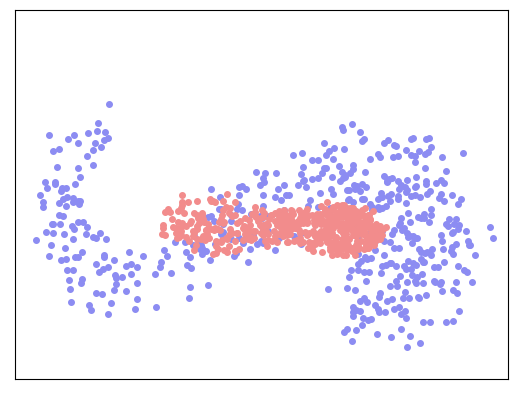}
         \caption{Reverse $\chi^2$.}
     \end{subfigure}
     
     \begin{subfigure}[b]{0.49\textwidth}
         \centering
         \includegraphics[height=.047\paperheight]{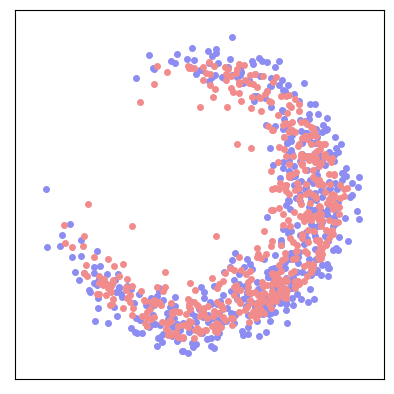}
         \includegraphics[height=.047\paperheight]{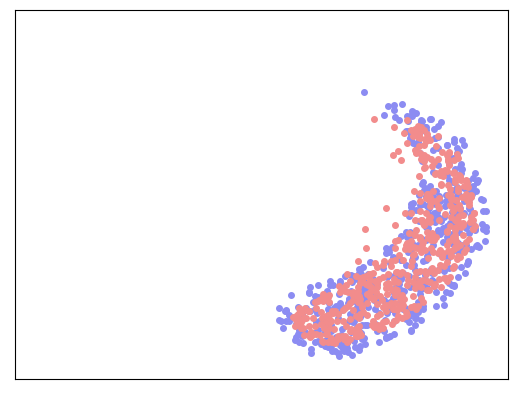}
         \includegraphics[height=.047\paperheight]{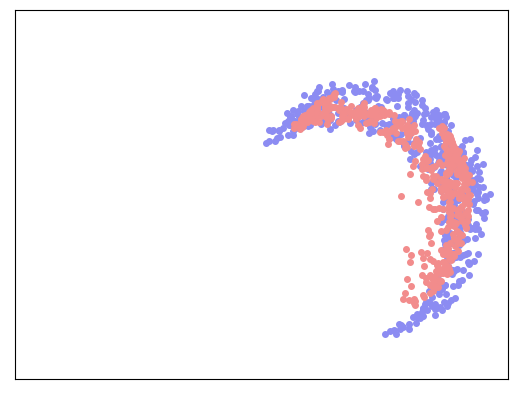}
         \includegraphics[height=.047\paperheight]{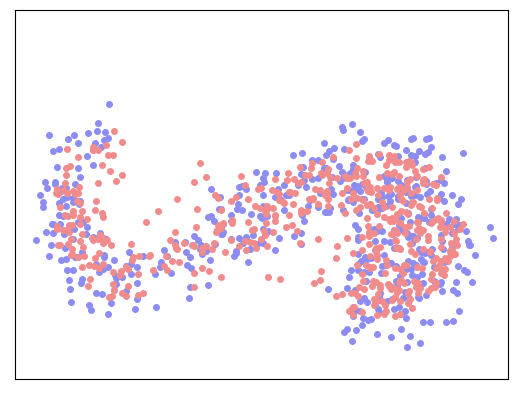}
         \caption{Squared Hellinger.}
     \end{subfigure}
     \begin{subfigure}[b]{0.49\textwidth}
         \centering
         \includegraphics[height=.047\paperheight]{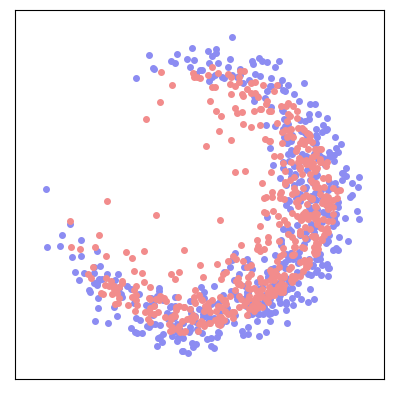}
         \includegraphics[height=.047\paperheight]{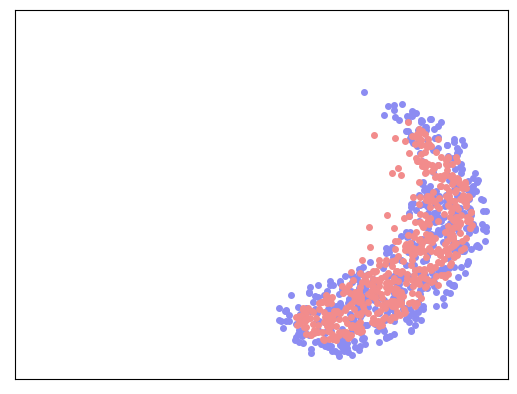}
         \includegraphics[height=.047\paperheight]{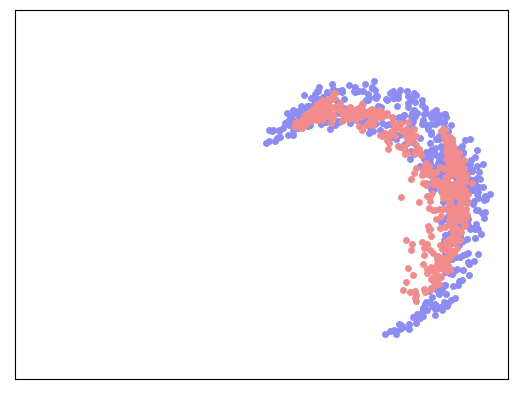}
         \includegraphics[height=.047\paperheight]{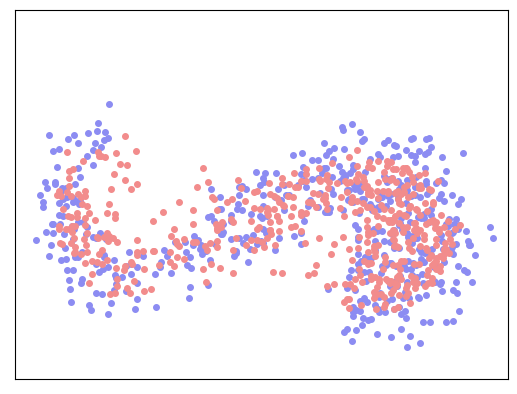}
         \caption{Jensen-Shannon.}
     \end{subfigure}
     
     \begin{subfigure}[b]{0.49\textwidth}
         \centering
         \includegraphics[height=.047\paperheight]{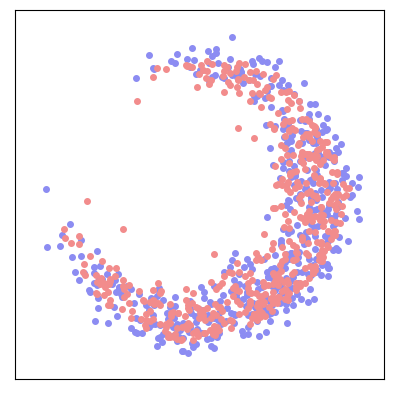}
         \includegraphics[height=.047\paperheight]{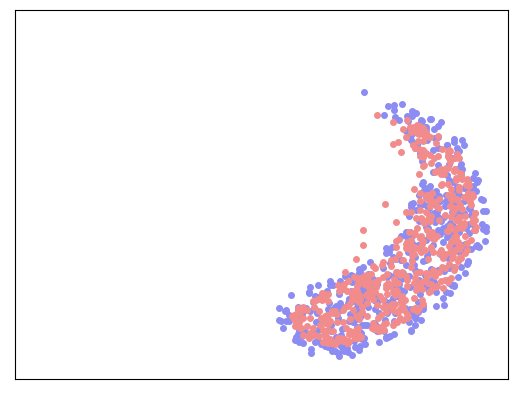}
         \includegraphics[height=.047\paperheight]{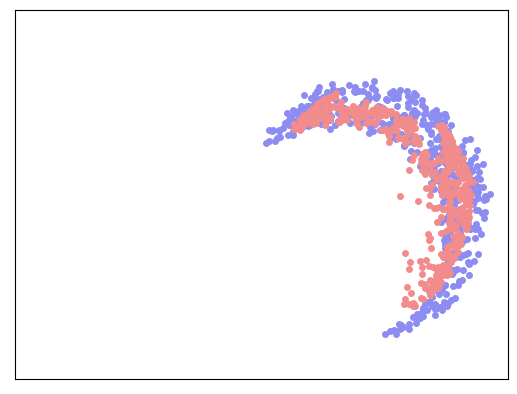}
         \includegraphics[height=.047\paperheight]{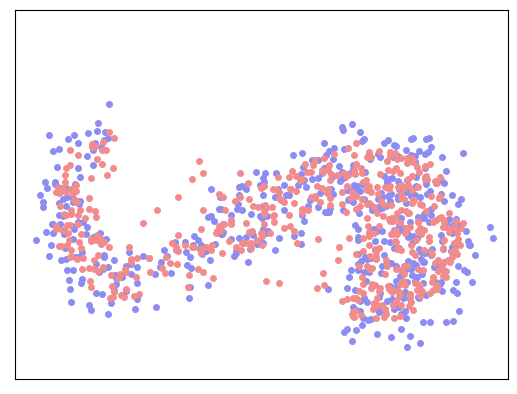}
         \caption{Jeffreys.}
     \end{subfigure}
     \begin{subfigure}[b]{0.49\textwidth}
         \centering
         \includegraphics[height=.047\paperheight]{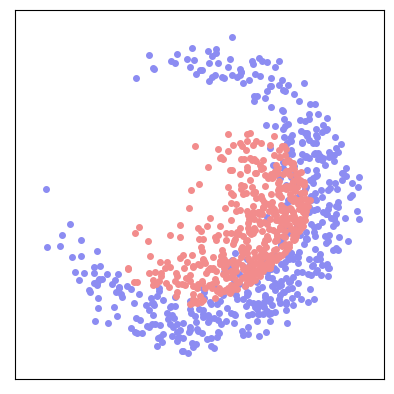}
         \includegraphics[height=.047\paperheight]{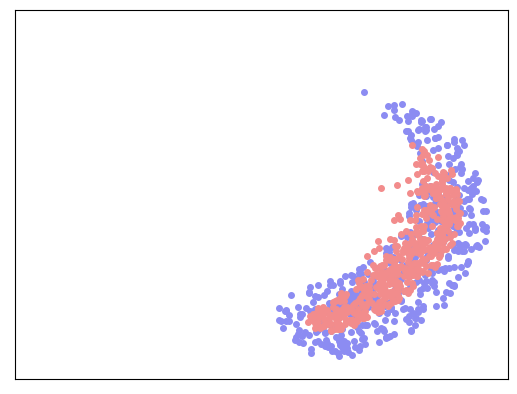}
         \includegraphics[height=.047\paperheight]{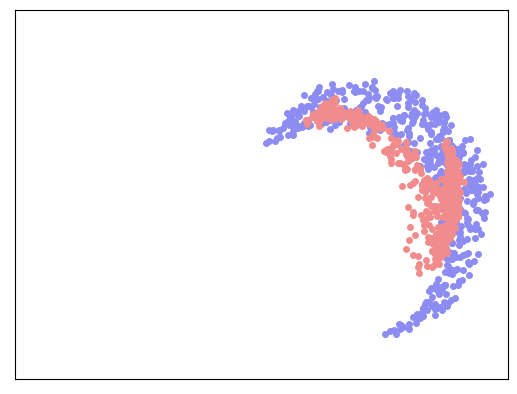}
         \includegraphics[height=.047\paperheight]{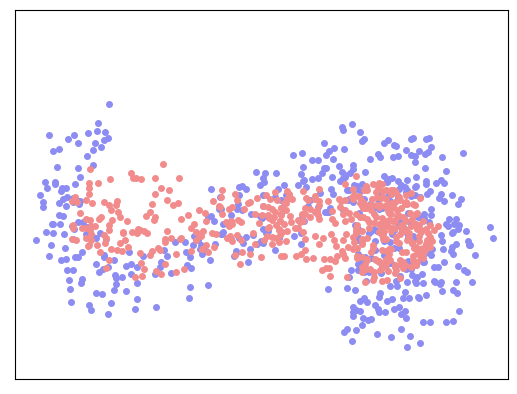}
         \caption{Triangular discrimination.}
     \end{subfigure}
     
        \caption{Bias of optimal couplings when $\epsilon$ is tuned to reach tolerance of $10^{-6}$ in $200$ Sinkhorn iterations.}
        \label{figure_200steps}
\end{figure}

Upon convergence, we backpropagated the resulting loss $\int c d\pi$ as described above, and took $1$ gradient descent step on the points belonging to the support of $\mu$ with a learning rate of $1$. If the coupling $\pi$ were unbiased, this procedure should transport the red pointcloud $\mu$ exactly onto the blue one $\nu$. The results are visualized in Figure~\ref{figure_200steps}. The tradeoff leads to a visually similar, small amount of bias in the case of the Kullback-Leibler, $\chi^2$, squared Hellinger, Jensen-Shannon and Jeffreys divergences.  On the other hand, the bias is more pronounced for the reverse Kullback-Leibler, reverse $\chi^2$ and triangular discrimination divergences. The bias can be reduced in all cases by decreasing $\epsilon$, at a price of slower convergence speed. For other application scenarios, practitioners might benefit from evaluating all considered $f$-divergences, since the biases in other tasks could differ.

\begin{figure}[H]
\begin{center}
\resizebox{0.3\columnwidth}{!}{
\begin{tikzpicture}
\pgfplotsset{set layers}
\begin{axis}[
	axis y line*=left,
    xmode = log,
    x dir=reverse,
    xlabel = $\tau$,
    x label style={at={(current axis.left of origin)},anchor=east, below=0mm},
    ylabel = marginal error,
    y label style={at={(axis description cs:0.175,1.05)},anchor=west,rotate=-90},
    yticklabel style={rotate=90},
]
\addplot[orange, thick] table [x=tolerance, y=marginal_error_1, col sep=comma] {csv/kl_tolmargerr.csv};
\end{axis}
\begin{axis}[
	axis y line*=right,
	axis x line=none,
    xmode = log,
    x dir=reverse,
	ymode = log, 
    ylabel = time (s),
    y label style={at={(axis description cs:1.00,1.05)},anchor=west,rotate=-90},
    yticklabel style={rotate=90},   
]
\addplot[orange, dashed, thick] table [x=tolerance, y=time, col sep=comma] {csv/kl_tolmargerr.csv};
\end{axis}
\end{tikzpicture}
}
\caption{Marginal error (solid line) can be decreased by decreasing the tolerance parameter $\tau$ at the cost of increased running time (dashed line). This plots was done using the Kullback-Leibler divergece, the "crescents" dataset, a point cloud size of 500 and 0 as the random seed.}
\label{plot_tol_marg_err}
\end{center}
\end{figure}
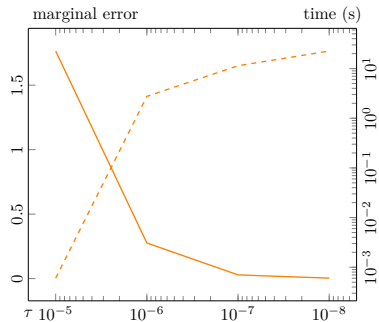

\subsection{Fixing marginal errors}

Depending of the hyperparemeters, the resulting coupling can have a large marginal error. There are algorithms in the literature that correct an approximate coupling to an exact one such as \citet[Algorithm~2]{Altschuler2017}, which we have included in the source code of the experiments. However, the final cost of the resulting coupling was worse in general than the one without this rounding step and also the sparsity of the coupling disappears. Hence, we decided not to include this rounding step in our algorithm. Without such a rounding step, practitioners should set a lower tolerance parameter $\tau$ to decrease the marginal error of the resulting coupling. We ran several experiments to see this effect with different divergences, random seeds, and point cloud sizes, and the results were very similar. Thus, we decided to include just one of them as an example in Figure~\ref{plot_tol_marg_err}.

\end{document}